\documentclass[12pt]{article}
\usepackage{times}

\usepackage{mathrsfs}
\usepackage{subfig}

\usepackage[shortlabels]{enumitem}
\usepackage[utf8]{inputenc}
\usepackage{commath}
\usepackage{amsthm}
\usepackage{wrapfig}
\usepackage{amsmath}
\usepackage{amssymb}
\usepackage{cite}
\usepackage{tikz}
\usepackage{setspace}

\usepackage{bigints}
\usepackage{comment}
\usepackage{mathtools}
\usepackage{mathrsfs}
\usepackage{fancyhdr}

\allowdisplaybreaks[4]

\numberwithin{equation}{section}
\usepackage{etoolbox}
\patchcmd{\thebibliography}
  {\settowidth}
  {\setlength{\itemsep}{0pt plus -10pt}\settowidth}
  {}{}
\apptocmd{\thebibliography}
  {
  }
  {}{}
\makeatletter
\newtheorem*{rep@theorem}{\rep@title}
\newcommand{\newreptheorem}[2]{%
\newenvironment{rep#1}[1]{%
 \def\rep@title{#2 \ref{##1}}%
 \begin{rep@theorem}}%
 {\end{rep@theorem}}}
\makeatother

\theoremstyle{theorem}

\newreptheorem{theorem}{Theorem}
\newtheorem{thm}{Theorem}[section]
\newtheorem*{thm*}{Theorem}
\theoremstyle{definition}
\newtheorem{prop}[thm]{Proposition}
\newtheorem*{prop*}{Proposition}
\newtheorem{defn}[thm]{Definition}
\newtheorem{lem}[thm]{Lemma}
\newtheorem{cor}[thm]{Corollary}
\newtheorem*{cor*}{Corollary}
\theoremstyle{remark}
\newtheorem{rem}[thm]{Remark}

\title{Spanning trees, cycle-rooted spanning forests on\\
discretizations of flat surfaces and analytic torsion.} 

\author
{Siarhei Finski
}

\date{}

\usepackage[%
    left=1in,%
    right=1in,%
    top=1.1in,%
    bottom=0.8in,%
    paperheight=11in,%
    paperwidth=8.5in%
]{geometry}

\newcommand{\imun} {\sqrt{-1}}

\newcommand{\comp}{\mathbb{C}}
\newcommand{\real}{\mathbb{R}}

\newcommand{\nat}{\mathbb{N}}
\newcommand{\integ}{\mathbb{Z}}

\newcommand{\spec}{{\rm{Spec}}}
\newcommand{\dist}{{\rm{dist}}}

\newcommand{\enmr}[1]{\text{End}{(#1)}}

\newcommand{\ccal}{\mathscr{C}}

\newcommand{\laplcomp}{\Delta}

\newcommand{\rk}[1]{{\rm{rk}} ( #1 )}
\newcommand{\tr}[1]{{\rm{Tr}} \big[ #1 \big]}

\renewcommand{\Re}{\operatorname{Re}}
\renewcommand{\Im}{\operatorname{Im}}
\newcommand{\scal}[2]{\big< #1, #2 \big>}

\newcommand{\hh}{\mathbb{H}}

\newcommand{\Addresses}{{
  \bigskip
  \footnotesize
  \noindent \textsc{Siarhei Finski, Institut Fourier - Université Grenoble Alpes, France.}\par\nopagebreak
  \noindent  \textit{E-mail }: \texttt{finski.siarhei@gmail.com}.
}}

\newenvironment{sciabstract}{}

\begin{document} 

\maketitle

\begin{sciabstract}
  \textbf{Abstract.}
  We study the asymptotic expansion of the determinant of the graph Laplacian associated to discretizations of a half-translation surface endowed with a flat unitary vector bundle. By doing so, over the discretizations, we relate the asymptotic expansion of the number of spanning trees and the sum of cycle-rooted spanning forests weighted by the monodromy of the connection of the unitary vector bundle, to the corresponding zeta-regularized determinants.
  \par As one application, by combining our result with a recent work of Kassel-Kenyon, modulo some universal topological constants, we give an explicit formula for the limit of the probability that a cycle-rooted spanning forest with non-contractible loops, sampled uniformly on discretizations approaching a given surface, induces the given lamination by its cycles.
  We also calculate an explicit value for the limit of certain topological observables on the associated loop measures.
\end{sciabstract}

\tableofcontents

\section{Introduction}\label{sect_intro}
	  In this paper, over the discretizations of a given surface, we study the asymptotic behaviour of the number of spanning trees and the sum of cycle-rooted spanning forests, weighted by the monodromy of the connection of a unitary vector bundle, as the mesh of discretization goes to $0$.
	 \par 
	More precisely, by a spanning tree in a graph we mean a subtree covering all the vertices.
	By a cycle-rooted spanning forest (CRSF in what follows) on a graph $G = (V(G), E(G))$ we mean a subset $S \subset E(G)$, spanning all vertices and with the property that each connected component of $S$ has as many vertices as edges (in particular, it has a unique cycle).
	\begin{figure}[!htbp]%
    \centering
    \subfloat[A spanning tree]{{\includegraphics[width=0.3\textwidth]{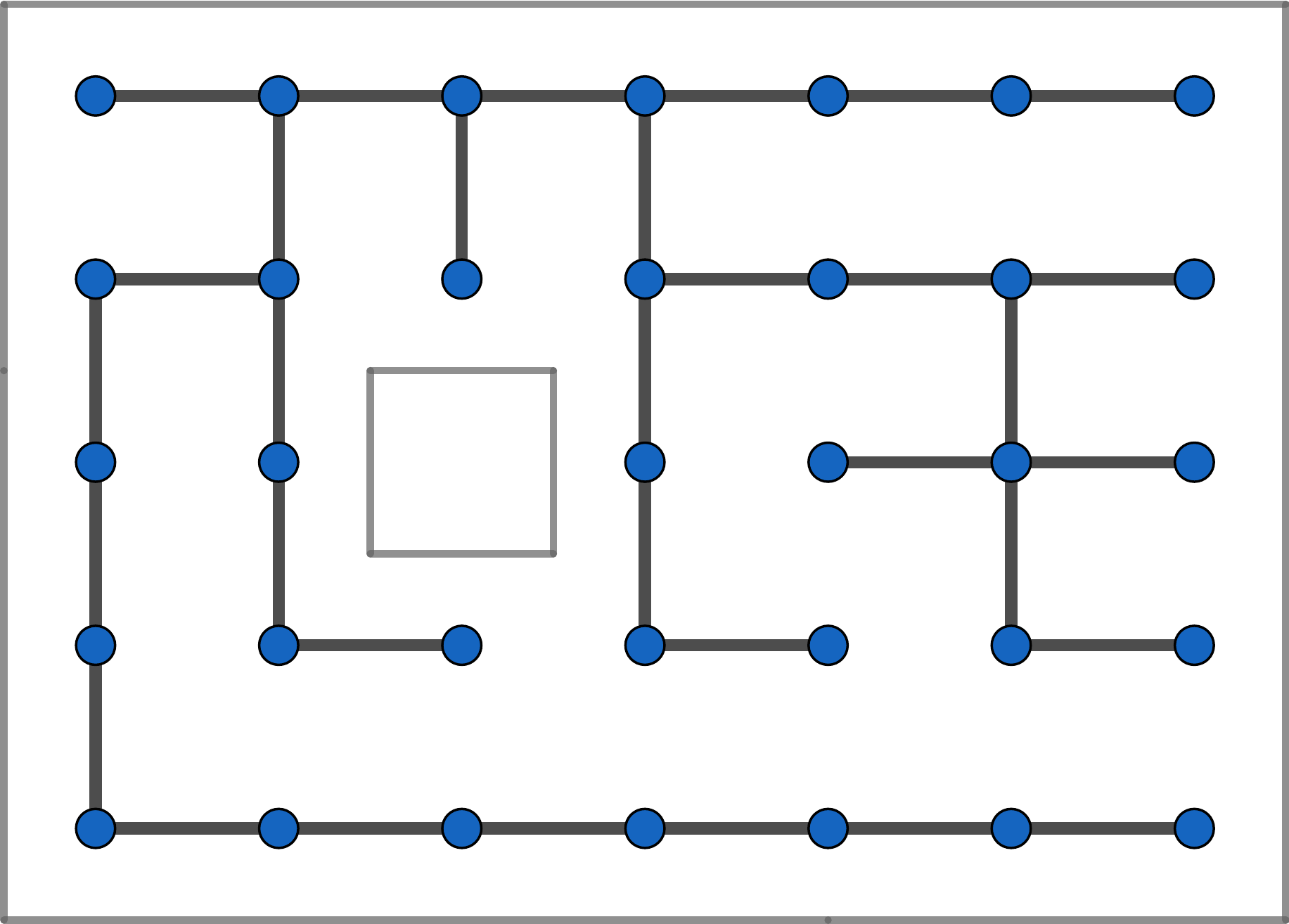}}}%
    \qquad \qquad \qquad \qquad
    \subfloat[A CRSF]{{\includegraphics[width=0.3\textwidth]{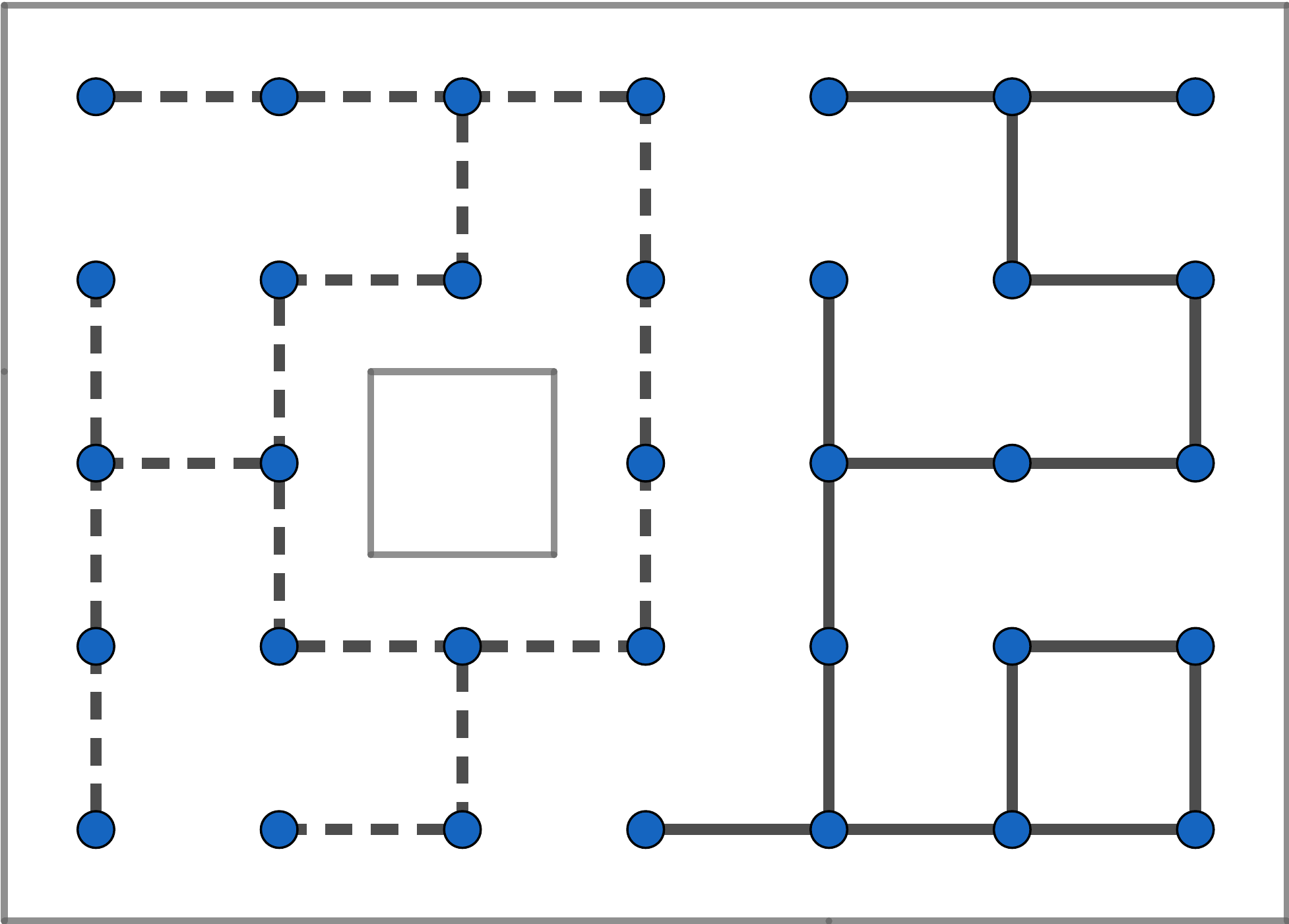}}}%
    \caption{A spanning tree and a CRSF on a graph approximating an annulus. The edges of the graph are not drawn, but they connect the nearest neighbors. The CRSF has two components. The dotted component is non-contractible in the annulus and the non-dotted one is contractible.}%
    \label{fig_crsf}%
	\end{figure}
	\par 
	The number of spanning trees on a finite graph $G$ is often called the \textit{complexity} of the graph, denoted here by $N_{G, 1}^{0}$. 
	Kirchhoff in his famous \textit{matrix-tree theorem} showed that for a finite graph $G$, the product of nonzero eigenvalues of the combinatorial Laplacian $\Delta_G$, defined as a difference of degree and adjacency operators, is related to the complexity of the graph as follows
	\begin{equation}\label{eq_matrixtree}
		N_{G, 1}^{0} = \frac{1}{\# V(G)} \det{}' \Delta_G := \frac{1}{\# V(G)} \prod_{\lambda \in \spec(\Delta_G ) \setminus \{0\}} \lambda.
	\end{equation}
	\par Forman in \cite[Theorem 1]{Forman} extended this result to the setting of a line bundle with a unitary connection on a graph (in Section \ref{sect_sq_t_s_disc} we give a precise meaning for those notions).
	Kenyon in \cite[Theorems 8,9]{KenyonFlatVecBun} generalized it to vector bundles of rank $2$ endowed with ${\rm{SL}}_2(\comp)$ connections. 
	\par 
	To state Kenyon's result, recall that for a finite graph $G$ and a vector bundle $V$ over $G$, endowed with a connection $\nabla^V$, the \textit{twisted Laplacian} is an operator, acting on $f \in {\rm{Map}}(V(G), F)$ by
	\begin{equation}\label{eq_comb_lapl_tw}
		\laplcomp_G^{V} f(v) = \sum_{(v, v') \in E(G)} \big( f(v) - \phi_{v' v} f(v') \big), \qquad v \in V(G),
	\end{equation}	
	where $\phi_{v' v}$ is the parallel transport from $v'$ to $v$ associated to $\nabla^V$. In case if $V$ is a trivial line bundle and $\nabla^V$ is the trivial connection, we recover the combinatorial Laplacian $\laplcomp_G$ on graph $G$.
	\par 
	Now, to simplify, for a Hermitian vector bundle $(F, h^F)$ of rank $2$ with a unitary connection $\nabla^F$, Kenyon in \cite[Theorems 8, 9]{KenyonFlatVecBun}, cf. also Kassel-Kenyon \cite[Theorem 15]{KassKenyon}, proved 
	\begin{equation}\label{eq_crsf_1}
		\sqrt{\smash[b]{\det{}' \laplcomp_{G}^{F}}} = \sum_{T \in {\rm{CRSF}}(G)} \prod_{\gamma \in {\rm{cycles}}(T)} \Big( 2 - {\rm{Tr}}(w_{\gamma}) \Big),
	\end{equation}
	where ${\rm{CRSF}}(G)$ is the set of all CRSF's on $G$, the product is over all cycles $\gamma$ of the CRSF and ${\rm{Tr}}(w_{\gamma})$ is the trace of the monodromy of $\nabla^F$, evaluated along $\gamma$.
	\par 
	In this article instead of considering a single graph $G$, we consider a family $\Psi_n$, $n \in \nat^*$ of graphs, constructed as approximations of a surface $\Psi$. Our goal is to understand the asymptotics of (\ref{eq_matrixtree}) and (\ref{eq_crsf_1}), as $n \to \infty$, and to see how the geometry of $\Psi$ is reflected in it.
	From (\ref{eq_matrixtree}) and (\ref{eq_crsf_1}), it is equivalent to study the associated determinants on the approximation graphs $\Psi_n$.
	\par 
	Such families of graphs arise naturally in many contexts of mathematics, when one studies some continuous quantity or process by using scaling limits of grid-based approximations.
	\par For rectangles $\Psi$, this problem has received considerable attention in the past.
	The principal term of the asymptotics of (\ref{eq_matrixtree}) has been evaluated by Kastelyn in \cite{KastelynEntropy}, and it is related to the area of the rectangle.
	Then Fischer in \cite{Fisch61} calculated the next term and related it to the perimeter of the rectangle.
	Further, Ferdinand in \cite{Ferdin1967} obtained the asymptotic expansion up to the constant term and Duplantier-David in \cite{DuplDav} related the constant term to the zeta-regularized determinant of the Laplacian of the rectangle (to be defined later).
	\par 
	This article can be regarded as an extension of the result of Duplantier-David from rectangles to any surface tillable by equal squares, endowed with a vector bundle. Our study leads, in particular, to a complete and partial answers to Open problems 2 and 4 respectively from Kenyon \cite[\S 8]{KenyonDet}.
	\par 
	As one application of our result, we give an alternative interpretation of the conformal invariance of some functionals, which appeared in the recent work of Kassel-Kenyon \cite{KassKenyon}.
	As another application, by relying on the work \cite{KassKenyon}, as the mesh of discretization tends to $0$, we calculate the limit of certain topological observables over the loops induced by the uniform measure on the set of non-contractible CRSF's.
	We also precise the analytic terms in a formula from \cite{KassKenyon} for the limit of the probability that a CRSF with non-contractible loops, sampled uniformly on discretizations approaching a given surface, induces the given lamination by its cycles.
	\par 
	Let's describe our results more precisely.
	We fix a \textit{half-translation} surface $(\Psi, g^{T \Psi})$ with piecewise geodesic boundary.
	By a \textit{half-translation} surface $(\Psi, g^{T \Psi})$, we mean a surface endowed with a flat metric $g^{T \Psi}$ which has conical singularities of angles $k \pi$, $k \in \nat^* \setminus \{2\}$. The name half-translation comes from the fact that it has a ramified double cover, which is a translation surface, i.e. it can be obtained by identifying parallel sides of some polygon in $\comp$, see Section \ref{sect_sq_t_s_disc}.
	\par
	We denote by ${\rm{Con}}(\Psi)$ the conical points of the surface $\Psi$, and by ${\rm{Ang}}(\Psi)$ the points where two different smooth components of the boundary meet (corners). 
	Let $\angle : {\rm{Con}}(\Psi) \to \real$ be the function which associates to a conical point its angle and let $\angle : {\rm{Ang}}(\Psi) \to \real$ be the function which associates the interior angle between the smooth components of the boundary.
	Denote by ${\rm{Ang}}^{\neq \pi/2}(\Psi)$ (resp. ${\rm{Ang}}^{= \pi/2}(\Psi)$) the subset of ${\rm{Ang}}(\Psi)$ corresponding to angles $\neq \frac{\pi}{2}$ (resp. $= \frac{\pi}{2}$).
	\par 
	For example, if $\Psi$ is a rectangular planar domain, then there are no conical angles and the angles between smooth components of the boundary are either equal to $\frac{\pi}{2}$, $\frac{3\pi}{2}$ or $2\pi$. In this case, the angles $\frac{\pi}{2}$ are called \textit{convex}, the angles $\frac{3\pi}{2}$ are called \textit{concave} and the arcs of the boundary meeting in the angle $2 \pi$ are called \textit{slits}. We note that if we have a slit, the corresponding arc in $\partial \Psi$ will be listed several times, taking into account the “direction" in which $\partial \Psi$ goes.
	See also Figures \ref{fig_L_shape}, \ref{fig_slit}, \ref{fig_cpi_graph}, \ref{fig_c4pi}, \ref{fig_c3pi} and the discussion in Section \ref{sect_decomp} for less trivial examples.
	\par 
	We fix a flat unitary vector bundle $(F, h^F, \nabla^F)$ on the compactification 
	\begin{equation}
		\overline{\Psi} := \Psi \cup {\rm{Con}}(\Psi).	
	\end{equation}
	By flat we mean that the monodromies over the contractible loops in $\overline{\Psi}$ vanish, or equivalently $(\nabla^F)^2 = 0$.
	By a unitary connection we mean a connection $\nabla^F$ preserving the metric $h^F$.
	\par 
	We suppose that our surface $\Psi$ can be tiled completely and without overlaps over subsets of positive Lebesgue measure by \textit{euclidean squares} of the same size and area $1$. In particular, the boundary $\partial \Psi$ get's tiled by the boundaries of the tiles and the angles between smooth components of the boundary are of the form $\frac{k \pi}{2}$, $k \in \nat^* \setminus \{2\}$.
	Such surfaces are also called \textit{pillowcase covers}, and in case if there is no boundary, they can be characterized as certain ramified coverings of $\mathbb{CP}^1$.
	\par 
	For example, if $\Psi$ is a torus, then it can be tiled by euclidean squares of the same size if and only if the ratio of its periods is rational.
	If $\Psi$ is a rectangular domain, this happens only if the ratios of the pairwise distances between neighboring corners on the boundary of $\Psi$ are rational. 
	\par We fix a tiling of $\Psi$. 
	We construct a graph $\Psi_1 = (V(\Psi_1), E(\Psi_1))$ by taking vertices $V(\Psi_1)$ as the centers of tiles and edges $E(\Psi_1)$ in such a way that the resulting graph $\Psi_1$ is the \textit{nearest-neighbor graph} with respect to the flat metric on $\Psi$.
	This means that an edge connects two vertices if and only if they are the closest neighbors with respect to the metric $g^{T \Psi}$.
	\par 
	The vector bundle $F_1$ over $\Psi_1$ and the Hermitian metric $h^{F_1}$ on $F_1$ are constructed by the restriction from $F$ and $h^F$. 
	The connection $\nabla^{F_1}$ is constructed as the parallel transport of $\nabla^F$ with respect to the straight path between the vertices.
	It is a matter of a simple verification to see that the fact that $(F, h^F, \nabla^F)$ is unitary implies that $(F_1, h^{F_1}, \nabla^{F_1})$ is unitary as well.
	\par 
	By considering regular subdivisions of tiles into $n^2$ squares, $n \in \nat^*$, and repeating the same procedure, we construct a family of graphs $\Psi_n = (V(\Psi_n), E(\Psi_n))$ with unitary vector bundles $(F_n, h^{F_n}, \nabla^{F_n})$ over $\Psi_n$, for $n \in \nat^*$. 
	Note that we have a natural injection
	\begin{equation}\label{eq_inj_vert}
		V(\Psi_n) \hookrightarrow \Psi.
	\end{equation}
	\par For example, in case if $\Psi$ is a rectangular domain in $\comp$ with integer vertices, the family of graphs $\Psi_n$ coincides with the subgraphs of $\frac{1 + \imun}{2n} + \frac{1}{n} \integ^2$, which stays inside of $\Psi$. See Figure \ref{fig_L_shape}.
	\begin{figure}[!htbp]%
    \centering
    \subfloat{{\includegraphics[width=0.2\textwidth]{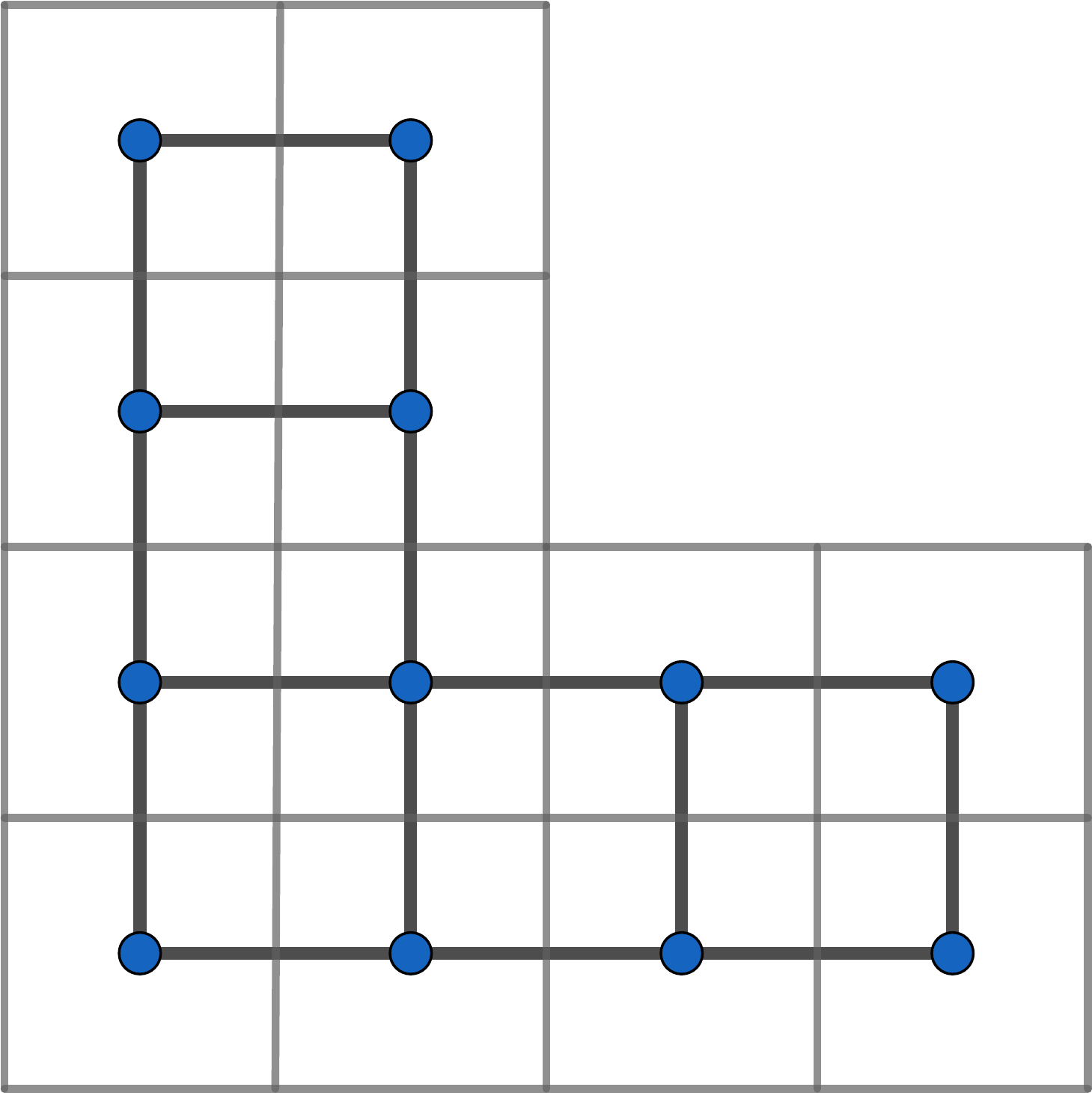}}}%
    \qquad \qquad \qquad \qquad
    \subfloat{{\includegraphics[width=0.2\textwidth]{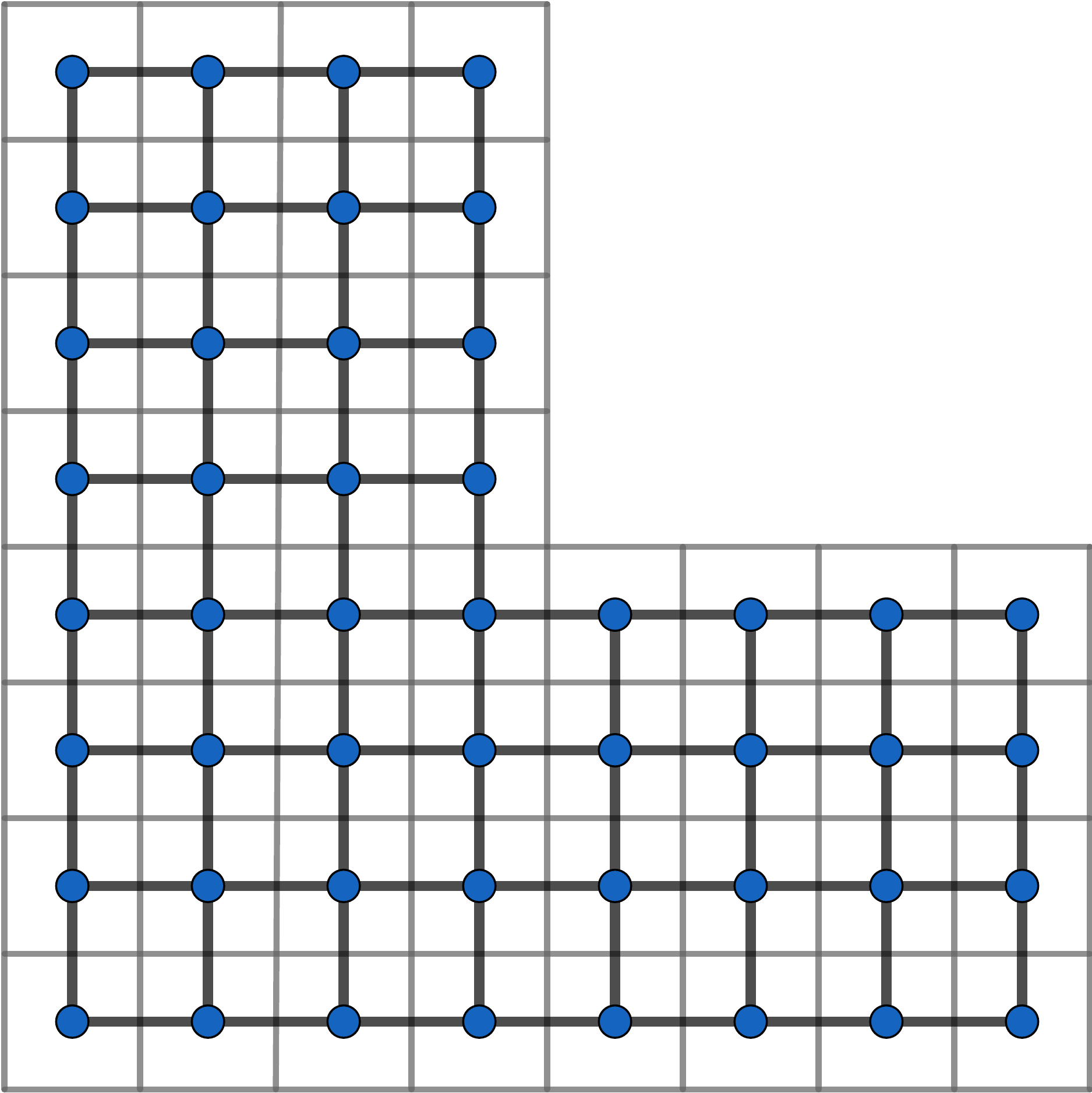}}}%
    \caption{An $L-$shape and its discretizations. It is a rectangular domain in $\comp$ with a single corner with angle $\frac{3 \pi}{2}$ and five corners  with angles $\frac{\pi}{2}$; in the notation of Section \ref{sect_decomp}, it corresponds to $\mathcal{A}_{\frac{3 \pi}{2}}$.}%
    \label{fig_L_shape}%
	\end{figure}
	\par 
	We denote by $(\nabla^F)^*$ the formal adjoint of $\nabla^F$ with respect to the $L^2$-metric induced by $g^{T \Psi}$ and $h^F$.
	We denote by $\laplcomp_{\Psi}^{F}$ the scalar Laplacian on $(\Psi, g^{T\Psi})$, associated with $(F, h^F, \nabla^F)$. 
	In other words, it is a differential operator acting on sections of $F$ by
	\begin{equation}\label{eq_lapl}
		\laplcomp_{\Psi}^{F} := (\nabla^F)^* \nabla^F.
	\end{equation}
	If $(F, h^{F}, \nabla^{F})$ is trivial, and $\Psi$ is a rectangular domain, it is easy to see that $\laplcomp_{\Psi}^{F}$ coincides with the usual Laplacian, given by the formula $- \frac{\partial^2}{\partial x^2} - \frac{\partial^2}{\partial y^2}$.
	\par 
	In this paper we always consider $\laplcomp_{\Psi}^{F}$ with \textit{von Neumann boundary conditions} on $\partial \Psi$. In other words, the sections $f$ from the domain of our Laplacian satisfy
	\begin{equation}\label{eq_bound_vn}
		\nabla^{F}_{n} f = 0 \quad \text{over } \partial \Psi. 
	\end{equation}
	\par It is well-known that because of conical singularities and non-smoothness of the boundary, the Laplacian $\laplcomp_{\Psi}^{F}$ is not necessarily \textit{essentially self-adjoint} even after precising the boundary condition (\ref{eq_bound_vn}). 	
	Thus, to define the spectrum of $\laplcomp_{\Psi}^{F}$, we are obliged to specify the self-adjoint extension of $\laplcomp_{\Psi}^{F}$ we are working with.
	We choose the Friedrichs extension and, by abuse of notation, we denote it by the same symbol $\laplcomp_{\Psi}^{F}$. See Section \ref{sect_zeta_flat_fun} and \cite{FinFinDiffer} for more on Friedrichs extension and the motivation for this choice of a self-adjoint extension.
	\par 
	Similarly to smooth domains, the spectrum of $\laplcomp_{\Psi}^{F}$ is discrete (cf. Proposition \ref{prop_spec_discr}), in other words, we may write (by convention, the sets in this article take into account the multiplicity)
	\begin{equation}\label{eq_spec_lapl_defn}
		\spec(\laplcomp_{\Psi}^{F}) = \{ \lambda_1, \lambda_2, \cdots \},
	\end{equation}
	where $\lambda_i \geq 0$, $i \in \nat^*$ form a non-decreasing sequence.
	\par 
	 The \textit{zeta-regularized determinant} of $\laplcomp_{\Psi}^{F}$ (also called the \textit{analytic torsion}) is defined non-rigorously by the following \textit{non-convergent} infinite product
	 \begin{equation}\label{defn_an_tors}
	 	\det{}' \laplcomp_{\Psi}^{F} := \prod_{\lambda \in \spec(\laplcomp_{\Psi}^{F}) \setminus \{ 0 \}} \lambda.
	 \end{equation}
	\par More formally, one can consider the associated zeta-function $\zeta_{\Psi}^{F}(s)$, defined for $s \in \comp$, $\Re(s) > 1$, by the formula (cf. Corollary \ref{cor_as_linear_growth})
	\begin{equation}\label{defn_zeta}
		\zeta_{\Psi}^{F}(s) = \sum_{\lambda \in \spec(\laplcomp_{\Psi}^{F}) \setminus \{0\}} \frac{1}{\lambda^s}.
	\end{equation}
	Similarly to the case of smooth manifolds, $\zeta_{\Psi}^{F}$ extends meromorphically to $\comp$ and $0$ is a holomorphic point of this extension (cf. end of Section \ref{sect_zeta_flat_fun}). Thus, the following definition makes sense
	\begin{defn}[{Ray-Singer \cite{RSReal}, \cite{Ray73}}]\label{defn_an_tors_zeta}
	The zeta-regularized determinant $\det{}' \laplcomp_{\Psi}^{F}$ is defined by
		\begin{equation}
			\det{}' \laplcomp_{\Psi}^{F} := \exp \big(- (\zeta_{\Psi}^{F})'(0) \big).
		\end{equation}
	\end{defn}
	\begin{sloppypar}	
	Clearly, by (\ref{defn_zeta}), Definition \ref{defn_an_tors_zeta} corresponds formally to (\ref{defn_an_tors}).
	\par 
	The value $\zeta_{\Psi}^{F}(0)$ is also interesting and it can be evaluated as follows (cf. Proposition \ref{prop_zeta_zero_val})
	\begin{equation}\label{eq_value_zeta_0}
	 	\zeta_{\Psi}^{F}(0) = -\dim H^0(\Psi, F) + \frac{\rk{F}}{12}
		\Big(	 	
	 	 \sum_{P \in {\rm{Con}}(\Psi)} \frac{4 \pi^2 - \angle(P)^2}{2 \pi \angle(P)}
		+
	 	\sum_{Q \in {\rm{Ang}}(\Psi)} \frac{\pi^2 - \angle(Q)^2}{2 \pi \angle(Q)}
	 	\Big),
	\end{equation}
	where $H^0(\Psi, F)$ is the vector space of flat sections.
	\par 
	Recall that the Catalan's constant $G \in \real$ is defined by
	\begin{equation}\label{eq_defn_cat}
		G = 1 - \frac{1}{3^2} + \frac{1}{5^2} - \frac{1}{7^2} + \cdots.
	\end{equation}
	We define the renormalized logarithm $\log^{{\rm{ren}}} (\det{}'\laplcomp_{\Psi_n}^{F_n})$, $n \in \nat^*$ by
	\begin{multline}\label{eq_ren_number_sp_tree}
		\log^{{\rm{ren}}} (\det{}'\laplcomp_{\Psi_n}^{F_n})
		:=
		\log (\det{}'\laplcomp_{\Psi_n}^{F_n})  - \frac{4 G}{\pi} \cdot \rk{F} \cdot A(\Psi) \cdot n^2 
		\\
		- \frac{\log(\sqrt{2} - 1)}{2} \cdot \rk{F} \cdot |\partial \Psi| \cdot n 
		+ 2 \zeta_{\Psi}^{F}(0) \cdot \log (n).
	\end{multline}
	Our main result shows that up to some universal contribution, depending only on the sets $\angle({\rm{Con}}(\Psi))$ and $\angle({\rm{Ang}}^{\neq \pi/2}(\Psi))$, the renormalized logarithm $\log^{{\rm{ren}}} (\det{}'\laplcomp_{\Psi_n}^{F_n})$ converges to $\log( \det{}' \laplcomp_{\Psi}^{F})$ modulo some explicit constant.
	More precisely, in Section \ref{sect_idea_proof}, we prove
	\end{sloppypar}
	\begin{thm}\label{thm_asympt_exansion}
		Consider a \textit{half-translation} surface $(\Psi, g^{T \Psi})$ with piecewise geodesic boundary. Suppose that the surface $\Psi$ can be tiled by euclidean squares of area $1$ as explained above.  Construct a family of graphs $\Psi_n$, $n \in \nat^*$ as above.
		We endow $\overline{\Psi}$ with a flat unitary vector bundle $(F, h^F, \nabla^F)$ and $\Psi_n$ with the induced unitary vector bundles $(F_n, h^{F_n}, \nabla^{F_n})$. 
		\par 
		Then, for any $k, K \in \nat$, the following asymptotic bound holds
		\begin{equation}\label{eq_log_n_ren_logn}
			\log^{{\rm{ren}}} (\det{}'\laplcomp_{\Psi_n}^{F_n}) = o(\log(n)), \qquad \text{for $n = k^l \cdot K$, $l \in \nat$}. 
		\end{equation}
		\par Moreover, there is a sequence $CA_n$, $n \in \nat^*$, which depends only on the set of conical angles $\angle({\rm{Con}}(\Psi))$ and the set of angles $\angle({\rm{Ang}}^{\neq \pi/2}(\Psi))$, such that, as $n \to \infty$, we have
		\begin{equation}\label{eq_asympt_trees_exp}
			\log^{{\rm{ren}}} (\det{}'\laplcomp_{\Psi_n}^{F_n}) - \rk{F} \cdot CA_n \to \log \det{}' \laplcomp_{\Psi}^{F} - \frac{\log(2) \cdot \rk{F}}{16} \cdot \# {\rm{Ang}}^{= \pi/2}(\Psi).
		\end{equation}
		\par Also, the sequence $CA_n$ depends additively on the sets $\angle({\rm{Con}}(\Psi))$, $\angle({\rm{Ang}}^{\neq \pi/2}(\Psi))$, where by the addition of sets we mean their union (which takes into account the multiplicity).
	\end{thm}
	\begin{rem}\label{rem_main_thm}
		a) For rectangles, Theorem \ref{thm_asympt_exansion} has been proved by Duplantier-David by using explicit expression for the spectrum in \cite[(4.7) and (4.23)]{DuplDav}, cf. also Appendix \ref{sect_kron}.
		\par
		b) The sequence $CA_n$, $n \in \nat^*$ can be expressed as a sum of certain constants which depend additively on the sets $\angle({\rm{Con}}(\Psi))$, $\angle({\rm{Ang}}^{\neq \pi/2}(\Psi))$ and a sum of renormalized logarithms $\log^{{\rm{ren}}} (\det{}' \laplcomp_{\Phi_n})$, corresponding to a number of model spaces $\Phi$.
		For example, if we restrict our attention to rectangular domains in $\comp$ without slits, then there is only one model space $\Phi$, which is the $L$-shape depicted in Figure \ref{fig_L_shape}.
		If we allow slits in $\Psi$, this amounts to adding one more model space, which is the model slit, depicted in Figure \ref{fig_slit}.
		See Section \ref{sect_idea_proof} and (\ref{eq_ca_n_full_formula}) for more on this.
		\par 
		c) We believe that for a certain constant $c$, depending on $(\Psi, g^{T \Psi})$ and $(F, h^F, \nabla^F)$, one can improve the bound (\ref{eq_log_n_ren_logn}) to $c + o(1)$, as $n \to \infty$. By Remark \ref{rem_main_thm}b) and (\ref{eq_asympt_trees_exp}), to prove such a statement, it is enough to do so only for a number of model cases.
	\end{rem}
	\begin{figure}[h]
		\includegraphics[width=0.15\textwidth]{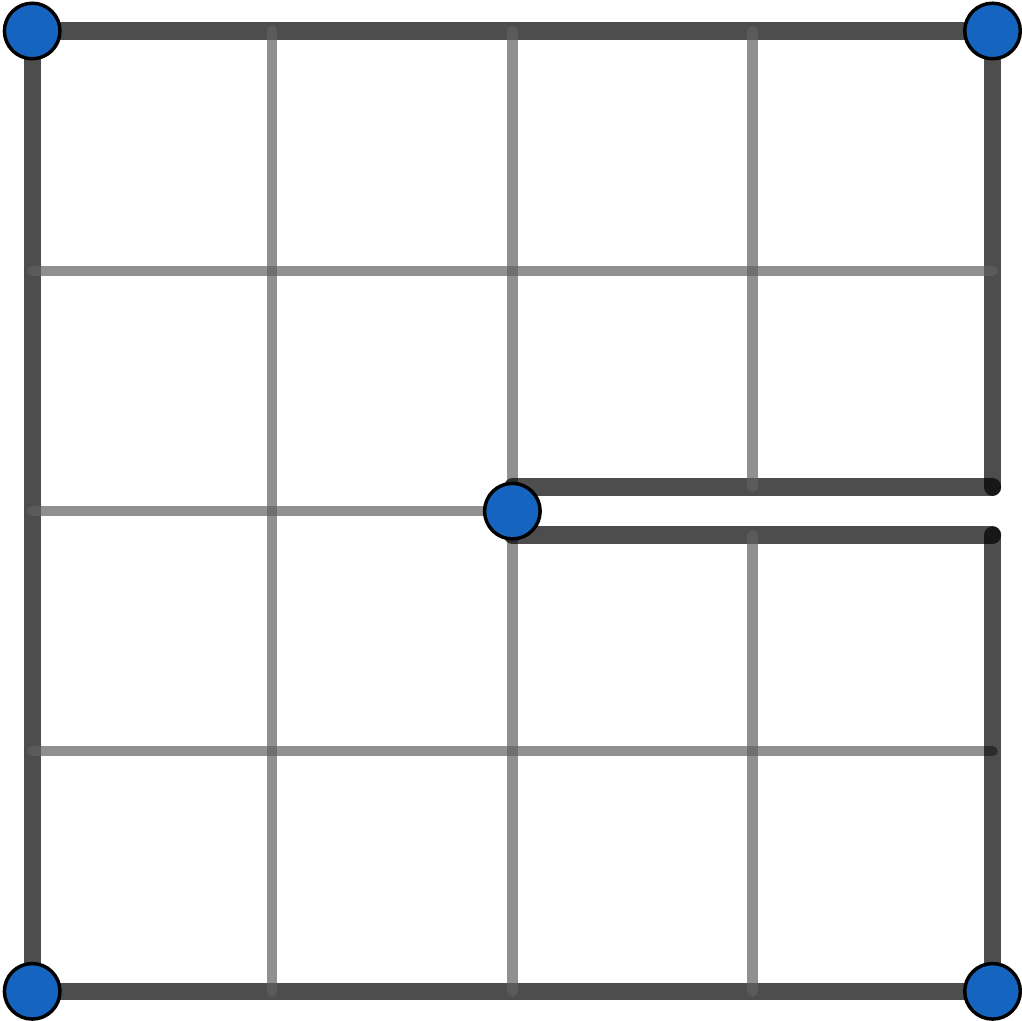}	
		\centering
		\caption{The model slit. In the notation of Section \ref{sect_decomp}, it corresponds to $\mathcal{A}_{2 \pi}$; it has one corner with angle $2\pi$ and $6$ corners with angles $\frac{\pi}{2}$. It is tiled with 16 squares and its perimeter is 20.}
		\label{fig_slit}
	\end{figure}	
	Directly from Theorem \ref{thm_asympt_exansion}, and the fact that for $\Psi$, considered in the following theorem, the sets ${\rm{Con}}(\Psi), {\rm{Ang}}(\Psi)$ are empty, we get
	\begin{cor}\label{cor_tor_rect_conv}
		Let $(\Psi, g^{T \Psi})$ be a torus with integer periods or a cylinder with integer height and circumference.
		We endow $\Psi$ with a flat unitary vector bundle $(F, h^F, \nabla^F)$. Define the graphs $\Psi_n$, $n \in \nat^*$, with induced unitary vector bundles $(F_n, h^{F_n}, \nabla^{F_n})$, as in Theorem \ref{thm_asympt_exansion}. 
		As $n \to \infty$:
		\begin{equation}\label{eq_asympt_trees_exp_tori}
			\log^{{\rm{ren}}} (\det{}'\laplcomp_{\Psi_n}^{F_n}) \to \log \det{}' \laplcomp_{\Psi}^{F}.
		\end{equation}
	\end{cor}
	\begin{rem}
		a) When $\Psi$ is a tori with perpendicular axes and $(F, h^F, \nabla^F)$ is trivial, this has been obtained by Duplantier-David in \cite[(3.18) and (3.41)]{DuplDav}.
		See also Chinta-Jorgenson-Anders \cite{JorgChintaKarl}, \cite{Chinta_jorg_karl_comp} and Vertman \cite{VertmanTori} for related results for tori in any dimension endowed with a trivial vector bundle.
		For tori of any dimension with perpendicular axes, endowed with a line bundle, this has been proved by Friedli \cite{FriedliTorus}.
		For cylinders with trivial vector bundles, Brankov-Priezzhev in \cite[(46)]{BrankPriezzhev} gave an explicit expression for the left-hand side of (\ref{eq_asympt_trees_exp_tori}) without relating it to the analytic torsion.
		\par b) Quite easily, our proof can be generalized to non-orientable surfaces by working with invariant subspaces with respect to the Deck transformations on the orientable double cover. This would entail (\ref{eq_asympt_trees_exp_tori}) for Möbius bands and Klein bottles with integer periods. For trivial vector bundles, Brankov-Priezzhev in \cite[(48)]{BrankPriezzhev} and Izmailian-Oganesyan-Hu in \cite{IOH} gave an explicit expression for the left-hand side of (\ref{eq_asympt_trees_exp_tori}) for Möbius bands and Klein bottles respectively  without identifying the result as the logarithm of the analytic torsion.
	\end{rem}
	Directly from Theorem \ref{thm_asympt_exansion}, we obtain
	\begin{thm}\label{thm_rel_as_compexity}
		Let $(\Psi, g^{T \Psi})$  and $(\Phi, g^{T \Phi})$  be two \textit{half-translation} surfaces satisfying the assumptions from Theorem \ref{thm_asympt_exansion}.
		Construct the graphs $\Psi_n$ and $\Phi_n$, $n \in \nat^*$ as in Theorem \ref{thm_asympt_exansion}.
		Endow $\Psi$ (resp. $\Phi$) with a flat unitary vector bundle $(F, h^F, \nabla^F)$ (resp. $(V, h^V, \nabla^V)$) and induce unitary vector bundles $(F_n, h^{F_n}, \nabla^{F_n})$ (resp. $(V_n, h^{V_n}, \nabla^{V_n})$) on $\Psi_n$ (resp. $\Phi_n$). 
		We suppose $\rk{F} = \rk{V}$.
		\par 
		Suppose, moreover, that the areas $A(\Psi)$ of $\Psi$ and $A(\Phi)$ of $\Phi$, the perimeters $|\partial \Phi|$ and $|\partial \Psi|$, the sets of the conical angles $\angle({\rm{Con}}(\Psi))$ of $\Psi$ and  $\angle({\rm{Con}}(\Phi))$ of $\Phi$ and the set of interior angles between the consecutive arcs  $\angle({\rm{Ang}}(\Psi))$ of $\partial \Psi$ and  $\angle({\rm{Ang}}(\Phi))$ of $\partial \Phi$ coincide respectively. Suppose, finally, that  the dimensions of $H^0(\Psi, F)$ and $H^0(\Phi, V)$ coincide.
		Then
		\begin{equation}\label{eq_limit_compar}
			\lim_{n \to \infty} \frac{\det{}'\laplcomp_{\Psi_n}^{F_n}}{\det{}' \laplcomp_{\Phi_n}^{V_n}} = \frac{\det{}' \laplcomp_{\Psi}^{F}}{\det{}' \laplcomp_{\Phi}^{V}}.
		\end{equation}
		In particular, if we apply (\ref{eq_limit_compar}) for $(F, h^F, \nabla^F)$, $(V, h^V, \nabla^V)$ trivial, by (\ref{eq_matrixtree}), we obtain
		\begin{equation}\label{eq_limit_compar2}
			\lim_{n \to \infty} \frac{N_{\Psi_n, 0}}{N_{\Phi_n, 0}} = \frac{\det{}' \laplcomp_{\Psi}}{\det{}' \laplcomp_{\Phi}}.
		\end{equation}
	\end{thm}
	\begin{rem}\label{rem_as_compl_compar}
		a) As it follows from Theorem \ref{thm_asympt_exansion}, the conditions $A(\Psi) = A(\Phi)$,  $|\partial \Phi| = |\partial \Psi|$ and $\zeta_{\Psi}^{F}(0) = \zeta_{\Phi}^{V}(0)$ (see (\ref{eq_value_zeta_0})) are necessary if one wants the limit (\ref{eq_limit_compar}) to be finite and non-zero.
		\par 
		b) For $\Psi = \Phi$ tori, {Dub\'edat}-{Gheissari} \cite[Proposition 4]{DubGheiss} established (\ref{eq_limit_compar}) for non-trivial unitary line bundles $(F, h^F, \nabla^F)$, $(V, h^V, \nabla^V)$. In their setting, they have also shown in \cite[Proposition 3]{DubGheiss} that under some mild assumptions on the approximation $\Psi_n$, $n \in \nat^*$ - that the random walk on $\Psi_n$ converges to the Brownian motion on $\Psi$, as $n \to \infty$, - the limit (\ref{eq_limit_compar}) does not depend on the choice of the approximation $\Psi_n$.
		According to \cite[end of \S 1]{DubGheiss}, the universality in this sense continues to hold for general surfaces (not only for tori). 
		This means that (\ref{eq_limit_compar}) holds (for $\Psi = \Phi$, $H^0(\Psi, F) = 0$ and $H^0(\Psi, V) = 0$) for any discretizations $\Psi_n$ satisfying the above assumption.
		\par 
		c) For simply-connected rectangular domains in $\comp$ (and, consequently, trivial $(F, h^F, \nabla^F)$), the fact that the limit (\ref{eq_limit_compar}) exists and finite was proved by Kenyon in \cite[Corollary 2 and Remark 4]{KenyonDet} (cf. also Theorem \ref{thm_kenyon_thm}). No relation between this limit and the analytic torsion was given in \cite{KenyonDet}.
		\par 
		d) For $\Psi = \Phi$, and $F$, $V$ of rank $2$, the fact that the limit (\ref{eq_limit_compar}) exists was proved by Kassel-Kenyon in \cite[Theorem 17]{KassKenyon}, see \cite[\S 4.1]{KassKenyon} (cf. also Theorem \ref{thm_kas_kenyon_conv}). 
		Their proof, is very different from ours, and they don't argue by the analogue of Theorem \ref{thm_asympt_exansion}. 
		Also, no relation between the value of this limit and the analytic torsion was given in \cite{KassKenyon}.
	\end{rem}
	As one application of Theorem \ref{thm_rel_as_compexity}, we see that (\ref{eq_limit_compar2}) makes a connection between the maximization of asymptotic complexity and maximization of the analytic torsion.
	In realms of polygonal domains in $\comp$, the last problem has been considered by Aldana-Rowlett \cite[Conjecture 1 and Theorem 5]{AldRowl}.
	For maximization of the analytic torsion for flat tori in any dimension, see the article of Sarnak-Strömbergsson \cite[Theorems 1, 2]{SarnakStrombergsson}.
	See also Osgood-Philips-Sarnak \cite{OPS1} for a similar maximization problem inside the class of all smooth metrics.
	\par 
	Another interesting consequence of (\ref{eq_limit_compar}) is the conformal invariance of the left-hand side of (\ref{eq_limit_compar}) for $\Psi = \Phi$ and $(F, \nabla^F)$, $(V, \nabla^V)$, satisfying $H^0(\Psi, F) = H^0(\Phi, V) = 0$.
	In fact, it follows from a simple application of the anomaly formula, see Bismut-Gillet-Soulé \cite[Theorem 1.23]{BGS3}, as the $L^2$-metric on $H^1(\Psi, F)$, which appears as one of the terms in the anomaly formula, is conformally invariant.
	Remark that in \cite{BGS3} authors consider smooth manifolds, and thus their results cannot be directly applied for manifolds with conical singularities and corners. However, considerable advances has been done to extend the anomaly formula for manifolds in this case, see Aurell-Salomonson \cite{AurellSalomon}, Kokotov-Korotkin \cite{KokKorDetFried}, Aldana-Rowlett \cite{AldRowl}, Kalvin \cite{KalvinPol}.
	\par 
	The conformal invariance of the left-hand side of (\ref{eq_limit_compar}) in a setting of Remark \ref{rem_as_compl_compar}d) has already been proved by Kassel-Kenyon in \cite[Theorem 17]{KassKenyon} by different methods. They used this result to construct the conformally invariant measure supported on non-contractible multicurves on a non-simply connected domain as a limit of corresponding measures on discretizations.
	\par 
	Let's recall some notions from  \cite{KassKenyon} and by doing so we will also state another applications of Theorem \ref{thm_rel_as_compexity}.
	We fix $(\Psi, g^{T \Psi})$, $\Psi_n$, $(F, h^F, \nabla^F)$, $(F_n, h^{F_n}, \nabla^{F_n})$ as in Theorem \ref{thm_asympt_exansion}.
	We denote by $\rm{CRSF}_{{\rm{nonc}}}(\Psi_n)$ the subset of $\rm{CRSF}(\Psi_n)$ containing CRSF's with non-contractible loops in $\overline{\Psi}$.
	We endow $\rm{CRSF}_{{\rm{nonc}}}(\Psi_n)$ with the uniform measure, which we denote by $\mu_{{\rm{nonc}}}$.
	\begin{thm}\label{thm_exp_value}
		There is a constant $Z(\Psi) > 0$ such that for any flat unitary vector bundle $(F, h^F, \nabla^F)$ of rank $2$ over $\overline{\Psi}$, satisfying $H^0(\Psi, F) = 0$, we have
		\begin{equation}\label{eq_exp_value}
			\lim_{n \to \infty} \mathbb{E}_{\rm{CRSF}_{{\rm{nonc}}}}^{\Psi_n} 
			\Big[
				\prod_{\gamma \in {\rm{CRSF}}} 
				\Big( 2 - {\rm{Tr}}(w_{\gamma}) \Big)
			\Big]
			=
			\frac{\sqrt{\smash[b]{\det{}' \laplcomp_{\Psi}^{F}}}}{Z(\Psi)},
		\end{equation}
		where the expected value is taken with respect to $\mu_{{\rm{nonc}}}$ on $\rm{CRSF}_{{\rm{nonc}}}(\Psi_n)$, the product is over all cycles $\gamma$ of the CRSF and ${\rm{Tr}}(w_{\gamma})$ is the trace of the monodromy of $\nabla^F$ evaluated along $\gamma$.
		Also, for any $(F, h^F, \nabla^F)$ as above, the right-hand side of (\ref{eq_exp_value}) is conformally invariant.
	\end{thm}
	\begin{rem}\label{rem_exp_value}
		a) The existence of the limit on the left-hand side of (\ref{eq_exp_value}) and the conformal invariance of its value, $G(F)$, has been proved by Kassel-Kenyon \cite[Theorem 17]{KassKenyon}, cf. Theorem \ref{thm_kas_kenyon_conv}. 
		However, no geometric meaning for $G(F)$ was given in \cite{KassKenyon}.
		Our proof is a direct combination of their result, the result of Kenyon, (\ref{eq_crsf_1}), and Theorem \ref{thm_rel_as_compexity}.
		\par 
		b) In a similar situation, when the loops are induced by double-dimers, the analogical result has been proven by Dubédat \cite[Theorem 1]{DubedDDimersCLE}. In this situation, it is expected, cf. \cite[Corollary 3]{DubedDDimersCLE}, that the values of those limits, along with the calculation of the expected values for ${\rm{CLE}}_4$, is the only analytical prerequisite to proving the convergence of double-dimer loops to ${\rm{CLE}}_4$.
	\end{rem}
	\par Now, a \textit{finite lamination} on $\overline{\Psi}$ is an isotopy class of a \textit{non contractible} multiloop (i.e. a finite collection of pairwise disjoint simple non contractible closed curves).
	For $T \in \rm{CRSF}_{{\rm{nonc}}}(\Psi_n)$, we denote by $[T]$ the lamination induced by its set of loops.
	For any lamination $L$ on $\overline{\Psi}$, we define
	\begin{equation}
		E_L(\Psi_n) := \Big\{ 
			T \in \rm{CRSF}_{{\rm{nonc}}}(\Psi_n) : [T] = L
		\Big\}.
	\end{equation}
	\par We suppose from now on that $\Psi$ has at least one boundary component. Theorem \ref{thm_meas_cyl_ev} gives an explicit formula for $\lim_{n \to \infty} \mu_{\rm{nonc}}(E_L(\Psi_n))$ in case if $\Psi$ can be tiled by squares of equal sizes.
	\par 
	More precisely, let $\mathcal{M}$ be the space of flat unitary vector bundles of rank $2$ on $\overline{\Psi}$ modulo gauge transformations. 
	This space is compact, it is usually called \textit{the representation variety} and
	\begin{equation}
		\mathcal{M} = {\rm{Hom}}(\pi_1(\overline{\Psi}), {\rm{SU}}(2)) / {\rm{SU}}(2),
	\end{equation}
	where the action of ${\rm{SU}}(2)$ on ${\rm{Hom}}(\pi_1(\overline{\Psi}), {\rm{SU}}(2))$ is given by the conjugation.
	\par 
	Since we assumed that the number of boundary components of $\Psi$ is positive, the fundamental group $\pi_1(\overline{\Psi})$ is a free group $\mathbb{F}_m$ on
$m = 2g(\Psi)-1+b(\Psi)$ letters, where $g(\Psi)$ is the genus of $\overline{\Psi}$ and $b(\Psi)$ is the number of boundary components of $\Psi$. 
	\par 
	For any choice of generators for $\pi_1(\overline{\Psi})$, we consider the measure $\nu$ on $\mathcal{M}$ obtained by the image of the Haar measure on ${\rm{SU}}(2)^{\times m}$. 
	This measure is independent of that choice of the basis, cf. \cite[end of \S 4.5.1]{KassKenyon}, and we call it the \textit{canonical Haar measure} on $\mathcal{M}$. 
	\par Fock-Goncharov in \cite[Theorem 12.3]{FockGonch} proved that, seen as functions on $\mathcal{M}$, the functions 
	\begin{equation}\label{eq_defn_sl}
	S_L(F) = \prod_{\gamma \in L} {\rm{Tr}}(\omega_{\gamma}),
	\end{equation}
	where $\omega_{\gamma}$ is the monodromy of $(F, \nabla^F)$, evaluated along $\gamma$, are linearly independent, when $L$ runs through all finite laminations. 
	Hence the following functions are linearly independent as well
	\begin{equation}\label{eq_defn_tl}
		T_L(F) =\prod_{\gamma \in L} (2 - {\rm{Tr}}(\omega_{\gamma})).
	\end{equation}
	\par We fix a triangulation on $\Psi$ and define the complexity $n(L)$ of a given lamination $L$ on $\overline{\Psi}$ as the minimal number of intersections of loops having isotopy type of the lamination with the edges of the triangulation. We introduce a partial order $\prec$ on the laminations as follows: we say that $L' \prec L$ if the inequality of the same sign holds for all the intersection numbers with all the edges of the triangulation associated to $L'$ and $L$. 
	\par 
	Choose an ordering of $T_L(F)$ consistent with the partial order on laminations defined above. 
	Let $P_L(F)$ be the Gram-Schmidt orthonormalization of $T_L(F)$ with respect to this ordering and with respect to the $L^2$-product induced by $\nu$ on functions over $\mathcal{M}$.
	Let $A_{L,L'}$ be the coefficients of the change of the basis between $P_L(F)$ and $T_{L'}(F)$.
	Those coefficients depend only on the topological type of $\overline{\Psi}$, they satisfy $A_{L, L'} = 0$ for $L \prec L'$, and they are defined by
	\begin{equation}\label{eq_allprim_defn}
		P_L(F) = \sum_{L' \preceq L} A_{L, L'} \cdot T_{L'}(F).
	\end{equation}
	\begin{thm}\label{thm_meas_cyl_ev}
		We fix a half-translation surface $(\Psi, g^{T\Psi})$ and suppose that it can be tiled by euclidean squares of equal sizes.
		Suppose also that $\Psi$ has at least one boundary component.
		Then for any lamination $L$ on $\overline{\Psi}$, the integrals in the following sum converge, as well as the sum itself
		\begin{equation}\label{eq_infin_sum_conv_cyl_ev}
			\sum_{L \preceq L'} A_{L', L} \cdot \int_{\mathcal{M}} \sqrt{\smash[b]{\det{}' \laplcomp_{\Psi}^{F}}}
			 \cdot P_{L'}(F) d \nu.
		\end{equation}
		Moreover, for a constant $Z(\Psi)$ from Theorem \ref{thm_exp_value}, and for any lamination $L$, we have
		\begin{equation}\label{eq_meas_cyl_ev}
			\lim_{n \to \infty} \mu_{\rm{nonc}}(E_L(\Psi_n)) = \frac{1}{Z(\Psi)} \sum_{L \preceq L'} A_{L', L} \cdot \int_{\mathcal{M}}  \sqrt{\smash[b]{\det{}' \laplcomp_{\Psi}^{F}}} \cdot P_{L'}(F) d \nu.
		\end{equation}
	\end{thm}
	\begin{rem}\label{rem_meas_cyl_ev}
		The formula (\ref{eq_meas_cyl_ev}) has already appeared in Kassel-Kenyon \cite[Lemma 16, proof of Theorem 18 and end of \S 4.1]{KassKenyon}, but instead of $\sqrt{\smash[b]{\det{}' \laplcomp_{\Psi}^{F}}} / Z(\Psi)$ there was a functional $G(F)$  from Remark \ref{rem_exp_value}a), which didn't have any explicit geometrical meaning. So the proof of Theorem \ref{thm_meas_cyl_ev} is a direct combination of Theorem \ref{thm_exp_value}, the results from \cite{KassKenyon} (which state, in particular, the conformal invariance of $G(F)$ and the convergence of (\ref{eq_infin_sum_conv_cyl_ev}) with $\sqrt{\smash[b]{\det{}' \laplcomp_{\Psi}^{F}}}$ replaced by $G(F)$) and the exponential bound on the coefficients $A_{L', L}$ in terms of the complexity $n(L')$, which was assumed in \cite{KassKenyon} and later proved by Basok-Chelkak \cite[Theorem 4.9]{BasokChelk}.
	\end{rem}
	\par 
	Now, let's put Theorem \ref{thm_asympt_exansion} in the context of the previous works. 
	Recall the following result
	 \begin{thm}[{Kenyon  \cite[Theorem 1 and Corollary 2]{KenyonDet}}]\label{thm_kenyon_thm}
		We fix a simply-connected rectangular domain $\Omega$ in $\comp$, and denote by $\Omega_n$ a sequence of subgraphs of $\frac{1}{n} \integ^2$, approximating $\Omega$. 
		Let $\partial V(\Omega_n) \subset V(\Omega_n)$ be the set of closest points to $\partial \Omega$. We define $\log^{{\rm{ren}}}(N_{\Omega_n, 0})$, $n \in \nat^*$ as follows
		\begin{multline}
			\log^{{\rm{ren}}}(N_{\Omega_n, 0})
			:=
			\log(N_{\Omega_n, 0})
			-
			\frac{4G}{\pi} \cdot \# V(\Omega_n)
			\\
			-
			\frac{\log(\sqrt{2} - 1)}{2} \cdot \# \partial V(\Omega_n)
			+
			\Big(
				\frac{1}{2} + \frac{\# {\rm{Ang}}^{\neq \pi/2}(\Omega)}{18}
			\Big)
			\cdot
			\log(n)
		\end{multline}
		Then, as $n \to \infty$, the following asymptotic bound holds
		\begin{equation}
			\log^{{\rm{ren}}}(N_{\Omega_n, 0}) = o(\log(n)).
		\end{equation}
		\par Moreover, there is a universal constant $C \in \real$ and a sequence $DA_n$, which depends only on the number of convex angles of $\Omega$, such that, as $n \to \infty$, the following limit holds
		\begin{equation}
			\log^{{\rm{ren}}}(N_{\Omega_n, 0}) - DA_n \to \int_{\Omega}^{{\rm{reg}}} | \nabla h|^2 dx dy + C,
		\end{equation}
		where $\int_{\Omega}^{{\rm{reg}}} | \nabla h|^2 dx dy$ is the regularized Dirichlet energy of the the limiting average height function on $\Omega$, \cite[\S 2.3 and \S 2.4]{KenyonDet}, obtained by subtraction of the logarithmic divergent part from $c_2(1/n) \log(n) + c_3(\Omega)$ in the notation of \cite[Theorem 1]{KenyonDet}.
	\end{thm}
	\begin{rem}\label{rem_ken_shift_latt}
		a)
		This theorem continues to hold for approximations induced by $\frac{1}{n}\integ^2 + \frac{1 + \imun}{2n}$.
		\par 
		b)
		Kenyon asked in \cite[\S 8, question 2]{KenyonDet} if the same asymptotics also holds for multiply-connected domains. By (\ref{eq_matrixtree}), we see that Theorem \ref{thm_asympt_exansion} answers this question.
	\end{rem}
	Now, remark first that for a simply-connected rectangular domain $\Omega$ in $\comp$, by (\ref{eq_value_zeta_0}), we have
	\begin{equation}\label{eq_zeta_zero_scrd}
		\zeta_{\Omega}(0) = -1 + 
		\Big(
			\frac{1}{4} + \frac{\# {\rm{Ang}}^{\neq \pi/2}(\Omega)}{36}
		\Big).
	\end{equation}
	By Remark \ref{rem_ken_shift_latt}a), (\ref{eq_matrixtree}) and (\ref{eq_zeta_zero_scrd}), Theorems \ref{thm_asympt_exansion} and \ref{thm_kenyon_thm} are compatible. Moreover, we deduce 
	\begin{cor}
		There are universal constants $C, D \in \real$ such that for any simply-connected rectangular domain $\Omega$ in $\comp$ with integer vertices, we have
		\begin{equation}\label{eq_height_fun_det}
			\int_{\Omega}^{{\rm{reg}}} | \nabla h|^2 dx dy = \log ( \det{}' \laplcomp_{\Omega} ) - \log ( A(\Omega)) + D \cdot \# {\rm{Ang}}^{\neq \pi/2}(\Omega) + C.
		\end{equation}
	\end{cor}
	\begin{rem}
		a) The requirement of integer vertices is of course not necessary. In fact, since Theorem \ref{thm_asympt_exansion} can be stated in the scale-independent level, the relation (\ref{eq_height_fun_det}) holds for simply-connected rectangular domains $\Omega$ in $\comp$ with rational vertices. Now, as any rectangular domain in $\comp$ can be approximated by domains with rational vertices, it would be enough to argue that the quantities involved in (\ref{eq_height_fun_det}) depend continuously under small perturbations of the domain.
		For the left-hand side of (\ref{eq_height_fun_det}), this follows from a formula for the limit of the average height function in terms of the uniformization map of the domain, \cite[\S 2.3]{KenyonDet}.
		For the right-hand side, the continuity follows by the arguments similar to the ones Bismut-Gillet-Soulé used in \cite[Theorem 1.6]{BGS3} to prove the smoothness of Quillen metric. The details will be given elsewhere.
		\par
		b) The relation (\ref{eq_height_fun_det}) appeared in the list of open questions of Kenyon, \cite[\S 8, question 4]{KenyonDet}.
	\end{rem}
	\par 
	This article is organized as follows. In Section \ref{sec_main_asymp}, we recall the necessary definitions, we introduce zeta-functions and we prove Theorem \ref{thm_asympt_exansion} modulo certain uniform bounds on the spectrum and zeta functions and modulo some Szegő-type asymptotic result on the mesh graphs.
	We then give proofs of Theorems \ref{thm_exp_value}, \ref{thm_meas_cyl_ev}.
	In Section \ref{sect_main_zeta}, we establish the needed bounds on the spectrum and on zeta functions.
	In Section \ref{sect_four_an_sq}, by doing Fourier analysis on discrete square, we establish the needed asymptotic result on the mesh graphs.
	In Appendix \ref{sect_hk_sm_t_exp} we recall the small-time asymptotic expansion of the heat kernel.
	Finally, in Appendix \ref{sect_kron} we recall the explicit evaluation of the analytic torsion for tori and rectangles.
	The results of this paper and \cite{FinFinDiffer} were announced in \cite{FinDetCRAS}.	
	\par 
	\textbf{Notation.}
	For a graph $G$, we denote by $V(G)$, $E(G)$ the sets of vertices and edges of $G$ respectively.
	By abuse of notation, we do not distinguish between multigraphs (graphs with multiplicities on edges) and graphs.
	We denote by $\scal{\cdot}{\cdot}_{L^2(G, F)}$ the $L^2$-scalar product on ${\rm{Map}}(V(G), F)$ by
	\begin{equation}
		\scal{f}{g}_{L^2(G, F)} = \sum_{v \in V(G)} \scal{f(v)}{g(v)}_{h^F}, \qquad f, g \in {\rm{Map}}(V(G), F).
	\end{equation}
	Analogically we denote by $\scal{\cdot}{\cdot}_{L^2(G)}$ the $L^2$-scalar product on the set ${\rm{Map}}(E(G), \comp)$.
	We denote by $\norm{\cdot}_{L^2(G)}$ the associated norms.
	We denote by $d_{G} : {\rm{Map}}(V(G), \comp) \to {\rm{Map}}(E(G), \comp)$ the discrete differentiation operator, defined by
	\begin{equation}
		(d_{G}f)(e) := f(h(e)) - f(t(e)), \qquad f \in {\rm{Map}}(V(G), \comp),
	\end{equation}
	where $h(e)$, $t(e)$ are the head and tail of $e$ respectively.
	\par 
	Finally, for $k \in \nat$, we denote by 
	\begin{equation}\label{eq_defn_ck_b}
		\ccal^{k}(\overline{\Psi}, F) := \Big \{ 
			f \in \ccal^{k}(\Psi, F) : \nabla^l f \in L^{\infty}(\Psi, F), \text{ for }l \leq k
		\Big \}.
	\end{equation}
	\par
	\textbf{Acknowledgements.} I would like to thank Dmitry Chelkak, Yves Colin de Verdière for related discussions and their interest in this article, and especially Xiaonan Ma for important comments and remarks. I also would like to thank the colleagues and staff from Institute Fourier, Université Grenoble Alpes, where this article has been written, for their hospitality.

\section{Asymptotic of the discrete determinant}\label{sec_main_asymp}
	The main goal here is to prove Theorems \ref{thm_asympt_exansion}, \ref{thm_exp_value}, \ref{thm_meas_cyl_ev}.
	In Section \ref{sect_sq_t_s_disc}, we recall some necessary definitions regarding flat surfaces and vector bundles on the graphs. 
	In Section \ref{sect_decomp}, we construct a covering of a pillowcase cover $\Psi$ by a finite number of model surfaces. This will be used in the proof of Theorem \ref{thm_asympt_exansion}.
	In Section \ref{sect_zeta_flat_fun}, we study the Laplacian and its zeta function on a flat surface with conical singularities.
	In Section \ref{sect_idea_proof}, we prove  Theorem \ref{thm_asympt_exansion} modulo certain statements which will be proved later in this article.
	Finally, Section \ref{sect_cyl_event}, we prove Theorems \ref{thm_exp_value}, \ref{thm_meas_cyl_ev}.

		\subsection{Pillowcase covers and their discretizations}\label{sect_sq_t_s_disc}
	\par Here we recall the definition of flat surfaces, pillowcase covers, explain some properties of discretizations of pillowcase covers and give a short introduction to vector bundles over graphs.
	\par 
	By Gauss-Bonnet theorem, the only closed Riemann surface admitting a flat metric has the topology of the torus. 
	However, any Riemann surface can be endowed with a flat metric having a finite number of cone-type singularities. 
	Let's explain this point more precise.
	\par 
	A cone-type singularity is a Riemannian metric
	\begin{equation}\label{eq_conical_metric}
		ds^2 = dr^2 + r^2 dt^2,
	\end{equation}
	on the manifold
	\begin{equation}\label{eq_v_theta}
		C_{\theta} := \{(r, t) : r > 0; t \in \real / \theta \integ \},
	\end{equation}
	where $\theta > 0$.
	In what follows, when we speak of cones, we assume $\theta \neq 2\pi$.
	By a \textit{flat metric with a finite number of cone-type singularities} we mean a metric defined away from a finite set of points such that there is an atlas for which the metric looks either like the standard metric on $\real^2$, or like the conical metric (\ref{eq_conical_metric}) on an open subset of (\ref{eq_v_theta}).
	\par 
	Let $M$ be a compact surface endowed with a flat metric with cone-type singularities of angles $\theta_1, \ldots, \theta_m$ at points  $P_1, \ldots, P_m$. 
	By Gauss-Bonnet theorem (cf. Troyanov \cite[Proposition 3]{Troyanov}):
	\begin{equation}\label{eq_gauss_bonnet}
		\sum_{i = 1}^{m} (2\pi - \theta_i) = 2 \pi \chi(M),
	\end{equation}
	where $\chi(M)$ is the Euler characteristic of $M$. 
	Although we will not need this in what follows, by a theorem of Troyanov \cite[Théorème, \S 5]{Troyanov}, this is the only obstruction for the existence of a metric with cone-type singularities of angles $\theta_1, \ldots, \theta_m$ at points  $P_1, \ldots, P_m$ in a conformal class of $M$.
	\par 
 	In this article we are primarily interested in compact surfaces endowed with a metric with a finite number of cone-type singularities of angles $\pi k$, $k \in \nat^*$. 
 	Those Riemann surfaces can also be described by possession of an atlas, defined away from the singularities of the metric, such that the transition maps between charts are given by $z \to \pm z + {\rm{const}}$ (cf. \cite[\S 3.3]{ZorichFlatSurf}).
	In literature, such surfaces are called \textit{half-translation surfaces}.
 	If all the angles are of the form $2 \pi k$, $k \in \nat^*$, the atlas can be chosen in such a way so that the transition maps between charts are given by $z \to z + {\rm{const}}$, and the surface in this case is called a \textit{translation surface}.
 	The name translation surface comes from the fact that such surfaces can be realized by identifying parallel sides of a polygon in $\comp$.
 	Clearly, from any closed half-translation surface we can construct a translation surface by double covering, ramified over the conical points with angles $(2 k + 1) \pi$, $k \in \nat^*$.
 	\par 
 	It is easy to see that for a half-translation surface, a notion of a \textit{straight line} makes sense, and for translation surfaces, even a notion of a \textit{ray} (or a \textit{direction}) is well-defined.
 	\par There is an alternative description of closed \textit{half-translation surfaces} with prescribed line (resp. \textit{translation surfaces} with prescribed direction) in terms of Riemann surfaces endowed with a meromorphic quadratic differential with at most simple poles (resp. a holomorphic differential). The zeros\footnote{Here we interpret a pole of a quadratic differential as a zero of order $-1$.} of order $k$ in this description correspond to the conical points with angles $(k + 2) \pi$ (resp. $2 \pi (k + 1)$). In this description, one could easily identify the moduli spaces of closed half-translation surfaces with line (resp. translation surfaces with direction) with corresponding Hodge bundles on the moduli of curves (cf. \cite[\S 8.1]{ZorichFlatSurf}). The orders of zeros of a meromorphic quadratic differential (resp. of a holomorphic differential) induce stratifications of those moduli spaces.
 	\par 
	For the most part of this paper, we will be interested in considering special type of half-translation surfaces which are called \textit{pillowcase covers}. Those are surfaces with a fixed tiling by squares of equal sizes.
	In case if the surface has no boundary, those surfaces can be thought as rational points in so-called period coordinates (cf. \cite[\S 7.1]{ZorichFlatSurf}) inside the respective stratas of the moduli space, and thus, the set of pillowcase covers form a dense subset of the corresponding moduli spaces (similarly to the set of tori with rational periods inside the moduli space of tori).
	\par Let's now set up some notations associated with discretization $\Psi_n$ of a pillowcase cover $\Psi$, constructed as in Introduction.
For $P \in {\rm{Con}(\Psi)} \cup {\rm{Ang}(\Psi)}$, we define the set $V_n(P)$ as follows
	\begin{equation}\label{eq_defn_vp}
		V_n(P) := \Big\{ Q \in V(\Psi_n) : \dist_{\Psi}(P, Q) = \frac{1}{2n} \Big\},
	\end{equation}
	in other words, $V_n(P)$ is the set of the nearest neighbors of $P$ from the vertex set $V(\Psi_n)$ with respect to the flat metric on $\Psi$. 
	It's easy to verify that for $n \geq 2$, we have 
	\begin{equation}
		\# V_n(P) = \frac{2 \angle(P)}{\pi}.
	\end{equation}
	Let's define for $P \in {\rm{Con}(\Psi)} \cup {\rm{Ang}(\Psi)}$ the open subset $U_n(P) \subset \Psi$ as follows
	\begin{equation}\label{eq_defn_up}
		U_n(P) := \Big\{ z \in \Psi : {\rm{dist}}_{\Psi}(z, P) < \frac{1}{2n} \Big\}.
	\end{equation}
	Remark that the points $V_n(P)$ lie on the boundary of $U_n(P)$.
	\begin{rem}
		For $n \geq 2$, the edges of $\Psi_n$ have at most double multiplicity, see Figure \ref{fig_cpi_graph}.
		Moreover, the number of edges with double multiplicity is equal to $\# \{ P \in {\rm{Con}}(\Psi) : \angle(P) = \pi\}$. We are, thus, working with multigraphs, but by abuse of notation, we call them graphs.
	\end{rem}
	\begin{figure}[h]
		\includegraphics[width=0.2\textwidth]{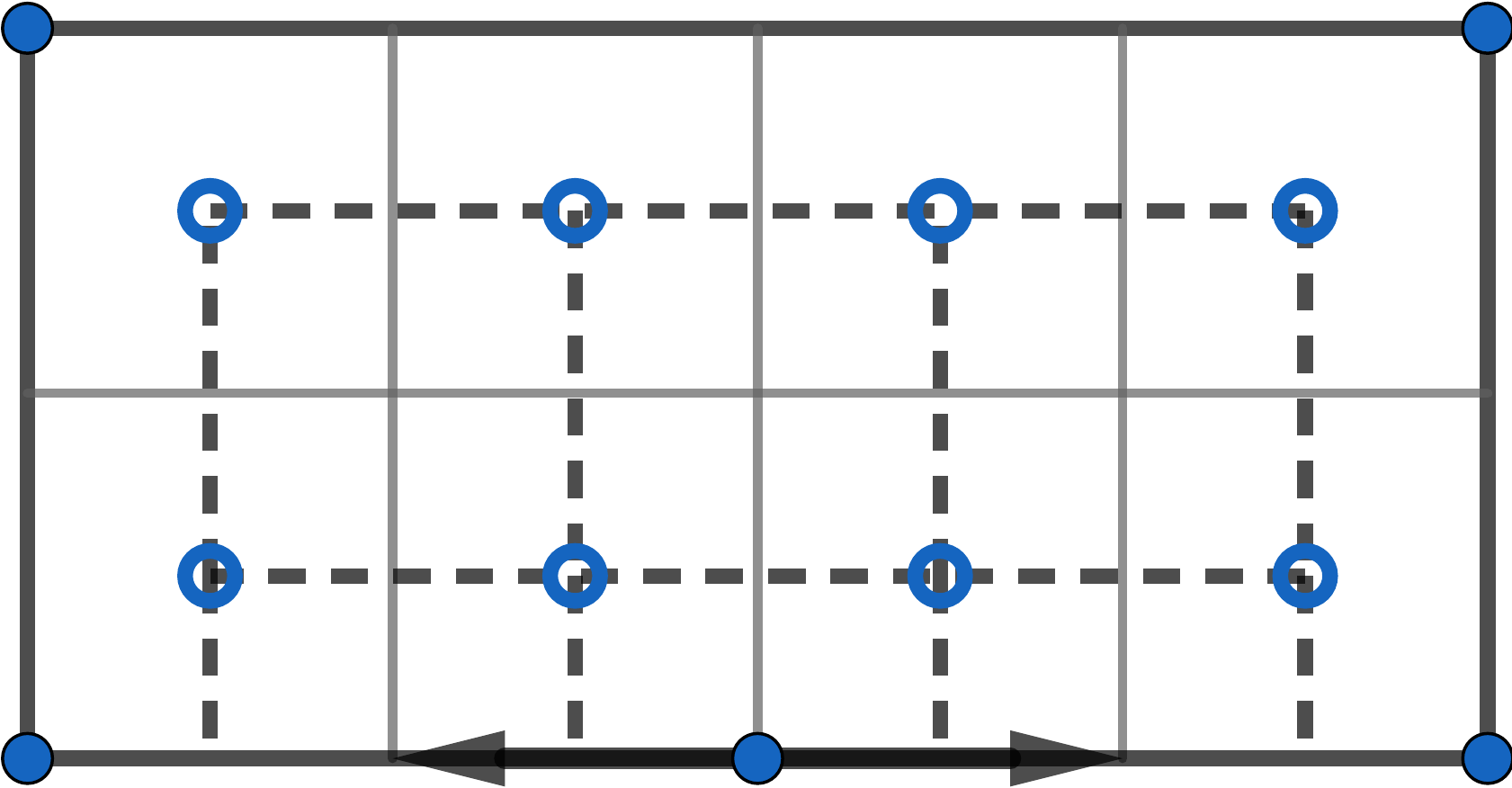}	
		\centering
		\caption{An example of multiple edges in a discretization. The surface is obtained by identifying the edges along the directed lines. The edges of the graph are dotted. }
		\label{fig_cpi_graph}
	\end{figure}
	\par Now, by a \textit{vector bundle} $V$ on a graph $G = (V(G), E(G))$, we mean the choice of a vector space $V_v$ for any $v \in V(G)$ so that for any $v, v' \in V(G)$, the vector spaces $V_v$ and $V_{v'}$ are isomorphic.
	The set of \textit{sections} ${\rm{Map}}(V(G), V)$ of $V$ is defined by
	\begin{equation}
		{\rm{Map}}(V(G), V) = \oplus_{v \in V(G)} V_v.
	\end{equation}
	A \textit{connection} $\nabla^{V}$ on a vector bundle $V$ is the choice for each edge $e =(v, v') \in E(G)$ of an isomorphism $\phi_{v v'}$ between the corresponding vector spaces $\phi_{v v'} : V_v \to V_{v'}$, with the property that $\phi_{v v'} = \phi_{v' v}^{-1}$. This isomorphism is called the \textit{parallel transport} of vectors in $V_v$ to vectors in $V_{v'}$.
	\par 
	A \textit{Hermitian metric} $h^{V}$ on the vector bundle $V$ is a choice of a positive-definite Hermitian metric $h_{v}$ on $V_v$ for each $v \in V(G)$. We say that a connection $\nabla^{V}$ is unitary with respect to $h^V$ if its parallel transports preserve $h^{V}$.
	\par The \textit{Laplacian} $\laplcomp_G^{V}$ on $(V, \nabla^V)$ is the linear operator $\laplcomp_G^{V} : {\rm{Map}}(V(G), V) \to {\rm{Map}}(V(G), V)$, defined for $f \in {\rm{Map}}(V(G), V)$ by (\ref{eq_comb_lapl_tw}).
	Remark that unlike Laplace–Beltrami operator on a smooth manifold, we don't use the metric to define the Laplacian (\ref{eq_comb_lapl_tw}).
	\par Consequently, in general, the operator $\laplcomp_G^{V}$ is not self-adjoint, see for example \cite[\S 3.2 and equation (1)]{KenyonFlatVecBun}.
	However, if one assumes that the connection is unitary with respect to $h^{V}$, then it becomes self-adjoint (cf. Kenyon \cite[\S 3.3]{KenyonFlatVecBun}). 
	\par 
	More precisely, one can extend the definition of a vector bundle to the edges of $G$. 
	A vector bundle $V'$ over $V(G) \oplus E(G)$ is a choice of a vector spaces $V_e$ for each edge $e \in E(G)$ as well as $V_v$ for each vertex $v \in V(G)$. 
	A connection $\nabla^{V'}$ on $V'$ is a choice of a connection $\nabla^V$ on $V$ as well as connection isomorphisms $\phi_{ve} : V_v \to V_e$, $\phi_{ev} : V_e \to V_v$ for each $v \in V(G)$ and $e \in E(G)$, and satisfying $\phi_{ve} = \phi_{ev}^{-1}$ and
	$\phi_{v v'} = \phi_{ev'} \circ \phi_{ve}$.
	Similarly to ${\rm{Map}}(V(G), V)$, we defined the set of sections of ${\rm{Map}}(E(G), V)  := \oplus_{e \in E(G)} V_e$.
	\par Quite easily, for any vector bundle $V$ and a connection $\nabla^V$ on $V(G)$, we may extend it to a vector bundle $V'$ and a connection $\nabla^{V'}$ on $E(G) \oplus V(G)$.
	Note, however, that such a choice would not be unique.
	If the initial vector bundle $V$ is endowed with a Hermitian metric $h^V$, for which the connection $\nabla^V$ is unitary, then one might endow the vector bundles $V_e$, $e \in E(G)$ with metrics and choose the connections $\phi_{ev}$, $e \in E(G)$, $v \in V(G)$ so that $\nabla^{V'}$ is unitary as well.
	\par There is a natural map $\nabla^V_{G} : {\rm{Map}}(V(G), V) \to {\rm{Map}}(E(G), V)$, defined as follows 
	\begin{equation}
		(\nabla^V_{G} f)(e)
		=
		\phi_{t(e)e} f(t(e)) - \phi_{h(e)e} f(h(e)), 
		\qquad
		f \in {\rm{Map}}(V(G), V), 
	\end{equation}	 
	where $t(e)$ and $h(e)$ are tail and head respectively of an oriented edge $e$.
	We also define the operator $(\nabla^V_{G})^* : {\rm{Map}}(E(G), V) \to {\rm{Map}}(V(G), V)$ by the formula
	\begin{equation}
		((\nabla^V_{G})^* f)(v) = \sum_{\substack{e \in E(G) \\ t(e) = v}} \phi_{e v} f(e).
	\end{equation}
	It is an easy verification (cf. Kenyon \cite[\S 3.3]{KenyonFlatVecBun}) that for the Laplacian, defined by (\ref{eq_comb_lapl_tw}), we have
	\begin{equation}\label{eq_lapl_self_adjoint}
		\laplcomp_G^{V} = (\nabla^V_{G})^* \nabla^V_{G}.
	\end{equation}
	Note, however, that in general $(\nabla^V_{G})^*$ is not the adjoint of $\nabla^V_{G}$ with respect to the appropriate $L^2$-metrics. 
	But if the connection $\nabla^V$ is unitary, it is indeed the case, cf. \cite[\S 3.3]{KenyonFlatVecBun}.
	\par 
	In this article, all our connections are unitary, and thus, by (\ref{eq_lapl_self_adjoint}), the associated discrete Laplacians are self-adjoint and positive.

	\subsection{Covering a pillowcase cover with model surfaces}\label{sect_decomp}
	In this section we construct a covering of a pillowcase cover $\Psi$ by a finite number of model surfaces. We also construct a subordinate partition of unity in such a way that this partition of unity, regarded as functions on the covering, form a number of model cases.
	Those constructions will play an important role in our proof of Theorem \ref{thm_asympt_exansion}.
	\par 
	To explain our construction better, let's suppose first that $\Psi$ is a rectangular domain in $\comp$ without slits and with integer vertices.
	Then we can easily verify that up to making a homothety of factor $4$, one can construct a covering $U_{\alpha}$, $\alpha \in I$, of $\Psi$ so that each $U_{\alpha}$ is isomorphic either to a square, tiled with $4$ tiles of $\Psi$, or to the $L$-shape, $\mathcal{A}_{\frac{3 \pi}{2}}$, tiled with $12$ tiles of $\Psi$, see Figure \ref{fig_L_shape}.
	\par 
	Moreover, we see that for the number $N_{\mathcal{A}_{\frac{3 \pi}{2}}}(\Psi)$ of $U_{\alpha}$, isomorphic to $\mathcal{A}_{\frac{3 \pi}{2}}$, we have
	\begin{equation}\label{eq_count_na}
		N_{\mathcal{A}_{\frac{3 \pi}{2}}}(\Psi) = \# \Big\{
			Q \in {\rm{Ang}}(\Psi) : \angle(Q) = \frac{3\pi}{2}
		\Big\}.
	\end{equation}
	Analogically, denote by $N_{sq, i}(\Psi)$, $i = 0, 1, 2$ the number  of $U_{\alpha}$, which are isomorphic to a square sharing $i$ sides with the boundary of $\Psi$.
	We have the following identity
	\begin{equation}\label{eq_count_nsq2}
		N_{sq, 2}(\Psi) = \# {\rm{Ang}}^{=\pi/2}(\Psi).
	\end{equation}
	\par 
	Let's now give an expression for $N_{sq, i}(\Psi)$ for $i = 0, 1$.
	Clearly, if one puts the numbers on the tiles of the covering as it is shown in the Figure \ref{fig_numbers_l_shape_sq}, then after summing the numbers in the covering, corresponding to each tile, we get identically $1$ over $\Psi$. 
	\begin{figure}[h]
		\includegraphics[width=0.7\textwidth]{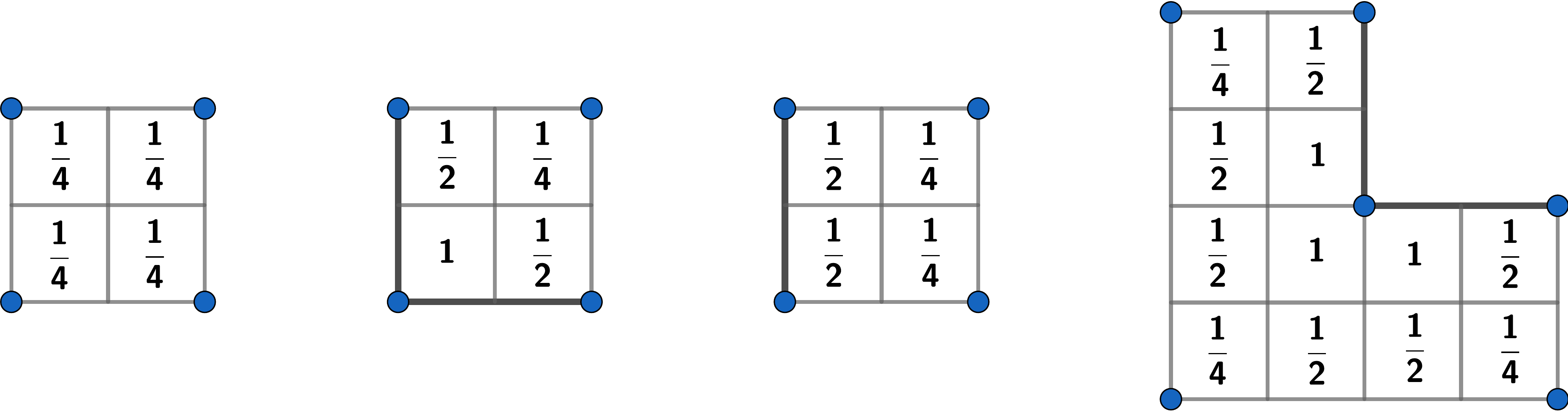}	
		\centering
		\caption{The bold lines represent the boundary $\partial \Psi$, small squares represent the tiles of $\Psi$.}
		\label{fig_numbers_l_shape_sq}
	\end{figure}	
	\par 
	The sum of the numbers in tiles corresponding to a square without common boundaries with $\partial \Psi$ is $1$, with $1$ common boundary - $\frac{3}{2}$, with $2$ common boundaries - $\frac{9}{4}$, and the sum of tiles corresponding to an $L$-shape is $\frac{27}{4}$.
	Also, the sum of the numbers in tiles sharing common side with $\partial \Psi$ in a square with $1$ common boundary with $\partial \Psi$ is $1$, with $2$ common boundaries - $3$, and and in the $L$-shape - $3$.
	By calculating $|\partial \Psi|$ and $A(\Psi)$ in two different ways, we obtain from this
	\begin{equation}\label{eq_count_areperim}
	\begin{aligned}
		&
		|\partial \Psi| = \Big( N_{sq, 1} + 3 N_{sq, 2} + 3 N_{\mathcal{A}_{\frac{3 \pi}{2}}} \Big) (\Psi),
		\\
		&
		A(\Psi) = \Big( N_{sq, 0} + \frac{3}{2} N_{sq, 1} + \frac{9}{4} N_{sq, 2} + \frac{27}{4} N_{\mathcal{A}_{\frac{3 \pi}{2}}} \Big) (\Psi).
	\end{aligned}
	\end{equation}
	\par 
	Now, similarly, a general pillowcase cover $\Psi$, considered up to a homothety of factor $4$, can be covered by squares, tiled with $4$ tiles of $\Psi$, and the following two types of of model examples of pillowcase covers $\Psi$: the model angles $\mathcal{A}_{\frac{k \pi}{2}}$, $k \geq 3$ of angle $\frac{k \pi}{2}$, the special case of which are the $L$-shape and slit considered in Figures \ref{fig_L_shape}, \ref{fig_slit}, and the model cones $\mathcal{C}_{k \pi}$, $k = 1$ or $k \geq 3$, of angle $k \pi$, which did not appear before, as they only appear in surfaces with conical singularities.
	\par 
	More precisely, let's describe a model cone $\mathcal{C}_{\pi}$ of angle $\pi$.
	Consider a rectangle, which has ratios of sides $2 : 1$ and which is tiled with $8$ squares.
	Glue the intervals on one side of it as it is shown in Figure \ref{fig_cpi_graph}.
	Endow the resulting surface $\mathcal{C}_{\pi}$ with the pillowcase structure coming from $8$ tiles and a metric induced by the standard metric on $\comp$.
	The resulting metric is flat with one conical angle $\pi$ and with two corners with angles $\frac{\pi}{2}$.
	The perimeter of the surface is equal to $8$.
	\par Now let's describe the model cones $\mathcal{C}_{2k\pi}$ of angles $2k\pi$, $k \in \nat^* \setminus \{1\}$.
	Consider the covering 
	\begin{equation}\label{eq_pik_defn}
		\pi_{k}: C_{2 k \pi} \to \comp,
	\end{equation}
	unramified everywhere except the origin, where the order of ramification is $k$.
	Then we define 
	\begin{equation}
		\mathcal{C}_{2 k \pi} = \pi_{k}^{-1}([-2, 2] \times [-2, 2]).
	\end{equation}
	The surface $\mathcal{C}_{2 k \pi}$ is endowed with the pull-back metric from $\comp$ and with the structure of the pillowcase cover coming from the $16$-tile structure on $[-2, 2] \times [-2, 2]$, given by euclidean squares of area $1$.
	In other words, $\mathcal{C}_{2 k \pi}$ has $16 k$ tiles, one conical angle $2 k \pi$ and $4k$ angles of the boundary equal to $\frac{\pi}{2}$. For an example, see Figure \ref{fig_c4pi}.
	\begin{figure}[h]
		\includegraphics[width=0.4\textwidth]{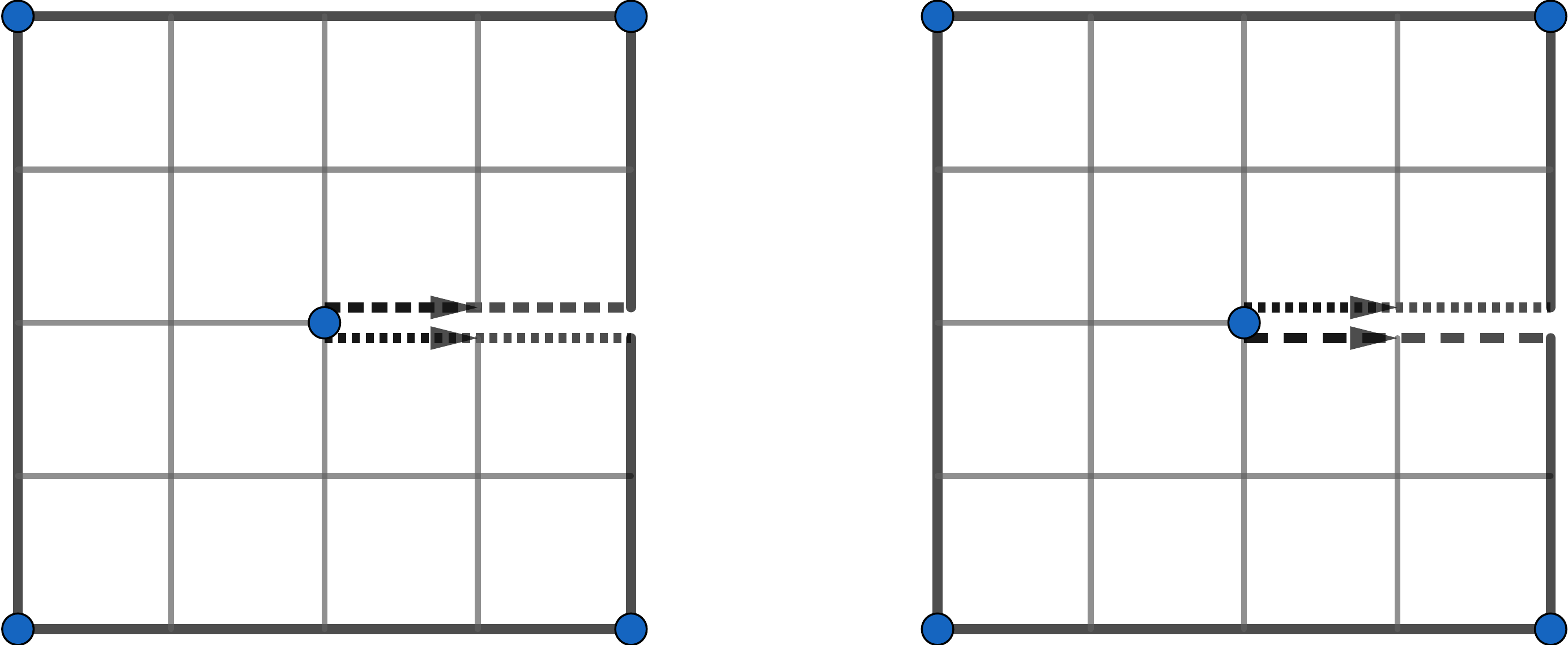}	
		\centering
		\caption{The model space $\mathcal{C}_{4\pi}$. It is obtained by gluing the edges of the same pattern.}
		\label{fig_c4pi}
	\end{figure}	
	\par Now let's construct  the model cones $\mathcal{C}_{(2k + 1)\pi}$ of angles $(2k+1)\pi$, $k \in \nat^*$.
	Consider a surface $\mathcal{C}_{2 k \pi}$ constructed in the previous step.
	Introduce a single cut on this surface such that its projection in $\comp$ coincides with the straight interval from $0$ to $2$.
	Consider a rectangle tiled with $8$ squares as in the construction of $\mathcal{C}_{\pi}$.
	Glue the sides of the slits with the sides of the rectangle as it is shown in Figure \ref{fig_c3pi}.
	Endow the resulting surface with the induced metric and pillowcase structure coming from corresponding structures on on $\mathcal{C}_{2 k \pi}$ and $\mathcal{C}_{\pi}$.
	The resulting surface $\mathcal{C}_{(2 k + 1) \pi}$ has $16 k + 8$ tiles and a flat metric with one conical angle $(2 k + 1) \pi$ and $4k + 2$ angles of the boundary equal to $\frac{\pi}{2}$.
	\begin{figure}[h]
		\includegraphics[width=0.4\textwidth]{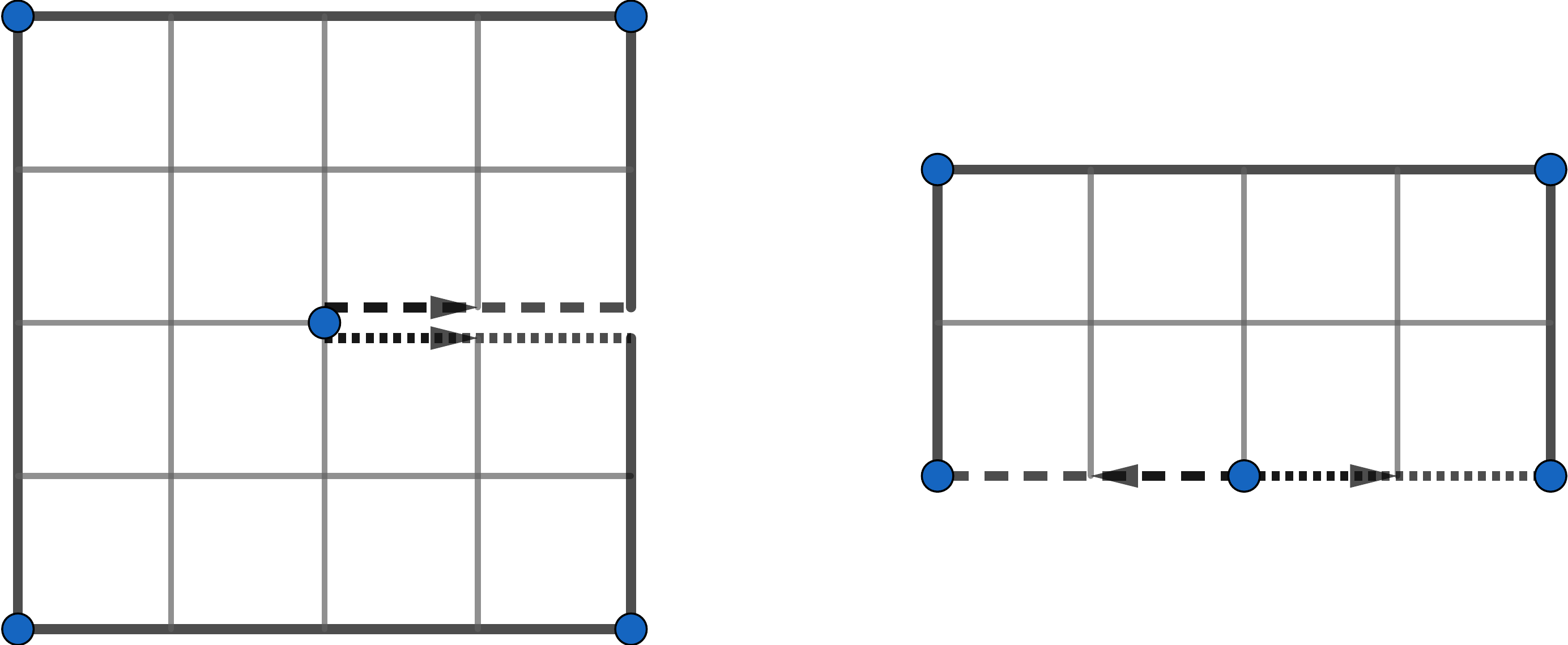}	
		\centering
		\caption{The model space $\mathcal{C}_{3\pi}$. It is obtained by gluing the edges of the same pattern.}
		\label{fig_c3pi}
	\end{figure}	
	\par Finally, let's describe the model angles $\mathcal{A}_{\frac{k \pi}{2}}$ of angle $\frac{k \pi}{2}$, $k \geq 3$.
	Construct a surface $\mathcal{C}_{k \pi}$ by the procedure above and introduce two cuts projecting to a vertical or to a horizontal intervals from $0$ of length $2$, so that the two different angles between those cuts at $0$ coincide. 
	Then the surface, obtained by deletion of those cuts, has $2$ isomorphic connected components, which we denote by $\mathcal{A}_{\frac{k \pi}{2}}$.
	The pillowcase structure and the metric is induced by the corresponding structures from $\mathcal{C}_{k \pi}$.
	The surface $\mathcal{A}_{\frac{k \pi}{2}}$ has $4k$ tiles, it is endowed with flat metric with no conical points; at the boundary it has one angle $\frac{k \pi}{2}$ and $k + 2$ angles $\frac{\pi}{2}$.
	For examples, see Figures \ref{fig_L_shape}, \ref{fig_slit}.
	\par 
	Clearly, the formulas (\ref{eq_count_na}), (\ref{eq_count_nsq2}) still hold, and the identities in the spirit of (\ref{eq_count_areperim}) continue to hold as well. 
	We will only need a weak version of them, which we now state. 
	\begin{lem}
		There are constants $c_0(\Psi), c_1(\Psi) \in \real$, which depend only on the sets $\angle({\rm{Ang}}(\Psi))$ and $\angle({\rm{Con}}(\Psi))$, for which the following holds
	\begin{equation}\label{eq_number_sq_count}
	\begin{aligned}
		&
		N_{sq, 0}(\Psi) = {\rm{A}}(\Psi) - \frac{3}{2} |\partial \Psi| + c_0(\Psi),
		\\
		&
		N_{sq, 1}(\Psi) = |\partial \Psi| + c_1(\Psi).
	\end{aligned}
	\end{equation}
	\end{lem}
	\begin{proof}
		The proof of (\ref{eq_number_sq_count}) is identical to the proof of (\ref{eq_count_areperim}).
	\end{proof}
	\par
	Now let's describe the choice of the subordinate partition of unity $\phi_{\alpha}$ associated with $U_{\alpha}$, so that, viewed as functions on the covering, for each model space, we have exactly one function.
	Remark that an example of non-continuous “partition of unity" of this form is given in Figure \ref{fig_numbers_l_shape_sq}.
	\par 
	We fix a function $\rho: [-1, 1] \to [0, 1]$, satisfying
	\begin{align}
		& \rho(x) = 
		\begin{cases} 
			\hfill 1, & \text{ in the neighborhood of $x = 0$}, \\
			\hfill 0, & \text{ in the neighborhood of $x = \pm 1$}.
 		\end{cases} 
 		\label{eq_rho1_rest1}
 		\\
 		& 
 		\rho(x) = \rho(-x),  
 		\label{eq_rho1_rest2}
 		\\
 		&
 		\rho(x + 1/2) + \rho(1/2 - x) = 1, \qquad \text{for} \quad x \in [0, 1/2].
 		\label{eq_rho1_rest3}
	\end{align}
	\par Now, suppose that for $\alpha \in I$, the open set $U_{\alpha}$ corresponds to a square. 
	Choose linear coordinates $x, y$ with axes parallel to the boundaries of the tiles and normalize them so that they identify the square with $[-1, 1] \times [-1, 1]$.
	If $U_{\alpha}$ doesn't share a boundary with $\partial \Psi$, we define 
	\begin{equation}
		\phi_{\alpha}(x, y) := \phi_{sq, 0}(x, y) := \rho(x)\rho(y).
	\end{equation}
	\par Now, suppose that $U_{\alpha}$ shares one boundary with $\partial \Psi$.
	Suppose without loosing the generality that this boundary corresponds to $\{x = -1\}$.
	Then we define 
	\begin{equation}\label{eq_psi_sq_1}
		\phi_{\alpha}(x, y) := \phi_{sq, 1}(x, y) := 
		\begin{cases} 
			\hfill \rho(x)\rho(y), & \text{ for $x \geq 0$}, \\
			\hfill \rho(y), & \text{ for $x \leq 0$}.
 		\end{cases} 
	\end{equation}
	\par Suppose, finally, that $U_{\alpha}$ shares two boundaries with $\partial \Psi$.
	Suppose without loosing the generality that those boundaries correspond to $\{x = -1\}$ and $\{ y = -1 \}$.
	Then we define 
	\begin{equation}\label{eq_psi_sq_2}
		\phi_{\alpha}(x, y) := \phi_{sq, 2}(x, y) := 
		\begin{cases} 
			\hfill \rho(x)\rho(y), & \text{ for $x \geq 0, y \geq 0$}, \\
			\hfill \rho(y), & \text{ for $x \leq 0, y \geq 0$}, \\
			\hfill \rho(x), & \text{ for $x \geq 0, y \leq 0$}, \\
			\hfill 1, & \text{ for $x \leq 0, y \leq 0$}.
 		\end{cases} 
	\end{equation}
	This finishes the description in case if $U_{\alpha}$ is a square.
	\par Now, to describe the choice of the function $\phi_{\alpha}$ on model angles and cones, let's define the function $\phi_{0} : [-2, 2] \times [-2, 2] \to [0, 1]$. For $(x, y) \in [-2, 2] \times [-2, 2]$, it satisfies the following 
	\begin{equation}
		\phi_{0}(x, y) = \phi_{0}(-x, y) = \phi_{0}(x, -y).
	\end{equation}
	Thus, it is enough to define it only in the quadrant $x, y \geq 0$.
	In this quadrant it is defined by
	\begin{equation}
		\phi_{0}(x, y) := \phi_{sq, 2}(x - 1, y - 1), \qquad \text{for} \quad (x, y) \in [0, 2] \times [0, 2]. 
	\end{equation}
	\par Now let's suppose that $U_{\alpha}$ is isometric with $\mathcal{C}_{\pi}$.
	Choose linear coordinates $x, y$ with axes parallel to the boundaries of the tiles and normalize them so that they identify the rectangle in Figure \ref{fig_cpi_graph} with $[-2, 2] \times [0, 2]$.
	Normalize $y$ so that the sides, which get glued, lie on the line $\{ y = 0\}$.
	Then we define the corresponding function $\phi_{\alpha}$ by restricting $\phi_{0}$.
	Clearly, by (\ref{eq_rho1_rest1}) and (\ref{eq_rho1_rest2}), the function $\phi_{\alpha}$ is well-defined and smooth on $\mathcal{C}_{\pi}$.
	\par Now let's suppose that $U_{\alpha}$ is isometric to $\mathcal{C}_{2 k \pi}$, $k \in \nat^*$.
	Then we define 
	\begin{equation}
		\phi_{\alpha} := \phi_{\mathcal{C}_{2 k \pi}} := (\pi_k)^{-1} \phi_{0},
	\end{equation} 
	where $\pi_k$ was defined in (\ref{eq_pik_defn}).
	\par For $U_{\alpha}$ isometric to $\mathcal{C}_{(2 k + 1) \pi}$, $k \in \nat^*$, the function $\phi_{\alpha} := \phi_{\mathcal{C}_{(2 k + 1) \pi}}$ is defined by gluing the corresponding functions on $\mathcal{C}_{\pi}$ and $\mathcal{C}_{2 k \pi}$ through a pattern used in the definition of $\mathcal{C}_{(2 k + 1) \pi}$. 
	This function is well-defined and smooth by (\ref{eq_rho1_rest1}) and (\ref{eq_rho1_rest2}).
	\par For $U_{\alpha}$ isometric to $\mathcal{A}_{\frac{k \pi}{2}}$, $k \geq 3$, the function $\phi_{\alpha} := \phi_{\mathcal{A}_{\frac{k \pi}{2}}}$ is defined by restriction of the corresponding function on $\mathcal{C}_{k \pi}$.
	\begin{lem}
		The functions $\phi_{\alpha}$, $\alpha \in I$, form a partition of unity of $\Psi$, subordinate to $U_{\alpha}$.
	\end{lem}
	\begin{proof}
		The assumptions (\ref{eq_rho1_rest2}) and (\ref{eq_rho1_rest3}) justify this statement.
	\end{proof}
	
	\subsection{Zeta functions on flat surfaces}\label{sect_zeta_flat_fun}
		Here we study the Laplacian $\laplcomp_{\Psi}^{F}$, (\ref{eq_lapl}) on a flat surface $(\Psi, g^{T \Psi})$ with conical singularities and piecewise geodesic boundary $\partial \Psi$ endowed with a flat unitary vector bundle $(F, h^F, \nabla^F)$. 
		We define the zeta function for the Friedrichs extension of the Laplacian, and study some of its properties.
		\par There are many ways to motivate the choice of Friedrichs extension. To name one, it is positive (cf. \cite[Theorem X.23]{ReedSimonII}), which is rather handy since the discrete Laplacians we consider here are positive as well (see the end of Section \ref{sect_sq_t_s_disc}). 
	Moreover in \cite{FinFinDiffer}, (cf. Theorems \ref{thm_eigval_convergence}, \ref{thm_eigvec_convergence}), we have proved that the eigenvalues and the eigenvectors of Friedrichs extension can be obtained as limits of the eigenvalues and eigenvectors of the rescaled twisted discrete Laplacians $n^2 \cdot \laplcomp_{\Psi_n}^{F_n}$, $n \in \nat^*$.
	\par Note, however, that in the end of the day, the only thing we need from our Laplacian in this article is its determinant (see Definition \ref{defn_an_tors_zeta}). 
	But by a result of Hillairet-Kokotov \cite[Theorem 1]{HilKok}, the determinants of different self-adjoint extensions are related (at least in the case when there is no boundary and the vector bundle is trivial, as considered in \cite{HilKok}).
		\par 
		The content of this section is certainly not new, but we weren't able to find a complete reference for all the results contained here.
		\par
		We consider $\laplcomp_{\Psi}^{F}$ as an operator acting on the functional space $\ccal^{\infty}_{0, vN}(\Psi, F)$, where
		\begin{equation}
			\ccal^{\infty}_{0, vN}(\Psi, F) := \Big\{ f \in \ccal^{\infty}_{0} \big(\Psi \setminus {\rm{Ang}(\Psi)} \big) : \nabla_n f = 0 \text{ over } \partial \Psi \Big\},
		\end{equation}	
		where $n$ is the normal to $\partial \Psi$.
		Unlike in the case of a manifold with smooth boundary, the operator $\laplcomp_{\Psi}^{F}$ is in general not essentially self-adjoint.
		\par 
		Let ${\rm{Dom}}_{\max}(\laplcomp_{\Psi}^{F})$ denote the maximal closure of $\laplcomp_{\Psi}^{F}$. In other words, for $u \in L^2(\Psi, F)$, we have $u \in {\rm{Dom}}_{\max}(\laplcomp_{\Psi}^{F})$ if and only if $\laplcomp_{\Psi}^{F} u \in L^2(\Psi, F)$, where $\laplcomp_{\Psi}^{F} u$ is viewed as a current.
		\par 
	 	Let's denote by $H^1(\Psi, F)$ the Sobolev space on $\Psi$, defined as
	 	\begin{equation}
	 		H^1(\Psi, F) = \Big\{ u \in L^2(\Psi, F) : \nabla u \in L^2(\Psi, F)  \Big\}.
	 	\end{equation}
	 	We denote by $\norm{\cdot}_{H^1(\Psi, F)}$ the norm on $H^1(\Psi, F)$, given for $u \in H^1(\Psi, F)$ by
	 	\begin{equation}
	 		\norm{u}_{H^1(\Psi, F)} = \|  u \|_{L^2(\Psi, F)} +  \| \nabla u \|_{L^2(\Psi, F)}.
	 	\end{equation}
	 	\par 
	 	For any positive symmetric operator, one can construct in a canonical way a self-adjoint extension, called Friedrichs extensions, through the completion of the associated quadratic form, cf. \cite[Theorem X.23]{ReedSimonII}.
	 	Once the definition is unraveled, the domain ${\rm{Dom}}_{Fr}(\laplcomp_{\Psi}^{F})$ of the Friedrichs extensions of the Laplacian $\laplcomp_{\Psi}^{F}$ on $\Psi$ with von Neumann boundary conditions on $\partial \Psi$ is given by
	 	\begin{equation}\label{eq_dom_friedr}
	 		{\rm{Dom}}_{Fr}(\laplcomp_{\Psi}^{F}) = {\rm{Dom}}_{\max}(\laplcomp_{\Psi}^{F})  \cap H^{1}_{0, vN}(\Psi, F),
	 	\end{equation}
	 	where $H^{1}_{0, vN}(\Psi, F)$ is the closure of $\ccal^{\infty}_{0, vN}(\Psi, F)$ in $H^1(\Psi, F)$.
	 	The value of the Friedrichs extensions of the Laplacian $\laplcomp_{\Psi}^{F}$ on $f \in {\rm{Dom}}_{Fr}(\laplcomp_{\Psi}^{F})$ is defined in the distributional sense. By the definition of ${\rm{Dom}}_{\max}(\laplcomp_{\Psi}^{F})$, it lies in $L^2(\Psi, F)$.
	 	\par 
	 	The following result is very well-known, see for example \cite[Proposition 2.3]{FinFinDiffer}. 
	 	\begin{prop}\label{prop_spec_discr}
	 		The spectrum of $\laplcomp_{\Psi}^{F}$ is discrete.
	 	\end{prop}
		By abuse of notation, in the following proposition and afterwards, we apply the trace operator (cf. \cite[Theorem 1.5.1.1]{Grisvard}) implicitly when we mention an integration over a smooth codimension $1$ submanifold $\Gamma \subset \Psi$ of a function from appropriate Sobolev space, see \cite[(2.16)]{FinFinDiffer}.
		\begin{prop}[{Green's identity, cf. \cite[Proposition 2.4]{FinFinDiffer}}]\label{prop_green_identity}
			For any open subset $U \subset \Psi$ with piecewise smooth boundary $\partial U$ not passing through ${\rm{Con}}(\Psi)$ and ${\rm{Ang}}(\Psi)$, and $u, v \in {\rm{Dom}}_{Fr}(\laplcomp_{\Psi}^{F})$:
			\begin{equation}\label{eq_green_identity}
				\scal{\laplcomp_{\Psi}^{F} u}{v}_{L^2(U, F)} = \scal{\nabla^F u}{\nabla^F v}_{L^2(U, F)} - \int_{\partial U} \nabla_n^{F} u  \cdot v dv_{\partial U},
			\end{equation}
			where $n$ is the outward normal to the boundary $\partial U$, and to simplify the notations, we omit the pointwise scalar product induced by $h^F$ in the last integral.
		\end{prop}
		By using Proposition \ref{prop_green_identity}, in \cite[Corollary 2.6]{FinFinDiffer}, we obtained 
	 	\begin{cor}\label{cor_kernel_const}
	 		There is an isomorphism between $\ker \laplcomp_{\Psi}^{F}$ and the space of flat sections of $F$, i.e.
	 		\begin{equation}
	 			\ker \laplcomp_{\Psi}^{F} \simeq H^0(\Psi, F).
	 		\end{equation}
	 	\end{cor}
	 	\par Consider now the heat operator $\exp(-t \laplcomp_{\Psi}^{F})$, $t > 0$. By Schwartz kernel theorem (cf. \cite[Proposition 2.14]{BGV}) and interior elliptic estimates, this operator has a smooth Schwartz kernel $\exp(-t \laplcomp_{\Psi}^{F})(x, y) \in \real$, defined for any $x, y \in \Psi^{\circ} := \Psi \setminus ({\rm{Con}}(\Psi) \cup{\rm{Ang}}(\Psi))$. This smooth kernel is also called the \textit{heat kernel}. By definition, for any $f \in L^{\infty}(\Psi)$, we have
	 	\begin{equation}\label{eq_defn_hk_mult_f}
	 		\tr{f  \cdot \exp(-t \laplcomp_{\Psi}^{F})} = \int_{\Psi^{\circ}} f(x) \cdot \exp(-t \laplcomp_{\Psi}^{F})(x, x) dv_{\Psi}(x).
	 	\end{equation}
	 	The proof of the following proposition uses standard techniques from the heat kernel analysis and explicit construction of the heat kernel on some model spaces. For the sake of making this paper self-contained, the proof of it is given in Appendix \ref{sect_hk_sm_t_exp}.
	 	\begin{prop}\label{prop_hk_expansion}
	 		For any $f \in \ccal^{\infty}(\overline{\Psi})$, which is constant in the neighborhood of ${\rm{Con}}(\Psi) \cup {\rm{Ang}}(\Psi)$, and any $k \in \nat$, as $t \to 0$, the following asymptotic expansion holds
	 		\begin{multline}
	 			\tr{f \cdot \exp(-t \laplcomp_{\Psi}^{F})}  
	 			= 
	 			\frac{\rk{F}}{4 \pi t} \int_{\Psi} f(x) dv_{\Psi}(x)
	 			+ 
	 			\frac{\rk{F}}{8 \sqrt{\pi} t^{1/2}}  \int_{\partial \Psi} f(x) dv_{\partial \Psi}(x)
	 			\\
	 			+
	 			\frac{\rk{F}}{12} \Big( \sum_{P \in {\rm{Con}}(\Psi)} \frac{4 \pi^2 - \angle(P)^2}{2 \pi \angle(P)} f(P)
	 			+
	 			\sum_{Q \in {\rm{Ang}}(\Psi)} \frac{\pi^2 - \angle(Q)^2}{2 \pi \angle(Q)} f(Q) \Big)
	 			\\
	 			+
	 			\frac{\rk{F}}{4 \pi \sqrt{t}} \sum_{i = 0}^{2k + 1} \frac{t^{i/2} \Gamma(\frac{i + 1}{2})}{i!} \int_{\partial \Psi} \frac{\partial^{2i} f(x)}{\partial n^{2i}} dv_{\partial \Psi}(x)
	 			+
	 			o(t^{k}),
	 		\end{multline}
	 		where $\Gamma$ is the $\Gamma$-function.
	 		In particular, for any $k \in \nat$, as $t \to 0$, we have 
	 		\begin{multline}
	 			\tr{\exp(-t \laplcomp_{\Psi}^{F})}  
	 			= 
	 			\frac{A(\Psi) \cdot \rk{F}}{4 \pi t}
	 			+ 
	 			\frac{|\partial \Psi| \cdot \rk{F}}{8 \sqrt{\pi} t^{1/2}} 
	 			\\
	 			+
	 			\frac{\rk{F}}{12} \Big( \sum_{P \in {\rm{Con}}(\Psi)} \frac{4 \pi^2 - \angle(P)^2}{2 \pi \angle(P)}
	 			+
	 			\sum_{Q \in {\rm{Ang}}(\Psi)} \frac{\pi^2 - \angle(Q)^2}{2 \pi \angle(Q)} \Big)
	 			+
	 			o(t^{k}).
	 		\end{multline}
	 	\end{prop}
	 	As a consequence of Karamata's theorem, cf. \cite[Theorem 2.42]{BGV}, and Proposition \ref{prop_hk_expansion}, by repeating the proof of \cite[Corollary 2.43]{BGV}, we get the following  
	 	\begin{cor}[Weyl's law]\label{cor_as_linear_growth}
	 		The number $N_{\Psi}^{F}(\lambda)$ of eigenvalues of $\laplcomp_{\Psi}^{F}$, smaller than $\lambda$, satisfies
	 		\begin{equation}
	 			N_{\Psi}^{F}(\lambda) \sim \frac{A(\Psi) \rk{F}}{4 \pi}  \lambda,
	 		\end{equation}
	 		as $\lambda \to + \infty$.
	 		In particular, for the $i$-th eigenvalue $\lambda_i$ of $\laplcomp_{\Psi}^{F}$, as $i \to \infty$, we have
	 		\begin{equation}\label{weyl_law}
	 			\lambda_i \sim \frac{4 \pi i}{A(\Psi) \rk{F}}.
	 		\end{equation}
	 	\end{cor}
	 	From Corollary \ref{cor_as_linear_growth} we deduce that the zeta function $\zeta_{\Psi}^{F}(s)$ from (\ref{defn_zeta}) is well-defined for $s \in \comp$, $\Re(s) > 1$. 
	 	Denote $\mu := \min\{ \spec(\laplcomp_{\Psi}^{F}) \setminus \{0\} \}$. 
	 	By Proposition \ref{prop_spec_discr} and Corollary \ref{cor_kernel_const}, we deduce that for some $C > 0$ and any $t > 1$, we have 
	 	\begin{multline}\label{eq_big_time_expansion}
	 		\Big| \tr{\exp(-t \laplcomp_{\Psi}^{F})} - \dim H^0(\Psi, F) \Big| 
	 		=
	 		\sum_{\lambda \in \spec(\laplcomp_{\Psi}^{F}) \setminus \{0\}} \exp(-t \lambda)
	 		\\
	 		\leq 
	 		\exp \Big(- \frac{t \mu}{2} \Big)
	 		\cdot
	 		\sum_{\lambda \in \spec(\laplcomp_{\Psi}^{F}) \setminus \{0\}} \exp(-\lambda/2)
	 		\leq
	 		C \exp \Big(- \frac{t \mu}{2} \Big).
	 	\end{multline}
	 	\par By Proposition \ref{prop_hk_expansion}, (\ref{eq_big_time_expansion}) and standard properties of the Mellin transform (cf. \cite[Lemma 9.35]{BGV}), we deduce that $\zeta_{\Psi}^{F}(s)$ extends meromorphically to the whole complex plane $\comp$, and $0$ is a holomorphic point of this extension. In particular, Definition \ref{defn_an_tors_zeta} makes sense.
	 	\begin{prop}\label{prop_zeta_zero_val}
	 		The zero value $\zeta_{\Psi}^{F}(0)$ can evaluated by the formula (\ref{eq_value_zeta_0}).
	 	\end{prop}
	 	\begin{proof}
	 		It follows from the classical properties of the Mellin transform and Proposition \ref{prop_hk_expansion}.
	 	\end{proof}
	 	\par For $c > 0$, we denote by $c \Psi$ the surface $\Psi$ endowed with the flat metric $c^2 \cdot g^{T \Psi}$. 
	 	The resulting surface is a flat surface with conical singularities.
	 	If $(\Psi, g^{T \Psi})$ can be embedded in $\comp$, then $c \Psi$ is just the homothety of $\Psi$. 
	 	\begin{prop}\label{prop_an_tors_rescaling}
	 		The analytic torsions of $c \Psi$ and $\Psi$ are related by
	 		\begin{equation}
	 			\log \det{}'(\laplcomp_{c \Psi}^{F}) = \log \det{}'(\laplcomp_{\Psi}^{F}) - 2 \log(c) \zeta_{\Psi}^{F}(0).
	 		\end{equation}
	 	\end{prop}
	 	\begin{proof}
	 		It follows trivially from (\ref{defn_zeta}) that the zeta-functions are related by
	 		\begin{equation}\label{eq_zeta_resc_rel}
	 			\zeta_{c \Psi}^{F}(s) = c^{2s} \zeta_{\Psi}^{F}(s).
	 		\end{equation}
	 		We conclude by Definition \ref{defn_an_tors_zeta} and (\ref{eq_zeta_resc_rel}).
	 	\end{proof}

	\subsection{Convergence of zeta-functions, a proof of Theorem \ref{thm_asympt_exansion}}\label{sect_idea_proof}
	The main goal here is to prove Theorem \ref{thm_asympt_exansion}. We conserve the notation from Theorem \ref{thm_asympt_exansion}. 
	\par 
	Let's describe the main steps of the proof before going into details.
	First, we recall that in our previous paper \cite[Theorem 1.1]{FinFinDiffer}, cf. Theorem \ref{thm_eigval_convergence}, we proved that the \textit{rescaled} spectrum 
	\begin{equation}
		\spec(n^2 \cdot \laplcomp_{\Psi_n}^{F_n}) = \{\lambda_1^{n}, \lambda_2^{n}, \cdots \},
	\end{equation}
	of discretization $\Psi_n$, ordered non-decreasingly, is asymptotically equal to the spectrum (\ref{eq_spec_lapl_defn}) of $\laplcomp_{\Psi}^{F}$. 
	In this article, in Theorem \ref{thm_unif_bnd_eig}, we  prove a uniform linear lower bound on the eigenvalues of $n^2 \cdot \laplcomp_{\Psi_n}^{F_n}$. By this and Theorem \ref{thm_eigval_convergence}, we prove in Corollary \ref{cor_zeta_conv_re1} that the discrete zeta-functions $\zeta_{\Psi_n}^{F_n}(s)$, associated to $n^2 \cdot \laplcomp_{\Psi_n}^{F_n}$, converge, as $n \to \infty$, to $\zeta_{\Psi}^{F}(s)$ of $\laplcomp_{\Psi}^{F}$ on the half-plane $\Re(s) > 1$.
	\par Then in Corollary \ref{cor_conv_zeta_ren}, we show that up to some \textit{local contributions}, depending on $n \in \nat^*$, this convergence extends to the whole complex plane $\comp$.
	To get the asymptotic of the logarithm of the determinant, now it only suffices to take the derivative of the zeta function at $0$.
	In our final step, Theorem \ref{thm_explicit_calc}, we study the asymptotic expansion of the derivative at $0$ of the introduced local contributions from Corollary \ref{cor_conv_zeta_ren}, as $n \to \infty$.
	\par More precisely, in our previous paper, we've shown
	\begin{thm}[{\cite[Theorem 1.1]{FinFinDiffer}}]\label{thm_eigval_convergence}
		For any $i \in \nat$, as $n \to \infty$, the following limit holds
		\begin{equation}\label{eq_eigval_convergence}
			\lambda_i^{n} \to \lambda_i.
		\end{equation}
	\end{thm}
	We also proved a similar statement for eigenvectors.  Before describing it, recall that in \cite[\S 3.2]{FinFinDiffer}, we have defined a “linearization" $L_n : {\rm{Map}}(V(\Psi_n), F_n) \to L^2(\Psi, F)$ functional.
	One should think of it as a sort of linear interpolation, which “blurs" the function near ${\rm{Con}}(\Psi) \cup {\rm{Ang}}(\Psi)$.
	In this article only the following properties of $L_n$ will be used.
	First, for eigenvectors $f_i^{n} \in {\rm{Map}}(V(\Psi_n), F_n)$ of $n^2 \cdot \laplcomp_{\Psi_n}^{F_n}$ corresponding to the eigenvalues $\lambda_i^{n}$, it satisfies
	\begin{prop}[{\cite[\S 3.2]{FinFinDiffer}}]\label{prop_cor_red_zeta_conv_aux_3}
		For any smooth function $\phi$, there is a constant $C > 0$, which depends only on $\norm{\phi}_{\ccal^1(\Psi)}$, such that for any $n \in \nat^*$, $f \in {\rm{Map}}(V(\Psi_n), F_n)$, we have
		\begin{equation}\label{eq_cor_red_zeta_conv_aux_3}
			\big\| L_n(\phi f) - \phi  L_n(f) \big\|_{L^{\infty}} \leq  \frac{C}{n} \big\| L_n(f) \big\|_{L^{\infty}}.
		\end{equation}
	\end{prop}
	\begin{prop}[{\cite[Proposition 3.8]{FinFinDiffer}}]\label{prop_ln_appr_easy}
			For any $\phi \in \ccal^{1}(\overline{\Psi})$ and $i, j \in \nat$ fixed, as $n \to \infty$:
			\begin{equation}\label{eq_ln_eig_appr_easy}
				\scal{\phi L_n(f_i^{n}) }{L_n(f_j^{n}) }_{L^2(\Psi, F)} = \frac{1}{n^2} \scal{\phi f_i^{n} }{f_j^{n}}_{L^2(\Psi_n, F_n)} + o(1).
			\end{equation}
	\end{prop}
	Now, let's finally state the analogue of Theorem \ref{thm_eigval_convergence} for the eigenvectors.
	Assume that $\lambda_i$, $i \in \nat^*$ has multiplicity $m_i$ in $\spec(\laplcomp_{\Psi}^{F})$.
	Let $f_{i, j} \in L^2(\Psi, F)$, $j = 1, \ldots, m_i$ be the orthonormal basis of eigenvectors of $\laplcomp_{\Psi}^{F}$ corresponding to $\lambda_i$.
	By Theorem \ref{thm_eigval_convergence}, we conclude that there is a series of eigenvalues $\lambda_{i, j}^{n}$, $j = 1, \ldots, m_i$ of $n^2 \cdot\laplcomp_{\Psi_n}^{F_n}$, converging to $\lambda_i$, as $n \to \infty$. 
	Moreover, no other eigenvalue of $n^2 \cdot\laplcomp_{\Psi_n}^{F_n}$, $n \in \nat^*$ comes close to $\lambda_i$ asymptotically.
	\begin{thm}[{\cite[Theorem 1.3]{FinFinDiffer}}]\label{thm_eigvec_convergence}
		For any $i \in \nat$, $n \in \nat^*$ there are $f_{i, j}^{n} \in {\rm{Map}}(V(\Psi_n), F_n)$,$1 \leq j \leq m_i$, which are pairwise orthogonal, satisfy $\| f_{i, j}^{n} \|_{L^2(\Psi_n, F_n)}^{2} = n^2$, and which are in the span of the eigenvectors of $n^2 \cdot\laplcomp_{\Psi_n}^{F_n}$, corresponding to the eigenvalues $\lambda_{i, j}^{n}$, $j = 1, \ldots, m_i$,   such that, as $n \to \infty$, in $L^2(\Psi, F)$, the following limit holds 
		\begin{equation}\label{eq_eigvec_convergence}
			L_n(f_{i, j}^{n}) \to f_{i, j}.
		\end{equation}
	\end{thm}
	\par Now, recall that we are trying to get the asymptotic expansion (\ref{eq_asympt_trees_exp}), which involves a product of terms, the number of which, $V(\Psi_n) - \dim H^0(\Psi, F)$, $n \in \nat^*$, tends to infinity quite quickly, as $n \to \infty$. Thus, there is almost no chance to get Theorem \ref{thm_asympt_exansion} by studying simply the convergence of \textit{individual} eigenvalues, as it would require much stronger convergence result compared to what we obtained in Theorem \ref{thm_eigval_convergence}. 
	Moreover, the analytic torsion is defined not through the renormalized product of the first eigenvalues, but through zeta-regularization procedure.
	\par We are, consequently, obliged to work with some statistics of eigenvalues. 
	The most important one in this paper is the corresponding zeta-function, defined for $s \in \comp$ by
	\begin{equation}\label{eq_discr_zeta}
		\zeta_{\Psi_n}^{F_n}(s) := \sum_{\lambda \in \spec(\laplcomp_{\Psi_n}^{F_n}) \setminus \{0\} } \frac{1}{(n^2 \cdot \lambda)^s} = \sum_{\lambda_i^{n} \neq 0} \frac{1}{(\lambda_i^{n})^s}.
	\end{equation}
	\par Remark that Theorem \ref{thm_asympt_exansion} would follow if we are able to prove that the sequence of functions $\zeta_{\Psi_n}^{F_n}(s)$, $n \in \nat^*$, converges to the function $\zeta_{\Psi}^{F}(s)$, as $n \to \infty$, for any $s \in \comp$ with all its derivatives. This is, of course, very optimistic (and false), as by Theorem \ref{thm_eigval_convergence} and (\ref{weyl_law}), we see that $\zeta_{\Psi_n}^{F_n}(s)$, $n \in \nat^*$ diverges for $s \leq 1$, as $n \to \infty$. Moreover, we know that the functions $\zeta_{\Psi_n}^{F_n}(s)$, $n \in \nat^*$, are holomorphic over $\comp$ and the function $\zeta_{\Psi}^{F}(s)$ is only meromorphic. 
	\par
	Thus, to get a convergence in a reasonable sense, it would be nice to restrict ourselves to a subspace of $\comp$, where no poles of $\zeta_{\Psi}^{F}(s)$ appear. 
	A very natural candidate for such subspace is $\{s \in \comp : \Re(s) > 1 \}$, as over this set, the formula (\ref{defn_zeta}) holds. 
	\par
	Theorem \ref{thm_eigval_convergence} shows that for fixed $k \in \nat$, the first $k$ terms of the discrete zeta function (\ref{eq_discr_zeta}) converge to the the corresponding terms of the continuous zeta function (\ref{defn_zeta}). This is why it is important to  bound the rest of the terms uniformly. 
	In Section \ref{sect_unif_weyl_law}, we prove
	\begin{thm}[Uniform weak Weyl's law]\label{thm_unif_bnd_eig}
		There is a constant $C > 0$ such that for any $n \in \nat^*$ and any $i \in \nat$, $i \leq V(\Psi_n)$, we have
		\begin{equation}\label{eq_unif_bnd_eig}
			\lambda_{i}^{n} \geq C i.
		\end{equation}
	\end{thm}
	By Theorems \ref{thm_eigval_convergence}, \ref{thm_unif_bnd_eig}, (\ref{defn_zeta}) and Weyl's law (\ref{weyl_law}), we get
	\begin{cor}\label{cor_zeta_conv_re1}
		For any $s \in \comp$, $\Re(s) > 1$, as $n \to \infty$, the following convergence holds
		\begin{equation}
			\zeta_{\Psi_n}^{F_n}(s) \to \zeta_{\Psi}^{F}(s).
		\end{equation}
	\end{cor}
	\par 
	Now, we try to nevertheless “extend" Corollary \ref{cor_zeta_conv_re1} to the whole complex plane. To overcome the fact that $\zeta_{\Psi}^{F}(s)$ is only a meromorphic function, we use the \textit{parametrix construction}. 
	\par In this step we were inspired by the approach used by Müller in his resolution of the Ray-Singer conjecture (now Cheeger-Müller theorem) in \cite{Mull78}.
	A simple but basic identity in this approach is
	\begin{equation}
		\zeta_{\Psi}^{F}(s) = \tr{(\laplcomp_{\Psi}^{F, \perp})^{-s}},
	\end{equation}	 
	where $(\laplcomp_{\Psi}^{F, \perp})^{-s}$ is a power of $\laplcomp_{\Psi}^{F}$, restricted to the vector space spanned by the eigenvectors corresponding to non-zero eigenvalues.
	\par 
	We consider a covering of $\Psi$ by a union of open sets $U_{\alpha}$, $\alpha \in I$, which are themselves pillowcase covers with tillable boundaries.
	Endow $U_{\alpha}$ with the restriction of the vector bundle $(F, h^F, \nabla^F)$, which we denote by the same symbol by the abuse of notation.
	We take a subordinate partition of unity $\phi_{\alpha}$, $\alpha \in I$ of $\Psi$ and another functions $\psi_{\alpha}$, $\alpha \in I$, satisfying 
	\begin{equation}\label{eq_phi_alpha_assumpt}
		\psi_{\alpha} = 1 \quad \text{over} \quad V_{\alpha} \supset {\rm{supp}}(\phi_{\alpha}) \qquad \text{and} \qquad {\rm{supp}}(\psi_{\alpha}) \subset U_{\alpha},
	\end{equation}
	for an open subset $V_{\alpha}$.
	Consider the \textit{renormalized zeta function}, defined for $s \in \comp$, $\Re(s) > 1$, by 
	\begin{equation}\label{eq_zeta_renorm}
		\zeta^{F, {\rm{ren}}}_{\Psi}(s) 
		:= 
		{\rm{Tr}} \Big[  (\laplcomp_{\Psi}^{F, \perp})^{-s} 
		- 
		\sum_{\alpha \in I} \phi_{\alpha} \cdot (\laplcomp_{U_{\alpha}}^{F, \perp})^{-s} \cdot \psi_{\alpha} \Big],
	\end{equation}
	where $(\laplcomp_{U_{\alpha}}^{F, \perp})^{-s}$ are defined analogously to $(\laplcomp_{\Psi}^{F, \perp})^{-s}$, and they are viewed as operators acting on $L^2(\Psi, F)$ by trivial extension.
	By Proposition \ref{prop_hk_expansion}, applied for $\Psi$ and $U_{\alpha}$, and the classical properties of the Mellin transform, we see that the following holds
	\begin{prop}\label{prop_zeta_ren_hol}
		The function $\zeta^{F, {\rm{ren}}}_{\Psi}(s)$ extends holomorphically to the whole complex plane $\comp$.
	\end{prop}
	\par Now, similarly, for $s \in \comp$, we construct the renormalized discrete zeta function
	\begin{equation}\label{eq_zeta_renorm_discr}
		\zeta^{F_n, {\rm{ren}}}_{\Psi_n}(s) := {\rm{Tr}} \Big[  (n^2 \cdot\laplcomp_{\Psi_n}^{F_n, \perp})^{-s} - \sum_{\alpha \in I} \phi_{\alpha} \cdot (n^2 \cdot  \laplcomp_{U_{\alpha, n}}^{F_n, \perp})^{-s} \cdot \psi_{\alpha} \Big],
	\end{equation}
	where the powers $(\laplcomp_{\Psi_n}^{F_n, \perp})^{-s}$, $(\laplcomp_{U_{\alpha, n}}^{F_n, \perp})^{-s}$ have to be understood as powers of the respective Laplacians, restricted to the vector spaces spanned by the eigenvectors corresponding to non-zero eigenvalues.
	The operators $(\laplcomp_{U_{\alpha, n}}^{F_n, \perp})^{-s}$ are viewed as operators on ${\rm{Map}}(V(\Psi_n), F_n)$ by the obvious inclusion $V(U_{\alpha, n}) \hookrightarrow V(\Psi_n)$, and $\phi_{\alpha}$ and $\psi_{\alpha}$ are given by pointwise multiplications on the elements of $V(U_{\alpha, n}) \hookrightarrow U_{\alpha}$.
	By using methods of Müller \cite{Mull78}, in Section \ref{sect_conv_whole_plane}, we prove the following 
	\begin{thm}\label{thm_bound_zeta}
		For any compact $K \subset \comp$, there is $C > 0$ such that for any $s \in K$, $n \in \nat^*$:
		\begin{equation}
			\big| \zeta^{F_n, {\rm{ren}}}_{\Psi_n}(s) \big| \leq C.
		\end{equation}
	\end{thm}
	Then, as in the proof of Corollary \ref{cor_zeta_conv_re1}, but using Theorem \ref{thm_eigvec_convergence}, in Section \ref{sect_unif_weyl_law}, we prove
	\begin{cor}\label{cor_red_zeta_conv_re1}
		For $s \in \comp$, $\Re(s) > 1$, as $n \to \infty$, the following convergence holds
		\begin{equation}
			{\rm{Tr}} \Big[  \phi_{\alpha} \cdot (n^2 \cdot  \laplcomp_{U_{\alpha, n}}^{F_n, \perp})^{-s} \cdot \psi_{\alpha} \Big] 
			\to 
			{\rm{Tr}} \Big[  \phi_{\alpha} \cdot (\laplcomp_{U_{\alpha}}^{F, \perp})^{-s} \cdot \psi_{\alpha} \Big].
		\end{equation}
	\end{cor}
	Recall a classical convergence result, stating that a sequence of uniformly locally bounded holomorphic functions converges on a connected domain if and only if it converges on some open subdomain.
	By this, Theorem \ref{thm_bound_zeta}, Proposition \ref{prop_zeta_ren_hol}  and  Corollaries \ref{cor_zeta_conv_re1}, \ref{cor_red_zeta_conv_re1},  we deduce
	\begin{cor}\label{cor_conv_zeta_ren}
		For any $s \in \comp$, as $n \to \infty$, the following limit holds
		\begin{equation}
			\zeta^{F_n, {\rm{ren}}}_{\Psi_n}(s) \to \zeta^{F, {\rm{ren}}}_{\Psi}(s).
		\end{equation}
	\end{cor}
	To conclude, on the whole complex plane, the discrete zeta functions $\zeta_{\Psi_n}^{F_n}$, $n \in \nat^*$, do not converge to the continuous zeta function $\zeta_{\Psi}^{F}$, as $n \to \infty$. However, we proved that by taking out local terms from $\zeta_{\Psi_n}^{F_n}$, the continuous analogues of which “produce non-holomorphicity" of the function $\zeta_{\Psi}^{F}(s)$, the convergence holds on the whole complex plane.
	From now on and almost till the end of this section we will essentially focus on understanding those local terms.
	\par From Proposition \ref{prop_hk_expansion} and the standard properties of the Mellin transform, similarly to Definition \ref{defn_an_tors}, we see that the following quantity is well-defined
	\begin{equation}
		\tr{\phi \cdot \log(\laplcomp_{\Psi}^{F, \perp})} := -\frac{\partial}{\partial s} \tr{\phi \cdot (\laplcomp_{\Psi}^{F, \perp})^{-s}} \Big|_{s = 0}.
	\end{equation}
	By Cauchy formula and Corollary \ref{cor_conv_zeta_ren}, we see that, as $n \to \infty$, we have
	\begin{multline}\label{eq_diff_log_det_first}
		\log \big(\det(n^2 \cdot\laplcomp_{\Psi_n}^{F_n, \perp})\big)
		-
		\sum_{\alpha \in I} \tr{\phi_{\alpha} \cdot \log \big(n^2 \cdot \laplcomp_{U_{\alpha, n}}^{F_n, \perp}\big)}
		\\
		\to
		\log \big( \det{}'\laplcomp_{\Psi}^{F} \big)
		-
		\sum_{\alpha \in I} \tr{\phi_{\alpha} \cdot \log \big(\laplcomp_{U_{\alpha}}^{F, \perp}\big)}.
	\end{multline}
	Remark that in (\ref{eq_diff_log_det_first}), for the first time in our analysis we see the analytic torsion.
	\par 
	Now, in our final step we choose $U_{\alpha}$ and $\phi_{\alpha}$ as in Section \ref{sect_decomp}, so that the terms which appear in the sum of the left-hand side of (\ref{eq_diff_log_det_first}) become relatively easy to handle. 
	We use the notations from Section \ref{sect_decomp}.
	As all the elements of the cover are contractible and the vector bundle $(F, h^F, \nabla^F)$ is flat unitary, by making a gauge transformation, in the calculations over $U_{\alpha}$ we may and we will suppose that  $(F, h^F, \nabla^F)$ restricts to a trivial vector bundle over the coverings.
	\par
	Recall that $c_0(\Psi)$ and $c_1(\Psi)$ were defined in (\ref{eq_number_sq_count}).
	If we specialize (\ref{eq_diff_log_det_first}) for the covering from Section \ref{sect_decomp}, and we use (\ref{eq_number_sq_count}), we see that, as $n \to \infty$, we have
	\begin{equation}\label{eq_diff_log_det_secon2}
	\begin{aligned}
		& \log \big(\det(n^2 \cdot\laplcomp_{\Psi_n}^{F_n, \perp})\big)
		-
		\rk{F}
		\cdot
		\Big(
		N_{sq, 0}(\Psi) \tr{\phi_{sq, 0} \cdot \log (n^2 \cdot \laplcomp_{U_{sq, n}}^{\perp}) }
		\\
		&
		\, \,
		-
		N_{sq, 1}(\Psi) \tr{\phi_{sq, 1} \cdot \log (n^2 \cdot \laplcomp_{U_{sq, n}}^{\perp}) }
		-
		N_{sq, 2}(\Psi) \tr{\phi_{sq, 2} \cdot \log (n^2 \cdot \laplcomp_{U_{sq, n}}^{\perp}) }
		\\
		&
		\, \,
		-
		\sum_{\alpha \in \angle({\rm{Con}}(\Psi))}
		\tr{\phi_{\mathcal{C}_{\alpha}} \cdot \log (n^2 \cdot \laplcomp_{U_{\mathcal{C}_{\alpha}, n}}^{\perp})  }
		-
		\sum_{\beta \in \angle({\rm{Ang}}^{\neq \pi/2}(\Psi))}
		\tr{\phi_{\mathcal{A}_{\beta}} \cdot \log (n^2 \cdot \laplcomp_{U_{\mathcal{A}_{\beta}, n}}^{\perp})  }
		\Big)
		\\
		&
		\qquad \qquad \qquad \qquad \qquad 
		\to
		\log \big( \det{}'\laplcomp_{\Psi}^{F} \big)
		-
		\rk{F} \cdot \Big(
		A(\Psi)C_0 
		+
		|\partial \Psi| C_1 
		+
		C_2(\Psi)
		\Big),
	\end{aligned}
	\end{equation}
	where $C_0, C_1, C_2(\Psi) \in \real$ are defined as follows
	\begin{equation}\label{eq_c0c1c2defn}
	\begin{aligned}
		&
		C_0 =  \tr{\phi_{sq, 0} \cdot \log ( \laplcomp_{U_{sq}}^{\perp}) },
		\\
		&
		C_1 = -\frac{3}{2}C_0 +  \tr{\phi_{sq, 1} \cdot \log ( \laplcomp_{U_{sq}}^{\perp}) },
		\\
		&
		C_2(\Psi) =  c_0(\Psi) C_0 + c_1(\Psi) C_1
		+
		N_{sq, 2}(\Psi) \tr{\phi_{sq, 2} \cdot \log (\laplcomp_{U_{sq}}^{\perp}) }
		\\
		&
		\qquad \qquad 
		+
		\sum_{\alpha \in \angle({\rm{Con}}(\Psi))}
		\tr{\phi_{\mathcal{C}_{\alpha}} \cdot \log (\laplcomp_{U_{\mathcal{C}_{\alpha}}}^{\perp})  }
		+
		\sum_{\beta \in \angle({\rm{Ang}}^{\neq \pi/2}(\Psi))}
		\tr{\phi_{\mathcal{A}_{\beta}} \cdot \log (\laplcomp_{U_{\mathcal{A}_{\beta}}}^{\perp})  }.
	\end{aligned}
	\end{equation}
	We remark that by (\ref{eq_number_sq_count}) and (\ref{eq_c0c1c2defn}), the constants $C_0, C_1$ are universal, and the constant $C_2(\Psi)$ depends purely on the sets $\angle({\rm{Ang}}(\Psi))$ and $\angle({\rm{Con}}(\Psi))$.
	\par 
	Next theorem studies the asymptotic expansion of the terms in the left-hand side of (\ref{eq_diff_log_det_secon2}), which correspond to squares. 
	By using Fourier analysis on $n \times n$ mesh, in Section \ref{sect_four_an_sq}, we prove
	\begin{thm}\label{thm_explicit_calc}
		Suppose $U$ is a square with a pillowcase structure, containing $4$ euclidean squares of area $1$. 
		Let $U_n$ be a sequence of graphs, constructed from $U$ as in Theorem \ref{thm_asympt_exansion}.
		Consider a map $\pi_U : U \to \tilde{U}$, from $U$ to a torus $\tilde{U}$, obtained by the identification of the opposite sides of $U$.
		Suppose that the function $\phi : U \to \comp$ is the pull-back of a smooth function on $\tilde{U}$. 
		Then there is a constant $c(U, \phi) \in \real$ such that, as $n \to \infty$, we have
		\begin{multline}\label{eq_thm_explicit_calc}
			\tr{ \phi \cdot \log \big( n^2\cdot \laplcomp_{U_{n}}^{\perp} \big)}
			= 
			\Big(
			8 n^2 \log(n) 
			+
			\frac{4Gn^2}{\pi}
			\Big)			
			 \int_{U} \phi dv_{U}
			\\
			+ 
			\frac{\log(\sqrt{2} - 1)}{2}
			\cdot
			 n  
			 \cdot
			 \int_{\partial U} \phi dv_{\partial U}
			-
			\frac{\log(n)}{2} \phi(P)
			-
			c(U, \phi) + o(1),
		\end{multline}
		where $P$ is some corner vertex of $U$ and $G$ is the Catalan's constant, see (\ref{eq_defn_cat}).
	\end{thm}
	Now, we plug Theorem \ref{thm_explicit_calc} into (\ref{eq_diff_log_det_secon2}) to see that what we get is more or less what we need to finish the proof of Theorem \ref{thm_asympt_exansion}.
	The local contributions from the conical points and the corners of the boundary in (\ref{eq_diff_log_det_first}) would constitute the sequence $CA_n$, $n \in \nat^*$ from Theorem \ref{thm_asympt_exansion}.
	The value $\log(2)$ in (\ref{eq_asympt_trees_exp}) comes from Duplantier-David calculations (cf. Remark \ref{rem_main_thm}a)) and  the fact that the undetermined constants, obtained in the course of the proof (as $c(U, \phi)$ from (\ref{eq_thm_explicit_calc})) are universal. 
	\par 
	More precisely, for simplicity of the notation, we denote 
	\begin{equation}
		\phi_{ca} :=  \sum_{\alpha \in I'} \phi_{\alpha, 0},
	\end{equation}
	where $\alpha \in I'$ if and only if $U_{\alpha}$ corresponds to a square.
	\par 
	We cannot apply Theorem \ref{thm_explicit_calc} directly to the terms corresponding to the squares on the left-hand side of (\ref{eq_diff_log_det_secon2}), as functions $\phi_{sq, 1}$ and $\phi_{sq, 2}$ do not satisfy the assumption of Theorem \ref{thm_explicit_calc}. But if one replaces them by their average with respect to the symmetry group of the square, they would satisfy this assumption by (\ref{eq_rho1_rest1}), (\ref{eq_psi_sq_1}), (\ref{eq_psi_sq_2}). The trace is unchanged under this procedure.
	\par 
	By this, Theorem \ref{thm_explicit_calc} and (\ref{eq_diff_log_det_secon2}), we see that,as $n \to \infty$, the following holds
	\begin{equation}\label{eq_log_n2_det_asympt}
	\begin{aligned}
		&
		\log \big(\det(n^2 \cdot\laplcomp_{\Psi_n}^{F_n, \perp})\big)
		-
		\rk{F} \cdot 
		\Big(
		\sum_{\alpha \in \angle({\rm{Con}}(\Psi))}
		\tr{\phi_{\mathcal{C}_{\alpha}} \cdot \log (n^2 \cdot \laplcomp_{U_{\mathcal{C}_{\alpha}, n}}^{\perp})  }
		\\
		&
		\qquad \qquad \qquad \qquad \qquad \qquad  \qquad
		-
		\sum_{\beta \in \angle({\rm{Ang}}^{\neq \pi/2}(\Psi))}
		\tr{\phi_{\mathcal{A}_{\beta}} \cdot \log (n^2 \cdot \laplcomp_{U_{\mathcal{A}_{\beta}, n}}^{\perp})  }
		\Big)
		\\
		&
		=
		\rk{F} \cdot
		\Big(
			8 n^2 \log(n)
			+
			\frac{4Gn^2}{\pi}
		\Big) \int_{U} \phi_{ca} dv_{U}
		+ 
		\rk{F} \cdot \frac{\log(\sqrt{2} - 1)}{2}  n  \int_{\partial U} \phi_{ca} dv_{\partial U}
		\\
		&
		\qquad \qquad \qquad \qquad 
		-
		\rk{F} \cdot
		\frac{\log(n)}{8} N_{sq, 2}(\Psi)
		+
		\log \big( \det{}'\laplcomp_{\Psi}^{F} \big)
		\\
		&
		\qquad \qquad \qquad \qquad \qquad \qquad \qquad 
		-
		\rk{F}
		\cdot
		\Big(
		A(\Psi) C_0' 
		+
		|\partial \Psi| C_1' 
		+
		C_2'(\Psi)
		\Big)
		+ 
		o(1),
	\end{aligned}
	\end{equation}
	where $C_0', C_1', C_2'(\Psi) \in \real$ are defined by
	\begin{equation}\label{eq_defn_c2pr}
	\begin{aligned}
		&
		C_0' := C_0 + c(U_{sq}, \phi_{sq, 0}),
		\\
		&
		C_1' := C_1 - \frac{3}{2} c(U_{sq}, \phi_{sq, 0}) +  c(U_{sq}, \phi_{sq, 1}),
		\\
		&
		C_2'(\Psi) := C_2(\Psi) + c_0(\Psi) c(U_{sq}, \phi_{sq, 0}) +  c_1(\Psi) c(U_{sq}, \phi_{sq, 1}) + N_{sq, 2}(\Psi) c(U_{sq}, \phi_{sq, 2}).
	\end{aligned}
	\end{equation}
	Clearly, by (\ref{eq_c0c1c2defn}) and (\ref{eq_defn_c2pr}), the constants $C_0'$, $C_1'$ are universal and the constant $C_2'(\Psi)$ depends purely on the sets $\angle({\rm{Ang}}(\Psi))$ and $\angle({\rm{Con}}(\Psi))$.
	\par 
	Now we take a difference between (\ref{eq_log_n2_det_asympt}) and (\ref{eq_log_n2_det_asympt}), applied for $\mathcal{C}_{\alpha}$ and $\mathcal{A}_{\beta}$ for each $\alpha \in \angle({\rm{Con}}(\Psi))$, $\beta \in \angle({\rm{Ang}}(\Psi))$, $\beta \neq \frac{\pi}{2}$, to see that, as $n \to \infty$, we have
	\begin{equation}\label{eq_log_n2_det_asympt2}
	\begin{aligned}
		\log & \big(\det(n^2 \cdot\laplcomp_{\Psi_n}^{F_n, \perp})\big)
		-
		\rk{F}
		\cdot
		\Big(
		\sum_{\alpha \in \angle({\rm{Con}}(\Psi))}
		\log \big(\det(n^2 \cdot  \laplcomp_{U_{\mathcal{C}_{\alpha}, n}}^{\perp})\big)
		\\
		&
		\qquad \qquad \qquad \qquad  \qquad \qquad \qquad \qquad 
		-
		\sum_{\beta \in \angle({\rm{Ang}}^{\neq \pi/2}(\Psi))}
		\log \big(\det(n^2 \cdot \laplcomp_{U_{\mathcal{A}_{\beta}, n}}^{\perp})\big)
		\Big)
		\\
		&
		=
		\rk{F}
		\cdot
		\Big( 8 n^2 \log(n) + \frac{4Gn^2}{\pi} \Big) \cdot 
		\\
		&
		\qquad \qquad \qquad \qquad \qquad 
		\cdot
		\Big( A(\Psi) -
			\sum_{\alpha \in \angle({\rm{Con}}(\Psi))}
			A(\mathcal{C}_{\alpha})
			-
			\sum_{\beta \in \angle({\rm{Ang}}^{\neq \pi/2}(\Psi))}
			A(\mathcal{A}_{\beta})
		\Big)
		\\
		&
		\phantom{= \,}
		+ 
		\rk{F}
		\cdot
		\frac{\log(\sqrt{2} - 1)}{2}
		n  
		\Big( | \partial \Psi | -
			\sum_{\alpha \in \angle({\rm{Con}}(\Psi))}
			|\partial  \mathcal{C}_{\alpha}|
			-
			\sum_{\beta \in \angle({\rm{Ang}}^{\neq \pi/2}(\Psi))}
			|\partial  \mathcal{A}_{\beta}|
		\Big)
		\\
		&
		\phantom{= \,}
		-
		\rk{F}
		\cdot
		\frac{\log(n)}{8} 
		\Big( N_{sq, 2}(\Psi) -
			\sum_{\alpha \in \angle({\rm{Con}}(\Psi))}
			N_{sq, 2}(\mathcal{C}_{\alpha})
			-
			\sum_{\beta \in \angle({\rm{Ang}}^{\neq \pi/2}(\Psi))}
			N_{sq, 2}(\mathcal{A}_{\beta})
		\Big)
		\\
		&
		\phantom{= \,}
		+ 
		\log \big( \det{}'\laplcomp_{\Psi}^{F} \big)
		-
		\rk{F}
		\cdot
		\Big(
		A(\Psi)C_0' 
		+
		|\partial \Psi| C_1' 
		+
		C_2''(\Psi)
		\Big)
		 + o(1),
	\end{aligned}
	\end{equation}
	where $C_2''(\Psi) \in \real$ is defined as follows
	\begin{multline}\label{eq_defn_c2prpr}
		C_2''(\Psi) 
		:= 
		C_2'(\Psi)
		-
		\sum_{\alpha \in \angle({\rm{Con}}(\Psi))}
		\Big( 
			C_2'(U_{\mathcal{C}_{\alpha}})		
			+
			A(U_{\mathcal{C}_{\alpha}})C_0'
			+
			|\partial U_{\mathcal{C}_{\alpha}}| C_1'
			-
			\log ( \det{}' \laplcomp_{\mathcal{C}_{\alpha}})
		\Big)
		\\
		-
		\sum_{\beta \in \angle({\rm{Ang}}^{\neq \pi/2}(\Psi))}
		\Big(
			C_2'(U_{\mathcal{A}_{\beta}})
			+
			A(U_{\mathcal{A}_{\beta}})C_0'
			+
			|\partial U_{\mathcal{A}_{\beta}}| C_1'
			-
			\log ( \det{}' \laplcomp_{\mathcal{A}_{\beta}})
		\Big).
	\end{multline}
	Clearly, by (\ref{eq_defn_c2pr}) and (\ref{eq_defn_c2prpr}), the constant $C_2''(\Psi)$ depends only on the sets $\angle({\rm{Ang}}(\Psi))$ and $\angle({\rm{Con}}(\Psi))$.
	\par
	Note that by Corollary \ref{cor_kernel_const}, we have
	\begin{equation}\label{eq_log_det_log_n1}
		\log \big(\det(n^2 \cdot\laplcomp_{\Psi_n}^{F_n, \perp})\big)
		=
		\log \big(\det{}'\laplcomp_{\Psi_n}^{F_n} \big)
		+
		2 \log(n) \cdot \Big(
			\rk{F} A(\Psi)n^2 - \dim H^0(\Psi, F)
		\Big).
	\end{equation}
	Also remark that by Proposition \ref{prop_zeta_zero_val}, the following identity holds
	\begin{equation}\label{eq_fin_aux_10000}
	\begin{aligned}
		&
		\Big( 2 \zeta_{\Psi}^{F}(0) + 2 \dim H^0(\Psi, F) \Big)
		-
		\rk{F} \cdot
		\Big(
		\sum_{\alpha \in \angle({\rm{Con}}(\Psi))}
		\Big( 2 \zeta_{U_{\mathcal{C}_{\alpha}}}(0) + 2 \Big)
		+
		\\
		&
		\qquad \qquad \qquad \qquad \qquad \qquad  \qquad \qquad  \qquad \qquad
		\sum_{\beta \in \angle({\rm{Ang}}^{\neq \pi/2}(\Psi))}
		\Big( 2 \zeta_{U_{\mathcal{A}_{\beta}}}(0) + 2 \Big)
		\Big)
		\\
		&
		\qquad \qquad
		=
		\frac{\rk{F}}{8} 
		\Big( N_{sq, 2}(\Psi) -
			\sum_{\alpha \in \angle({\rm{Con}}(\Psi))}
			N_{sq, 2}(\mathcal{C}_{\alpha})
			-
			\sum_{\beta \in \angle({\rm{Ang}}^{\neq \pi/2}(\Psi))}
			N_{sq, 2}(\mathcal{A}_{\beta})
		\Big).
	\end{aligned}
	\end{equation}
	Recall that the quantity $\log^{{\rm{ren}}} (\det{}'\laplcomp_{\Psi_n}^{F_n} )$, $n \in \nat^*$ was defined in (\ref{eq_ren_number_sp_tree}). 
	By (\ref{eq_log_n2_det_asympt2}), (\ref{eq_log_det_log_n1}) and (\ref{eq_fin_aux_10000}), we see that, as $n \to \infty$, the following limit holds
	\begin{equation}\label{eq_as_exp_log_ren_pf_1}
	\begin{aligned}
		&
		\log^{{\rm{ren}}} \big(\det{}'\laplcomp_{\Psi_n}^{F_n} \big)
		 - 
		 \rk{F}
		 \cdot
		 \Big(
		 \sum_{\alpha \in \angle({\rm{Con}}(\Psi))} \log^{{\rm{ren}}} \big(\det{}' \laplcomp_{\mathcal{C}_{\alpha, n}} \big)
		 \\
		 &
		 \qquad \qquad \qquad \qquad \qquad \qquad  \qquad \qquad \qquad
		 - 
		 \sum_{\beta \in \angle({\rm{Ang}}^{\neq \pi/2}(\Psi))} \log^{{\rm{ren}}} \big(\det{}' \laplcomp_{\mathcal{A}_{\beta, n}} \big) 
		 \Big)
		 \\
		 &
		  \qquad \qquad \qquad \qquad \qquad
		 \to
		 \log \det{}'(\laplcomp_{\Psi}^{F})
		 -
		 \rk{F}
		 \cdot
		 \Big(
			 A(\Psi) C_0'
			 +
			 |\partial \Psi| C_1'
			 +
			 C_2''(\Psi)
		 \Big).
	\end{aligned}
	\end{equation}
	Clearly, if we could prove that $C_0', C_1' = 0$, and that we can write 
	\begin{equation}\label{eq_c2_psi}
		C_2''(\Psi) = \frac{\log(2) \cdot \# {\rm{Ang}}^{= \pi/2}(\Psi)}{16} + C'''(\Psi),
	\end{equation}
	where $C'''(\Psi)$ depends only on the sets $\angle({\rm{Ang}}^{\neq \pi/2}(\Psi))$ and $\angle({\rm{Con}}(\Psi))$, then (\ref{eq_as_exp_log_ren_pf_1}) would imply (\ref{eq_asympt_trees_exp}) for $CA_n(\Psi)$, $n \in \nat^*$, defined as follows
	\begin{equation}\label{eq_ca_n_full_formula}
		CA_n(\Psi) :=  \sum_{\alpha \in \angle({\rm{Con}}(\Psi))} \log^{{\rm{ren}}} \big(\det{}' \laplcomp_{\mathcal{C}_{\alpha, n}} \big)
		 +
		 \sum_{\alpha \in \angle({\rm{Ang}}^{\neq \pi/2}(\Psi))} \log^{{\rm{ren}}} \big(\det{}' \laplcomp_{\mathcal{A}_{\beta, n}} \big) 
		 -
		 C'''(\Psi),
	\end{equation}
	as $CA_n(\Psi)$, defined by (\ref{eq_ca_n_full_formula}), depends only on the sets $\angle({\rm{Ang}}^{\neq \pi/2}(\Psi))$ and $\angle({\rm{Con}}(\Psi))$ by (\ref{eq_defn_c2prpr}).
	\par 
	Let's prove that $C_0', C_1' = 0$.
	Recall that the surface $c \Psi$, $c > 0$, was defined before Proposition \ref{prop_an_tors_rescaling}.
	Clearly, for $c \in \nat^*$, the surface $c \Psi$ has a natural structure of a pillowcase cover with tillable boundary, coming from $\Psi$.
	The number of tiles of $c \Psi$ is $c^2$ times the number of tiles of $\Psi$.
	From (\ref{eq_as_exp_log_ren_pf_1}), applied for $\Psi$ and $c \Psi$, by (\ref{eq_ren_number_sp_tree}), we see that, for any $c \in \nat^*$, as $n \to \infty$, we have
	\begin{multline}\label{eq_as_exp_log_ren_pf_2}
		\log^{{\rm{ren}}} \big(\det{}' \laplcomp_{(c\Psi)_n}^{F_n} \big)
		 - 
		\log^{{\rm{ren}}} \big(\det{}'\laplcomp_{\Psi_n}^{F_n} \big)
		\to
		\log \big( \det{}' \laplcomp_{c\Psi}^{F} \big)
		-
		\log \big( \det{}' \laplcomp_{\Psi}^{F} \big)
		\\
		+
		\rk{F}
		\cdot
		\Big(
		(c^2 - 1)A(\Psi) C_0
		+
		(c-1)|\partial \Psi| C_1
		\Big).
	\end{multline}
	Remark that by the construction of the discretization $\Psi_n$ from Section \ref{sect_sq_t_s_disc}, for any $c \in \nat^*$, we have the following isomorphism of graphs 
	\begin{equation}\label{eq_gr_isom}
		\Psi_{cn} \simeq (c\Psi)_n.
	\end{equation}
	By (\ref{eq_ren_number_sp_tree}) and (\ref{eq_gr_isom}), we conclude that for any $c \in \nat^*$, $n \in \nat^*$, we have
	\begin{equation}\label{eq_as_exp_log_ren_pf_3}
		\log^{{\rm{ren}}} \big(\det{}' \laplcomp_{(c\Psi)_n}^{F_n} \big)
		=
		\log^{{\rm{ren}}} \big(\det{}' \laplcomp_{\Psi_{cn}}^{F_{cn}} \big)
		-
		2 \zeta_{\Psi}^{F}(0) \log(c).
	\end{equation}
	From Proposition \ref{prop_an_tors_rescaling}, (\ref{eq_as_exp_log_ren_pf_2}) and (\ref{eq_as_exp_log_ren_pf_3}), we deduce that for any $c \in \nat^*$, we have
	\begin{equation}\label{eq_fin_asympt_easy}
		\log^{{\rm{ren}}} \big(\det{}' \laplcomp_{\Psi_{cn}}^{F_{cn}} \big)
		 - 
		\log^{{\rm{ren}}} \big(\det{}' \laplcomp_{\Psi_{n}}^{F_n} \big)
		\to
		\rk{F}
		\cdot
		\Big(
			(c^2 - 1)A(\Psi) C_0'
			+
			(c-1)|\partial \Psi| C_1'
		\Big)
		.
	\end{equation}
	From (\ref{eq_fin_asympt_easy}), from the fact that the constants $C_0', C_1'$ are universal and from explicit calculations for rectangles due to Duplantier-David \cite{DuplDav}, (cf. Remark \ref{rem_main_thm}a)), we conclude that $C_0', C_1' = 0$.
	\par 
	Now, again by Remark \ref{rem_main_thm}a) and the fact that $C_2''(\Psi)$ depends additively on the sets $\angle({\rm{Ang}}(\Psi))$ and $\angle({\rm{Con}}(\Psi))$, we conclude that we have (\ref{eq_c2_psi}).
	\par 
	Now, since $C_0', C_1' = 0$, the limit (\ref{eq_fin_asympt_easy}) implies (\ref{eq_log_n_ren_logn}).
	The proof of Theorem \ref{thm_asympt_exansion} is finished.
	
\subsection{Loop measure induced by CRSF's, proofs of Theorems \ref{thm_exp_value}, \ref{thm_meas_cyl_ev}}\label{sect_cyl_event}
	The goal of this section is to prove Theorem \ref{thm_meas_cyl_ev}. As we already explained in Introduction, our proof is a combination of the results of Kassel-Kenyon \cite{KassKenyon}, Basok-Chelkak \cite{BasokChelk} and Theorem \ref{thm_rel_as_compexity}.
	\par 
	We use the notation from Theorem \ref{thm_meas_cyl_ev}. Remark that the results of \cite{KassKenyon} work for so-called conformal approximations $\Psi_n$ of a surface $\Psi$, see \cite[\S 4.1]{KassKenyon}.
	According to \cite[end of \S 4.1]{KassKenyon}, those approximations are more general than the one considered in this article.
	So the results of \cite{KassKenyon} can be applied to $\Psi_n$ and $(\Psi, g^{T \Psi})$ as in Theorem \ref{thm_asympt_exansion}.
	From now on, we fix $(\Psi, g^{T \Psi})$, $(F, h^F, \nabla^F)$ and $\Psi_n$, $(F_n, h^{F_n}, \nabla^{F_n})$ as in Theorem \ref{thm_asympt_exansion}.
	\begin{thm}[{Kassel-Kenyon \cite[Theorem 15, Lemma 16]{KassKenyon}}]\label{thm_kass_ken_prob}
		For any $n \in \nat^*$ and any lamination $L$ on $\overline{\Psi}$, the proportion of $\rm{CRSF}_{{\rm{nonc}}}(\Psi_n)$, which induce $L$ by their loops, is given by 
		\begin{equation}\label{eq_meas_cyl_ev_fin_gr}
			\mu_{\rm{nonc}}(E_L(\Psi_n)) = \sum_{L \preceq L'} A_{L', L} \cdot \int_{\mathcal{M}} \frac{\sqrt{\smash[b]{\det{}' \laplcomp_{\Psi_n}^{F_n}}}}{\# \rm{CRSF}_{{\rm{nonc}}}(\Psi_n)} \cdot P_{L'}(F) d \nu.
		\end{equation}
		Remark that since $\Psi_n$ is finite, the sum on the right-hand side of (\ref{eq_meas_cyl_ev_fin_gr}) is finite as well.
	\end{thm}
	\begin{thm}[{Kassel-Kenyon \cite[Theorem 17]{KassKenyon}}]\label{thm_kas_kenyon_conv}
		There is a function $G \in L^{\infty}(\mathcal{M})$ depending only on the conformal type of the surface $\overline{\Psi}$, such that for any non-trivial flat unitary vector bundle $(F, h^F, \nabla^F)$ of rank $2$, we have
		\begin{equation}\label{eq_kas_kenyon_conv2}
			\lim_{n \to \infty}
			\frac{\sqrt{\smash[b]{\det{}' \laplcomp_{\Psi_n}^{F_n}}}}{\# \rm{CRSF}_{{\rm{nonc}}}(\Psi_n)} 
			=
			G(F).
		\end{equation}
	\end{thm}
	\begin{proof}[Proof of Theorems  \ref{thm_exp_value}, \ref{thm_meas_cyl_ev}.]
		From Theorems \ref{thm_asympt_exansion} and \ref{thm_kas_kenyon_conv}, there is a constant $Z(\Psi)$, such that
		\begin{equation}\label{eq_kass_ken_fin_1}
			G(F) = \frac{\sqrt{\smash[b]{\det{}' \laplcomp_{\Psi}^{F}}}}{Z(\Psi)}.
		\end{equation}
		This finishes the proof of Theorem \ref{thm_exp_value} by Theorem \ref{thm_kas_kenyon_conv} and Kenyon's formula (\ref{eq_crsf_1}).
		\par 
		The function $P_{L'}(F)$ is bounded on $\mathcal{M}$ by definition, see (\ref{eq_defn_tl}), (\ref{eq_allprim_defn}). By this, the boundness of $G$ from Theorem \ref{thm_kas_kenyon_conv}, and the compactness of $\mathcal{M}$, we conclude
		\begin{equation}\label{eq_kas_ken_int}
			\int_{\mathcal{M}} G(F) \cdot P_{L'}(F) dv < + \infty.
		\end{equation}
		\par 
		Basok-Chelkak proved in \cite[Theorem 4.9]{BasokChelk} that the coefficients $A_{L', L}$ grow at most exponentially in terms of the complexity $n(L')$ of the lamination $L'$. 
		Remark that they prove this fact for the coefficients of the change of the basis between Gram-Schmidt orthonormalization of $S_L(F)$ from (\ref{eq_defn_sl}) and $S_L(F)$. However, the result still holds for $A_{L', L}$ defined by (\ref{eq_allprim_defn}), since the matrix of the change of the basis between $S_L(F)$ and $T_L(F)$ is lower triangular.
		By assuming this fact, Kassel-Kenyon in \cite[proof of Theorem 18]{KassKenyon} proved that for any lamination $L$ on $\overline{\Psi}$, the sum 
			\begin{equation}
			\sum_{L \preceq L'} A_{L', L} \cdot \int_{\mathcal{M}} G(F)
			 \cdot P_{L'}(F) d \nu,
		\end{equation}
		is convergent.
		By this and Theorems \ref{thm_kass_ken_prob}, \ref{thm_kas_kenyon_conv}, Kassel-Kenyon in \cite[proof of Theorem 18]{KassKenyon} concluded
		\begin{equation}\label{eq_kass_ken_fin_2}
			\lim_{n \to \infty} \mu_{\rm{nonc}}(E_L(\Psi_n)) = \sum_{L \preceq L'} A_{L', L} \cdot \int_{\mathcal{M}}  G(F) \cdot P_{L'}(F) d \nu.
		\end{equation}
		Theorem \ref{thm_meas_cyl_ev} now follows from (\ref{eq_kass_ken_fin_1}) and (\ref{eq_kass_ken_fin_2}).
	\end{proof}

\section{Zeta functions: discrete and continuous}\label{sect_main_zeta}
In this section we prove some bounds on the spectrum of graph approximations of a surface and on the respective zeta functions.
More precisely, this section is organized as follows. In Section \ref{sect_unif_weyl_law} we prove Theorem \ref{thm_unif_bnd_eig}, which gives a uniform linear lower bound on the eigenvalues of the discrete Laplacian.
As a consequence, we obtain Corollary \ref{cor_red_zeta_conv_re1}.
In Section \ref{sect_conv_whole_plane}, modulo some statements about the uniform bound on the powers of the discrete Laplacian, we prove Theorem \ref{thm_bound_zeta}.
Finally, in Section \ref{sect_unf_powrs_lapl}, we prove those left-out statements from Section \ref{sect_conv_whole_plane}.
Throughout the whole section, we conserve the notation from Section \ref{sect_idea_proof}.
\subsection{Uniform weak Weyl's law, proofs of Theorem \ref{thm_unif_bnd_eig} and Corollary \ref{cor_red_zeta_conv_re1}}\label{sect_unif_weyl_law}
	The goal of this section is to prove Theorem \ref{thm_unif_bnd_eig} and, as a consequence, Corollary \ref{cor_red_zeta_conv_re1}.
	The idea is to establish Theorem \ref{thm_unif_bnd_eig} by proving that the eigenvalues of $n^2 \cdot\laplcomp_{\Psi_n}^{F_n}$ are bounded by below by the eigenvalues of $\laplcomp_{\Psi}^{F}$ (up to some shifting and rescaling). 
	Then Theorem \ref{thm_unif_bnd_eig} would follow from Weyl's law, (\ref{weyl_law}). 
	The basic tool for proving such a bound is a construction of a map $\mu_n$, acting on a subspace of ${\rm{Map}}(V(\Psi_n), F_n)$ of uniformly bounded codimension (in $n$) with values in $\ccal^{\infty}_{0, vN}(\Psi, F)$, in such a way so that $\mu_n$ preserves scalar products and also scalar products associated with Laplacians $\laplcomp_{\Psi_n}^{F_n}$ and $\laplcomp_{\Psi}^{F}$ (see Theorem \ref{thm_prop_mu_n} for a precise statement).
	\par To explain the construction of the map $\mu_n$, consider a set $\Upsilon$ of functions  $\rho : [- 1; 1] \to \real$ satisfying (\ref{eq_rho1_rest1}), (\ref{eq_rho1_rest2}), (\ref{eq_rho1_rest3}) and
	\begin{equation}
 		\int_0^{1} \rho(x) (1 - \rho(x)) dx  = 0.
 		\label{eq_rho_rest4}
	\end{equation}
	\begin{prop}
		The set $\Upsilon$ is not empty.
	\end{prop}
	\begin{proof}
		Clearly, the assumptions (\ref{eq_rho1_rest1}), (\ref{eq_rho1_rest2}), (\ref{eq_rho1_rest3}) could be easily satisfied by a function $\rho_1: [-1, 1] \to \real$, having image inside of $[0, 1]$. 
		For such a function, we have
		\begin{equation}\label{eq_fun_rhosm}
			\int_0^{1} \rho_1(x) (1 - \rho_1(x)) dx  > 0.
		\end{equation}
		However, if one takes a function $\rho_2$ which satisfies (\ref{eq_rho1_rest1}), (\ref{eq_rho1_rest2}), (\ref{eq_rho1_rest3}), which has positive values over $[0, 1/2]$ and which takes value $4$ on $[1/4, 1/3]$, it would satisfy
		\begin{equation}\label{eq_fun_rhobig}
			\int_0^{1} \rho_2(x) (1 - \rho_2(x)) dx  < 1 -  \frac{4 \cdot 3}{12} < 0.
		\end{equation}
		We conclude by (\ref{eq_fun_rhosm}) and (\ref{eq_fun_rhobig}) that there is $t_0 \in [0, 1]$ such that the function $\rho := t_0 \rho_1 + (1- t_0) \rho_1$ satisfies (\ref{eq_rho_rest4}).
		However, since the set of functions, satisfying properties (\ref{eq_rho1_rest1}), (\ref{eq_rho1_rest2}), (\ref{eq_rho1_rest3}), is convex, we see that $\rho \in \Upsilon$.
	\end{proof}
	\par From (\ref{eq_rho1_rest3}) and (\ref{eq_rho_rest4}), we see that for any $\rho \in \Upsilon$, we have
	\begin{equation}\label{eq_rho__int}
		\int_0^{1} \rho(x) dx = \frac{1}{2}, \qquad \qquad \int_0^{1} \rho(x)^2 dx = \frac{1}{2}.
	\end{equation}
	\par 
	Let's fix a function $\rho \in \Upsilon$.
	Recall that the sets $V_n(P)$, $P \in {\rm{Con}}(\Psi) \cup {\rm{Ang}}(\Psi)$, were defined in (\ref{eq_defn_vp}).
	Introduce, for brevity, the set 
	\begin{equation}
		V_0(\Psi_n) := \Big\{
			v \in V(\Psi_n) : v \notin V_n(P), \text{ for any } P \in {\rm{Con}}(\Psi) \cup {\rm{Ang}}(\Psi)
		\Big\}.
	\end{equation}
	For $P \in V_0(\Psi_n)$, we define now a function $\mu_P \in \ccal^{\infty}_{0, vN}(\Psi)$ which will be constructed with the help of the fixed $\rho \in \Upsilon$. 
	We fix linear coordinates $x, y$, centered at $P$, which have axes parallel to the boundaries of the tiles of $\Psi$. 
	\par Suppose $P$ satisfies $\dist_{\Psi}(P, \partial \Psi) > \frac{1}{2n}$. 
	Normalize the coordinates $x, y$ in such a way that they identify the union of the $4$ squares, formed by the edges and the vertices of $\Psi_n$ containing $P$, with the square $[-1, 1] \times [-1, 1]$.
	Over $\{ (x, y) \in \real^2 : |x| < 1, |y| < 1 \}$, define 
	\begin{equation}
		\mu_{P}(x, y) := \rho(x)\rho(y),
	\end{equation}
	and extend it by zero to other values.
	\par Now, suppose $P$ satisfies $\dist_{\Psi}(P, \partial \Psi) = \frac{1}{2n}$. 
	Suppose that the boundary near $P$ is parallel with the axis of $y$-coordinate.
	Normalize the coordinates $x$, $y$ so that they identify the union of a rectangle and two squares, formed by the boundary, arcs perpendicular to the boundary and by the edges and the vertices of a subgraph of $\Psi_n$, containing $P$, with the rectangle $[0, 1] \times [-\frac{1}{2}, 1]$, so that $\{x < 0\}$ corresponds to a region containing $\partial \Psi$.
	Over this rectangle, define  
	\begin{equation}
		\mu_{P}(x, y) := 
		\begin{cases} 
			\hfill \rho(x)\rho(y), & \text{ for } x > 0, \\
			\hfill \rho(y), & \text{ for } x \leq 0,
 		\end{cases}
	\end{equation}
	and extend it by zero to other values.
	\par 
	We define the functional $\mu_n : {\rm{Map}}(V(\Psi_n), F_n) \to \ccal^{\infty}_{0, vN}(\Psi, F)$ by
	\begin{equation}
		\mu_n(f)(z) = \sum_{P \in V_0(\Psi_n)} \mu_P(z) f(P),
	\end{equation}
	where we implicitly used the parallel transport with respect to $\nabla^F$.
	It is an easy verification that (\ref{eq_rho1_rest1}) ensures that the image of $\mu_n$ lies in $\ccal^{\infty}_{0, vN}(\Psi, F)$.
	The main result of this section is
	\begin{thm}\label{thm_prop_mu_n}
		For any $f \in  {\rm{Map}}(V(\Psi_n), F_n)$ with ${\rm{supp}} f \subset V_0(\Psi)$, the following holds
		\begin{align}
			&
			\frac{1}{n^2} \scal{f}{f}_{L^2(\Psi_n, F_n)} = \scal{\mu_n (f)}{\mu_n (f)}_{L^2(\Psi, F)},\label{eq_mu_ort}
			\\
			&
			\scal{\laplcomp_{\Psi_n}^{F_n} f}{f}_{L^2(\Psi_n, F_n)} = \frac{1}{C} \scal{\laplcomp_{\Psi}^{F} \mu_n (f)}{\mu_n (f)}_{L^2(\Psi, F)},\label{eq_mu_lapl}
		\end{align}
		where the constant $C > 0$ is defined by
		\begin{equation}\label{eq_int_rho_pr_sq}
			C = \int_{0}^{1} \rho'(x)^2 dx.		
		\end{equation}
	\end{thm}
	The proof of Theorem \ref{thm_prop_mu_n} is a simple verification, and it is given in the end of this section.
	We will now explain how Theorem \ref{thm_prop_mu_n} could be used to prove Theorem \ref{thm_unif_bnd_eig}.
	We define $t \in \nat$ by
	\begin{equation}\label{eq_defn_t_num1}
		t = \frac{2 \rk{F}}{\pi}  \cdot \Big( \sum_{P  \in {\rm{Con}}(\Psi)} \angle(P) +  \sum_{Q  \in {\rm{Ang}}(\Psi)} \angle(Q) \Big).
	\end{equation}
	Clearly, for any $n \in \nat^*$, we have
	\begin{equation}\label{eq_defn_t_num}
		t = \rk{F} \cdot \sum_{P  \in {\rm{Con}}(\Psi) \cup {\rm{Ang}}(\Psi)} \# V_n(P).
	\end{equation}
	Let's see how Theorem \ref{thm_prop_mu_n} implies the following result
	\begin{thm}\label{thm_unif_w_l_bound1}
		For any $n \in \nat^*$, $i \in \nat$, $i + t \leq \# V(\Psi_n)$, the following inequality holds
		\begin{equation}
			C \lambda_{i + t}^{n} \geq \lambda_i,
		\end{equation}
		where  $C > 0$ is defined in (\ref{eq_int_rho_pr_sq}). 
	\end{thm}
	\begin{proof}
		Consider the vector space $V_{i + t}^{n} \subset {\rm{Map}}(V(\Psi_n), F_n)$, spanned by the first $i + t$ eigenvectors of $\laplcomp_{\Psi_n}^{F_n}$.
		Consider a subspace $V_{i + t, 0}^{n} \subset V_{i + t}^{n}$, which take zero values on $Q \notin V_0(\Psi_n)$.
		By (\ref{eq_defn_t_num}),
		\begin{equation}\label{eq_dim_vit_0_bound}
			\dim V_{i + t, 0}^{n} \geq i.
		\end{equation}
		Now, construct a vector space $\mu_n(V_{i + t, 0}^{n}) \subset \ccal^{\infty}_{0, vN}(\Psi, F)$. 
		By (\ref{eq_rho1_rest1}), the value of $\mu_n(f)$, evaluated at $V(\Psi_n)$, coincide with the value of $f$ at the evaluated point.
		Thus, we have
		\begin{equation}\label{eq_dim_vit_1_bound}
			\dim \mu_n(V_{i + t, 0}^{n}) = \dim V_{i + t, 0}^{n}.
		\end{equation}
		Clearly, Theorem \ref{thm_prop_mu_n} implies the following
		\begin{equation}\label{eq_unif_bound_fin_aux_1}
			\sup_{f \in \mu_n(V_{i + t, 0}^{n})} 
			\bigg\{ 
				\frac{\scal{\laplcomp_{\Psi}^{F} f}{f}_{L^2(\Psi, F)}}{\scal{f}{f}_{L^2(\Psi, F)}}
			\bigg\}
			=
			C 
			\sup_{f \in V_{i + t, 0}^{n}} 
			\bigg\{ 
				\frac{\scal{n^2 \cdot \laplcomp_{\Psi_n}^{F_n} f}{f}_{L^2(\Psi_n, F_n)}}{\scal{f}{f}_{L^2(\Psi_n, F_n)}}
			\bigg\}
		\end{equation}
		However, we trivially have
		\begin{equation}\label{eq_unif_bound_fin_aux_2}
			\sup_{f \in V_{i + t, 0}^{n}} 
			\bigg\{ 
				\frac{\scal{n^2 \cdot \laplcomp_{\Psi_n}^{F_n} f}{f}_{L^2(\Psi_n, F_n)}}{\scal{f}{f}_{L^2(\Psi_n, F_n)}}
			\bigg\}
			\leq 
			\sup_{f \in V_{i + t}^{n}} 
			\bigg\{ 
				\frac{\scal{n^2 \cdot \laplcomp_{\Psi_n}^{F_n} f}{f}_{L^2(\Psi_n, F_n)}}{\scal{f}{f}_{L^2(\Psi_n, F_n)}}
			\bigg\}
			=
			\lambda_{i + t}^{n}
		\end{equation}
		We use the characterization of the eigenvalues of $\laplcomp_{\Psi}^{F}$ through Rayleigh quotient
		\begin{equation}\label{eq_rayleigh_cont}
			\lambda_{i} = \inf_{\substack{V \subset {\rm{Dom}}_{Fr}(\laplcomp_{\Psi}^{F}) } }  \sup_{f \in V} 
			\bigg\{ 
				\frac{\scal{\laplcomp_{\Psi}^{F} f}{f}_{L^2(\Psi, F)}}{\scal{f}{f}_{L^2(\Psi, F)}}
				 : \dim V = i			
			\bigg\}.
		\end{equation}
		By (\ref{eq_dim_vit_0_bound}), (\ref{eq_dim_vit_1_bound}), (\ref{eq_rayleigh_cont}), we have
		\begin{equation}\label{eq_unif_bound_fin_aux_3}
			\lambda_i 
			\leq 
			\sup_{f \in \mu_n(V_{i + t, 0}^{n})} 
			\bigg\{ 
				\frac{\scal{\laplcomp_{\Psi}^{F} f}{f}_{L^2(\Psi, F)}}{\scal{f}{f}_{L^2(\Psi, F)}}
			\bigg\}.
		\end{equation}
		We conclude by (\ref{eq_unif_bound_fin_aux_1}), (\ref{eq_unif_bound_fin_aux_2}) and (\ref{eq_unif_bound_fin_aux_3}).
	\end{proof}
	\begin{proof}[Proof of Theorem \ref{thm_unif_bnd_eig}]
		It follows from Theorem \ref{thm_unif_w_l_bound1} and (\ref{weyl_law}).
	\end{proof}
	\begin{proof}[Proof of  Corollary \ref{cor_red_zeta_conv_re1}.]
		For simplicity of the presentation, we will suppose that the spectrum of $\laplcomp_{\Psi}^{F}$ is simple, i.e. there is no multiple eigenvalues.
		\par 
		Fix $s \in \comp$, by (\ref{eq_phi_alpha_assumpt}), we have the following
		\begin{equation}\label{eq_cor_red_zeta_conv_aux_1}
			{\rm{Tr}} \Big[  \phi_{\alpha} \cdot (n^2 \cdot  \laplcomp_{U_{\alpha, n}}^{F_n, \perp})^{-s} \cdot \psi_{\alpha} \Big] 
			=
			{\rm{Tr}} \Big[  \phi_{\alpha} \cdot (n^2 \cdot  \laplcomp_{U_{\alpha, n}}^{F_n, \perp})^{-s} \Big]
			=
			\frac{1}{n^2}
			\sum_{i \geq 1} 
			 \scal{\phi_{\alpha} f_i^{n}}{f_i^{n}}_{L^2(\Psi_n, F_n)} (\lambda_i^{n})^{-s}.
		\end{equation}
		Now, by Theorem \ref{thm_eigvec_convergence}, we know that in $L^2(\Psi)$, as $n \to \infty$, we have
		\begin{equation}\label{eq_cor_red_zeta_conv_aux_2}
			L_n(f_i^{n}) \to f_i,
		\end{equation}
		where $f_i$ is the eigenvector of $\laplcomp_{\Psi}^{F}$ corresponding to the eigenvalue $\lambda_i = \lim_{n \to \infty} \lambda_i^{n}$.
		By Propositions \ref{prop_cor_red_zeta_conv_aux_3}, \ref{prop_ln_appr_easy} and (\ref{eq_cor_red_zeta_conv_aux_2}), we conclude that, as $n \to \infty$, we have
	\begin{equation}\label{eq_cor_red_zeta_conv_aux_5}
			\frac{1}{n^2}
			 \scal{\phi_{\alpha} f_i^{n}}{f_i^{n}}_{L^2(\Psi_n, F_n)}
			 \to
			 \scal{\phi_{\alpha}  f_i }{f_i }_{L^2(\Psi)}.
		\end{equation}
		From Theorems \ref{thm_eigval_convergence}, \ref{thm_unif_bnd_eig}, (\ref{weyl_law}) and (\ref{eq_cor_red_zeta_conv_aux_5}), we conclude.
	\end{proof}
	\begin{proof}[Proof of Theorem  \ref{thm_prop_mu_n}.]
		Let's establish that for $P, Q \in V_0(\Psi_n)$, the following identity holds
		\begin{equation}\label{eq_aux_mu_prod_1}
			\scal{\mu_P}{\mu_Q}_{L^2(\Psi)} = \delta_{P, Q} \frac{1}{n^2},
		\end{equation}
		where $\delta_{P, Q}$ is the Kronecker delta symbol.
		If fact, suppose first that $P, Q$ are not connected neither by an arc nor by a combination of a vertical and a horizontal arc.
		Then clearly  (\ref{eq_aux_mu_prod_1}) holds as the supports of $\mu_P$ and $\mu_Q$ are disjoint by definition.
		\par 
		First, suppose that $P$ and $Q$ are connected by a horizontal arc.
		We fix coordinates $x, y$, which have axes parallel to the horizontal and vertical directions respectively. 
		Suppose that the coordinate $x$ is normalized in such a way that $x(P) = 0$ and $x(Q) = 1$.
		\par 
		Suppose that $\dist_{\Psi}(P, \partial \Psi) > \frac{1}{2n}$. Normalize the coordinate $y$ in such a way that $y$ takes values in $[-1, 1]$ on squares adjacent to $P, Q$.
		Then by (\ref{eq_rho_rest4}), we have
		\begin{equation}\label{eq_mu_ort}
			\scal{\mu_P}{\mu_Q}_{L^2(\Psi)}  = \frac{1}{n^2} \int_{-1}^{1} \int_{0}^{1} \rho(x)(1-\rho(x)) dx \rho(y)^2 dy = 0.
		\end{equation}
		\par Now, suppose that $\dist_{\Psi}(P, \partial \Psi) = \frac{1}{2n}$. Then normalize the coordinate $y$ in such a way that $y$ takes values in $[0, 1]$ on a square, formed by the edges of $\Psi_n$, adjacent to $P, Q$ and $[-\frac{1}{2}, 0]$ on the other side of the interval from $P, Q$, containing a part of $\partial \Psi$.
		Then by (\ref{eq_rho_rest4}), we have
		\begin{equation}\label{eq_mu_ort2}
			\scal{\mu_P}{\mu_Q}_{L^2(\Psi)}  = \frac{1}{n^2} \int_{0}^{1} \int_{0}^{1} \rho(x)(1-\rho(x)) dx \rho(y)^2 dy  + \frac{1}{2n^2} \int_{0}^{1} \rho(x)(1-\rho(x)) dx= 0.
		\end{equation}
		\par 
		By a similar calculation, one can see that, $\scal{\mu_P}{\mu_Q}_{L^2(\Psi)} = 0$ in the case if $P$ and $Q$ are connected by a combination of a vertical and a horizontal arc.
		\par Finally, from (\ref{eq_rho__int}), we see that for $P$ satisfying  $\dist_{\Psi}(P, \partial \Psi) > \frac{1}{2n}$, we have
		\begin{equation}\label{eq_mu_sq}
			\scal{\mu_P}{\mu_P}_{L^2(\Psi)}  = \frac{1}{n^2} \int_{-1}^{1} \int_{-1}^{1} \rho(x)^2 dx \rho(y)^2 dy = \frac{1}{n^2}.
		\end{equation}
		The same holds for $P$, satisfying $\dist_{\Psi}(P, \partial \Psi) = \frac{1}{2n}$. 
		From (\ref{eq_mu_ort}), (\ref{eq_mu_ort2}) and (\ref{eq_mu_sq}), we get (\ref{eq_aux_mu_prod_1}).
		\par Now let's establish (\ref{eq_mu_lapl}).
		Suppose that $z \in \Psi$ lies in a square $\Phi$, formed by the edges of $\Psi_n$, with vertices $P, Q, R, S \in V(\Psi_n)$.
		Suppose $P, Q$ and $R, S$ share the same horizontal coordinate and $P, R$ and $Q, S$ share the same vertical coordinate.
		We fix linear coordinates $x, y$, which have axes parallel to the horizontal and vertical directions respectively. 
		We normalize $x$ (resp. $y$) in such a way that $x(P) = 0$ and $x(Q) = 1$ (resp. $y(P) = 0$ and $y(R) = 1$).
		Then by (\ref{eq_rho1_rest3}), we have
		\begin{equation}\label{eq_ev_der_mu}
			 \nabla_{\frac{\partial}{\partial x}}^{F} \mu_n(f)(z) = (f(P) - f(Q))\rho'(x) \rho(y) + (f(R) - f(S))\rho'(x) (1 - \rho(y)).
		\end{equation}
		From (\ref{eq_rho_rest4}), (\ref{eq_ev_der_mu}), we deduce that
		\begin{multline}\label{eq_prop_mufin_1}
			\scal{\nabla_{\frac{\partial}{\partial x}}^{F} \mu_n(f)}{\nabla_{\frac{\partial}{\partial x}}^{F} \mu_n(f)}_{L^2(\Phi, F)} = 
			(f(P) - f(Q))^2 \cdot \int_{0}^{1} \int_{0}^{1} \rho'(x)^2 \rho(y)^2 dx dy
			\\
			+
			(f(R) - f(S))^2 \cdot \int_{0}^{1} \int_{0}^{1} \rho'(x)^2 (1-\rho(y))^2 dx dy
			\\
			=
			\frac{C}{2}(f(P) - f(Q))^2 + \frac{C}{2}(f(R) - f(S))^2. 
		\end{multline}
		Similarly, for the derivative with respect to $y$-variable, we have
		\begin{multline}\label{eq_prop_mufin_2}
			\scal{\nabla_{\frac{\partial}{\partial y}}^{F} \mu_n(f)}{\nabla_{\frac{\partial}{\partial y}}^{F} \mu_n(f)}_{L^2(\Phi, F)} = 
			(f(P) - f(R))^2 \cdot \int_{0}^{1} \int_{0}^{1} \rho'(y)^2 \rho(x)^2 dx dy
			\\
			+
			(f(Q) - f(S))^2 \cdot \int_{0}^{1} \int_{0}^{1} \rho'(y)^2 (1-\rho(x))^2 dx dy
			\\
			=
			\frac{C}{2}(f(P) - f(R))^2 + \frac{C}{2}(f(Q) - f(S))^2. 
		\end{multline}
		\par Now, suppose that $z$ lies near the boundary, let $P$ and $Q$ be the nearest neighbors of $z$. 
		Clearly, $P$ and $Q$ share either the same $x$ coordinate or the same $y$ coordinate.
		Suppose they share the same $x$ coordinate.
		We normalize the coordinate $x$ in such a way that $x(P) = 0$ and $x$ takes value $\frac{1}{2}$ on the boundary of $\Psi$ near $P$ and $Q$. 
		We normalize the coordinate $y$ in such a way that $y(P) = 0$ and $y(Q) = 1$.
		Let $\Phi$ be the rectangle with axis parallel to the $x$ and $y$ directions with one side lying in $\partial \Psi$ and two vertices $P, Q$.
		Then we clearly have $\nabla_{\frac{\partial}{\partial x}}^{F} \mu_n(f) = 0$ and 
		\begin{equation}\label{eq_prop_mufin_3}
			\scal{\nabla_{\frac{\partial}{\partial y}}^{F} \mu_n(f)}{\nabla_{\frac{\partial}{\partial y}}^{F} \mu_n(f)}_{L^2(\Phi, F)} 
			= 
			\frac{(f(P) - f(Q))^2}{2} \cdot \int_{0}^{1} \rho'(y)^2 dy
			=
			\frac{C}{2}(f(P) - f(Q))^2.
		\end{equation}
		As $\mu_n(f)$ vanishes in the neighborhood of $P \in V(\Psi_n) \setminus V_0(\Psi_n)$, we conclude by (\ref{eq_prop_mufin_1}), (\ref{eq_prop_mufin_2}), (\ref{eq_prop_mufin_3}) that
		\begin{equation}\label{eq_prop_mufin_4}
			\scal{\nabla^{F} \mu_n(f)}{\nabla^{F}  \mu_n(f)}_{L^2(\Psi, F)}
			=
			C
			\sum_{(P, Q) \in E(\Psi_n)} (f(P) - f(Q))^2.
		\end{equation}
		Then (\ref{eq_prop_mufin_4}) clearly implies (\ref{eq_mu_lapl}) by (\ref{eq_lapl_self_adjoint}).
	\end{proof}

\subsection{Uniform bound on discrete zeta-functions, a proof of Theorem \ref{thm_bound_zeta}}\label{sect_conv_whole_plane}
	The main goal of this section is to prove Theorem \ref{thm_bound_zeta}.
	We use the notation from Section \ref{sect_idea_proof}.
	\par
	In this section we were inspired a lot by Müller \cite{Mull78}.
	The following two theorems form the core of the proof of Theorem \ref{thm_bound_zeta}. 
	Their proofs will be given in Section \ref{sect_unf_powrs_lapl}.
	\begin{thm}\label{thm_power_bnd_supp_disj}
		Let's fix two open subsets $U, V \subset \Psi$ such that there exists another open subset $W \subset \Psi$ satisfying $\overline{U} \subset W$ and $V \cap W = \emptyset$.
		For any compact $K \subset \comp$, there is $C > 0$ such that for any $s \in K$, $n \in \nat^*$ and any functions $\phi, \psi \in \ccal^{\infty}(\Psi) \cap L^{\infty}(\Psi)$ with ${\rm{supp}}(\phi) \subset U$, ${\rm{supp}}(\psi) \subset V$:
		\begin{equation}
			\big\| \phi \cdot ( n^2 \cdot\laplcomp_{\Psi_n}^{F_n, \perp})^{s} \cdot \psi \big\|_{L^2(\Psi_n, F_n)}^{0} \leq C \norm{\phi}_{L^{\infty}(\Psi)} \norm{\psi}_{L^{\infty}(\Psi)},
		\end{equation}
		where $\norm{\cdot}_{L^2(\Psi_n, F_n)}^{0}$ is the operator norm and we interpret the multiplication by functions $\phi$, $\psi$ by the multiplication of their restriction on $V(\Psi_n)$.
	\end{thm}
	\begin{rem}
		It is trivial that Theorem \ref{thm_power_bnd_supp_disj} holds for $s \in \nat$, since the operators $\laplcomp_{\Psi_n}^{F_n}$, $n \in \nat^*$ are “local", and thus, their integer powers are “local" as well. This, of course, cannot be said about rational or complex powers.
	\end{rem}
	\begin{thm}\label{thm_power_diff_bnd}
		Let $\phi_{\alpha}$, $\psi_{\alpha}$ be as in (\ref{eq_zeta_renorm_discr}).
		For any compact $K \subset \comp$, there is $C > 0$ such that for any $n \in \nat^*$, we have
		\begin{equation}
			\Big\| 
				(n^2 \cdot\laplcomp_{\Psi_n}^{F_n, \perp})^{-s} - \sum_{\alpha \in I} \phi_{\alpha} \cdot (n^2 \cdot  \laplcomp_{U_{\alpha, n}}^{F_n, \perp})^{-s} \cdot \psi_{\alpha}
			\Big\|_{L^2(\Psi_n, F_n)}^{0} \leq C.
		\end{equation}
	\end{thm}
	\begin{proof}[Proof of Theorem \ref{thm_bound_zeta}.]
		Let's denote by $P_{\Psi_n}^{F_n}$ the orthogonal projection onto the functions from $\ker \laplcomp_{\Psi_n}^{F_n}$. 
		By the classical property of trace, we have
		\begin{multline}\label{eq_zeta_boot_aux_1}
			{\rm{Tr}}
			\Big[
			(n^2 \cdot\laplcomp_{\Psi_n}^{F_n, \perp})^{-s} - \sum_{\alpha \in I} \phi_{\alpha} \cdot (n^2 \cdot  \laplcomp_{U_{\alpha, n}}^{F_n, \perp})^{-s} \cdot \psi_{\alpha}			
			\Big]
			\\
			\leq
			\Big\| 
			(n^2 \cdot\laplcomp_{\Psi_n}^{F_n, \perp})^{-s+2} 
			- 
			\sum_{\alpha \in I} \phi_{\alpha} \cdot (n^2 \cdot  \laplcomp_{U_{\alpha, n}}^{F_n, \perp})^{-s} \cdot \psi_{\alpha} \cdot
			(n^2 \cdot\laplcomp_{\Psi_n}^{F_n} + P_{\Psi_n}^{F_n})^{2}
			\Big\|_{L^2(\Psi_n, F_n)}^{0}
			\\
			\cdot
			\tr{(n^2 \cdot\laplcomp_{\Psi_n}^{F_n} + P_{\Psi_n}^{F_n})^{-2}}.
		\end{multline}
		Since ${\rm{supp}}(\psi_{\alpha}) \subset U_{\alpha}$, by the fact that $\laplcomp_{\Psi_n}^{F_n}$ is “local", for $n$ big enough, we have
		\begin{equation}\label{eq_zeta_boot_aux00}
		\psi_{\alpha}
			(n^2 \cdot\laplcomp_{\Psi_n}^{F_n})^{2}
			=
		\psi_{\alpha}
			(n^2 \cdot \laplcomp_{U_{\alpha, n}}^{F_n})^{2}.
		\end{equation}
		By (\ref{eq_zeta_boot_aux00}), we can write
		\begin{multline}\label{eq_zeta_boot_aux0}
			\phi_{\alpha} (n^2 \cdot  \laplcomp_{U_{\alpha, n}}^{F_n, \perp})^{-s} \psi_{\alpha}
			(n^2 \cdot\laplcomp_{\Psi_n}^{F_n})^{2}
			=
			\phi_{\alpha} (n^2 \cdot  \laplcomp_{U_{\alpha, n}}^{F_n, \perp})^{-s + 2} \psi_{\alpha}
			\\
			+
			\phi_{\alpha} (n^2 \cdot  \laplcomp_{U_{\alpha, n}}^{F_n, \perp})^{-s + 1} 
			\big[(n^2 \cdot \laplcomp_{U_{\alpha, n}}^{F_n}), \psi_{\alpha} \big]
			\\
			+
			\phi_{\alpha} (n^2 \cdot  \laplcomp_{U_{\alpha, n}}^{F_n, \perp})^{-s} 
			\Big[(n^2 \cdot \laplcomp_{U_{\alpha, n}}^{F_n}),
			\big[(n^2 \cdot \laplcomp_{U_{\alpha, n}}^{F_n}), 
			\psi_{\alpha} \big] \Big].
		\end{multline}
		Now, as $\psi_{\alpha}$ is smooth, we conclude that there is $C > 0$ such that for any $n \in \nat^*$, we have
		\begin{equation}\label{eq_zeta_boot_aux1}
		\begin{aligned}
			&
			\Big\| \big[(n^2 \cdot \laplcomp_{U_{\alpha, n}}^{F_n}), \psi_{\alpha} \big] \Big\|_{L^{\infty}(\Psi)} \leq  C,
			\\
			&
			\Big\| \Big[(n^2 \cdot \laplcomp_{U_{\alpha, n}}^{F_n}),
			\big[(n^2 \cdot \laplcomp_{U_{\alpha, n}}^{F_n}), 
			\psi_{\alpha} \big] \Big] \Big\|_{L^{\infty}(\Psi)} \leq C
		\end{aligned}
		\end{equation}
		Also the supports of functions on the left-hand side of (\ref{eq_zeta_boot_aux1}) is located in a ball of radius $\frac{4}{n}$ around the support of $\psi_{\alpha}(1-\psi_{\alpha})$.
		Moreover, by (\ref{eq_phi_alpha_assumpt}), we see that the functions $\phi_{\alpha}$ and $\big[(n^2 \cdot \laplcomp_{U_{\alpha, n}}^{F_n}), \psi_{\alpha} \big]$ (or $\big[(n^2 \cdot \laplcomp_{U_{\alpha, n}}^{F_n}),
			\big[(n^2 \cdot \laplcomp_{U_{\alpha, n}}^{F_n, \perp}), 
			\psi_{\alpha} \big] \big]$) satisfy the assumption of Theorem \ref{thm_power_bnd_supp_disj}.
		By this, Theorem \ref{thm_power_bnd_supp_disj} and (\ref{eq_zeta_boot_aux1}), we conclude that for any compact $K \subset \comp$, there is $C > 0$ such that for any $s \in K$ and $n \in \nat^*$, we have
		\begin{equation}\label{eq_zeta_boot_aux3}
		\begin{aligned}
			&
			\Big\| \phi_{\alpha} (n^2 \cdot  \laplcomp_{U_{\alpha, n}}^{F_n, \perp})^{-s + 1} 
			\big[(n^2 \cdot \laplcomp_{U_{\alpha, n}}^{F_n}), \psi_{\alpha} \big] \Big\|_{L^2(\Psi_n, F_n)}^{0} 
			\leq  C,
			\\
			&
			\Big\| 
			\phi_{\alpha} (n^2 \cdot  \laplcomp_{U_{\alpha, n}}^{F_n, \perp})^{-s} 
			\Big[(n^2 \cdot \laplcomp_{U_{\alpha, n}}^{F_n}),
			\big[(n^2 \cdot \laplcomp_{U_{\alpha, n}}^{F_n}), 
			\psi_{\alpha} \big] \Big]
			\Big\|_{L^2(\Psi_n, F_n)}^{0} 
			\leq C.
		\end{aligned}
		\end{equation}
		By (\ref{eq_zeta_boot_aux0}), (\ref{eq_zeta_boot_aux3}), we see that there is $C > 0$ such that we have
		\begin{multline}\label{eq_zeta_boot_aux4}
		\Big\| 
			(n^2 \cdot\laplcomp_{\Psi_n}^{F_n, \perp})^{-s+2} 
			- 
			\sum_{\alpha \in I} \phi_{\alpha} (n^2 \cdot  \laplcomp_{U_{\alpha, n}}^{F_n, \perp})^{-s} \psi_{\alpha}
			(n^2 \cdot\laplcomp_{\Psi_n}^{F_n})^{2}
			\Big\|_{L^2(\Psi_n, F_n)}^{0}
			\\
			\leq 
			\Big\| 
			(n^2 \cdot\laplcomp_{\Psi_n}^{F_n, \perp})^{-s+2} 
			- 
			\sum_{\alpha \in I} \phi_{\alpha} (n^2 \cdot  \laplcomp_{U_{\alpha, n}}^{F_n, \perp})^{-s + 2} \psi_{\alpha}
			\Big\|_{L^2(\Psi_n, F_n)}^{0}
			+
			C.
		\end{multline}
		However, by Theorem \ref{thm_power_diff_bnd}, we conclude that there is $C > 0$ such that we have
		\begin{equation}\label{eq_zeta_boot_aux5}
			\Big\| 
			(n^2 \cdot\laplcomp_{\Psi_n}^{F_n, \perp})^{-s+2} 
			- 
			\sum_{\alpha \in I} \phi_{\alpha} (n^2 \cdot  \laplcomp_{U_{\alpha, n}}^{F_n, \perp})^{-s + 2} \psi_{\alpha}
			\Big\|_{L^2(\Psi_n, F_n)}^{0}
			\leq
			C.
		\end{equation}
		Now, by Theorem \ref{thm_unif_bnd_eig}, we know that there is $C > 0$ such that for any $n \in \nat^*$, we have
		\begin{equation}\label{eq_zeta_boot_aux6}
			\tr{(n^2 \cdot\laplcomp_{\Psi_n}^{F_n} + P_{\Psi_n}^{F_n})^{-2}} \leq C.
		\end{equation}
		We also clearly have 
		\begin{equation}\label{eq_zeta_boot_aux7}
			\phi_{\alpha} (n^2 \cdot  \laplcomp_{U_{\alpha, n}}^{F_n, \perp})^{-s} \psi_{\alpha} P_{\Psi_n}^{F_n} =
			\phi_{\alpha} (n^2 \cdot  \laplcomp_{U_{\alpha, n}}^{F_n, \perp})^{-s} P_{\Psi_n}^{F_n} 
			-
			\phi_{\alpha} (n^2 \cdot  \laplcomp_{U_{\alpha, n}}^{F_n, \perp})^{-s} (1 - \psi_{\alpha}) P_{\Psi_n}^{F_n}.
		\end{equation}
		But since the flat sections of $\Psi$ restrict to flat sections on $U_{\alpha}$, we have
		\begin{equation}\label{eq_zeta_boot_aux8}
				(n^2 \cdot  \laplcomp_{U_{\alpha, n}}^{F_n, \perp})^{-s} P_{\Psi_n}^{F_n} = 0.
		\end{equation}
		Also by Theorem \ref{thm_power_bnd_supp_disj} and (\ref{eq_phi_alpha_assumpt}), there is $C > 0$ such that we have
		\begin{equation}\label{eq_zeta_boot_aux9}
			\Big\| \phi_{\alpha} (n^2 \cdot  \laplcomp_{U_{\alpha, n}}^{F_n, \perp})^{-s} (1 - \psi_{\alpha}) P_{\Psi_n}^{F_n} \Big \|_{L^2(\Psi_n, F_n)}^{0}
			\leq
			\Big\| \phi_{\alpha} (n^2 \cdot  \laplcomp_{U_{\alpha, n}}^{F_n, \perp})^{-s} (1 - \psi_{\alpha}) \Big \|_{L^2(\Psi_n, F_n)}^{0}
			\leq
			C.
		\end{equation}
		We conclude by (\ref{eq_zeta_boot_aux7}),  (\ref{eq_zeta_boot_aux8}) and (\ref{eq_zeta_boot_aux9}) that there is $C > 0$ for any $n \in \nat^*$, such that
		\begin{equation}\label{eq_zeta_boot_aux10}
			\Big\| \phi_{\alpha} (n^2 \cdot  \laplcomp_{U_{\alpha, n}}^{F_n, \perp})^{-s} \psi_{\alpha} P_{\Psi_n}^{F_n} \Big \|_{L^2(\Psi_n, F_n)}^{0}
			\leq 
			C
		\end{equation}
		By (\ref{eq_zeta_boot_aux_1}), (\ref{eq_zeta_boot_aux4}), (\ref{eq_zeta_boot_aux5}), (\ref{eq_zeta_boot_aux6}) and (\ref{eq_zeta_boot_aux10}), we conclude.
	\end{proof}

	\subsection{Uniform bound on Laplacian powers, proofs of Theorems \ref{thm_power_bnd_supp_disj}, \ref{thm_power_diff_bnd}}\label{sect_unf_powrs_lapl}
	In this section we prove Theorems \ref{thm_power_bnd_supp_disj}, \ref{thm_power_diff_bnd}.
	As in Section \ref{sect_conv_whole_plane}, here we were inspired a lot by Müller \cite{Mull78}.
	We conserve the notation from Section \ref{sect_conv_whole_plane}.
	Let's start with Theorem \ref{thm_power_bnd_supp_disj}.
	The main ingredient is the following
	\begin{thm}\label{thm_exp_12_bound}
		Suppose open subsets $U ,V \subset \Psi$ are as in Theorem \ref{thm_power_bnd_supp_disj}. 
		Then there is $\epsilon > 0$ and $C > 0$ such that for any $t \in \comp$, $|t| < \epsilon$, $n \in \nat^*$ and functions $\phi, \psi \in \ccal^{\infty}(\Psi) \cap L^{\infty}(\Psi)$ with ${\rm{supp}}(\phi) \subset U$, ${\rm{supp}}(\psi) \subset V$, we have
		\begin{equation}\label{eq_exp_12_bound}
			\Big\| \phi \cdot \exp \big(-t (n^2 \cdot\laplcomp_{\Psi_n}^{F_n, \perp})^{1/2} \big) \cdot \psi \Big\|_{L^2(\Psi_n, F_n)}^{0} \leq C \norm{\phi}_{L^{\infty}(\Psi)} \norm{\psi}_{L^{\infty}(\Psi)},
		\end{equation}
		where $\norm{\cdot}_{L^2(\Psi_n, F_n)}^{0}$ is the operator norm.
	\end{thm}
	\begin{proof}
		Let's define $A_n(t), B_n(t) \in \enmr{{\rm{Map}}(V(\Psi_n), \comp)}$, $t \in \comp$, as follows
		\begin{equation}\label{eq_exp_12_bound_aux_1}
		\begin{aligned}			
			& A_n(t) :=
			\sum_{k = 0}^{+\infty}
			\frac{(-t)^{2k} (n^2 \cdot\laplcomp_{\Psi_n}^{F_n})^{k} }{(2k)!},
			\\
			& B_n(t) := \sum_{k = 0}^{+\infty}
			\frac{(-t)^{2k+1} (n^2  \cdot\laplcomp_{\Psi_n}^{F_n, \perp})^{k + \frac{1}{2}} }{(2k + 1)!}.
		\end{aligned}
		\end{equation}
		Trivially, we have the following upper bound on the spectrum of 
		\begin{equation}\label{eq_spec_bnd_triv}
			\spec (\laplcomp_{\Psi_n}^{F_n}) \subset [0, 8 \rk{F}].
		\end{equation}
		By (\ref{eq_spec_bnd_triv}) the series (\ref{eq_exp_12_bound_aux_1}) converge absolutely for $t \in \comp$.
		\par 
		By the assumption on the supports of $\phi$ and $\psi$, there is $c > 0$ such that for any $k < nc$, $k \in \nat^*$, we have $\phi \cdot (\laplcomp_{\Psi_n}^{F_n})^k \cdot \psi = 0$, and, as a consequence, we have
		\begin{equation}\label{eq_exp_12_bound_aux_2}
			\norm{\phi \cdot A_n(t) \cdot \psi}_{L^2(\Psi_n, F_n)}^{0}
			\leq
			\sum_{k = cn}^{+ \infty}
			\Big \|
				\phi \cdot \frac{(-t)^{2k} (n^2 \cdot\laplcomp_{\Psi_n}^{F_n})^{k} }{(2k)!} \cdot \psi
			\Big \|_{L^2(\Psi_n, F_n)}^{0}.
		\end{equation}
		Now, by the upper bound on the spectrum of $\laplcomp_{\Psi_n}^{F_n}$ from (\ref{eq_spec_bnd_triv}), the following bound holds
		\begin{equation}\label{eq_exp_12_bound_aux_3}
			\Big \|
				\phi \cdot \frac{(-t)^{2k} (n^2 \cdot\laplcomp_{\Psi_n}^{F_n})^{k} }{(2k)!} \cdot \psi
			\Big \|_{L^2(\Psi_n, F_n)}^{0}
			\leq
			\frac{t^{2k} 4^{2k} n^{2k} \rk{F}^k}{(2k)!}
			\norm{\phi}_{L^{\infty}(\Psi)} \cdot \norm{\psi}_{L^{\infty}(\Psi)}.
		\end{equation}
		By the famous Stirling's bound, we have
		\begin{equation}\label{eq_exp_12_bound_aux_4}
			(2k)! \geq \sqrt{2 \pi} (2k)^{2k + \frac{1}{2}} \exp(- 2k).
		\end{equation}
		By (\ref{eq_exp_12_bound_aux_3}) and (\ref{eq_exp_12_bound_aux_4}), we conclude
		\begin{equation}\label{eq_exp_12_bound_aux_5}
			\Big \|
			\phi \cdot \frac{(-t)^{2k} (n^2 \cdot\laplcomp_{\Psi_n}^{F_n})^{k} }{(2k)!} \cdot \psi
			\Big \|_{L^2(\Psi_n, F_n)}^{0}
			\leq
			\frac{(2t)^{2k} n^{2k} \exp(2k)  \rk{F}^k}{(2k)^{2k}}
			\norm{\phi}_{L^{\infty}(\Psi)} \cdot \norm{\psi}_{L^{\infty}(\Psi)}.
		\end{equation}
		Now, for any $k \in \nat^*$, satisfying $k \geq cn$ for some $c > 0$, we have
		\begin{equation}\label{eq_exp_12_bound_aux_6}
			\frac{(2t)^{2k} n^{2k} \exp(2k) }{(2k)^{2k}} \leq (2t)^{2k} e^{2k} c^{-2k}.
		\end{equation}
		From (\ref{eq_exp_12_bound_aux_2}), (\ref{eq_exp_12_bound_aux_5}) and (\ref{eq_exp_12_bound_aux_6}), we conclude that there is $C > 0$ such that for any $t \in \comp$, $|t| \leq \frac{c}{4 e \rk{F}}$ and $\phi, \psi$ as in the statement of the theorem we are proving, we have
		\begin{equation}\label{eq_exp_12_bound_aux_7}
			\norm{\phi \cdot A_n(t) \cdot \psi}_{L^2(\Psi_n, F_n)}^{0}
			\leq
			C \norm{\phi}_{L^{\infty}(\Psi)} \cdot \norm{\psi}_{L^{\infty}(\Psi)}.
		\end{equation}
		\par Now, note that since $\laplcomp_{\Psi_n}^{F_n}$ is a positive self-adjoint operator (see (\ref{eq_lapl_self_adjoint}) and the remark after), for any $t \in \comp$, $\Re(t) \geq 0$, the bound (\ref{eq_exp_12_bound}) trivially holds.
		From this and from the identity 
		\begin{equation}\label{eq_exp_12_bound_aux_8}
			B_n(t) 
			=
			\exp \big(-t (n^2 \cdot\laplcomp_{\Psi_n}^{F_n, \perp})^{1/2} \big) 
			-
			A_n(t),
		\end{equation}
		we conclude that there is $C > 0$ such that for any $t \in \comp$, $|t| \leq \frac{c}{2 e  \rk{F}}$, $\Re(t) \geq 0$, we have
		\begin{equation}\label{eq_exp_12_bound_aux_9}
			\norm{\phi \cdot B_n(t) \cdot \psi}_{L^2(\Psi_n, F_n)}^{0}
			\leq
			C \norm{\phi}_{L^{\infty}(\Psi)} \cdot \norm{\psi}_{L^{\infty}(\Psi)}.
		\end{equation}
		Note, however, that $B_n(t) = - B_n(-t)$. Thus, the bound (\ref{eq_exp_12_bound_aux_9}) also holds for $t \in \comp$, $|t| \leq \frac{c}{2 e  \rk{F}}$ without the restriction on the real part of $t$. 
		From this, (\ref{eq_exp_12_bound_aux_7}) and (\ref{eq_exp_12_bound_aux_8}), we conclude.
	\end{proof}
	Let's see now how Theorem \ref{thm_exp_12_bound} can be applied in the proof of  Theorem \ref{thm_power_bnd_supp_disj}.
	\begin{proof}[Proof of Theorem \ref{thm_power_bnd_supp_disj}.]
		Let $P_{\Psi_n}^{F_n}$ be as in the proof of Theorem \ref{thm_bound_zeta}.
		First, we clearly have
		\begin{equation}\label{eq_thm_exp_12_aux_0}
		\begin{aligned}
			&
			(n^2 \cdot\laplcomp_{\Psi_n}^{F_n} + P_{\Psi_n}^{F_n})^{s} = (n^2 \cdot\laplcomp_{\Psi_n}^{F_n, \perp})^{s}  + P_{\Psi_n}^{F_n},
			\\
			&
			\exp(-t(n^2 \cdot\laplcomp_{\Psi_n}^{F_n} + P_{\Psi_n}^{F_n})^{1/2}) = \exp(-t(n^2 \cdot\laplcomp_{\Psi_n}^{F_n, \perp})^{1/2})  + \exp(-t) P_{\Psi_n}^{F_n}.
		\end{aligned}
		\end{equation}
		thus it is enough to prove the bound of Theorem \ref{thm_power_bnd_supp_disj} for $(n^2 \cdot\laplcomp_{\Psi_n}^{F_n} + P_{\Psi_n}^{F_n})^{s}$ instead of $(n^2 \cdot\laplcomp_{\Psi_n}^{F_n, \perp})^{s}$.
		\par 
		Let's fix $\epsilon > 0$ as in the statement of Theorem \ref{thm_exp_12_bound}.
		For any $k \in \nat$ and $s \in \comp$, $\Re(s) > 0$, by the definition of the $\Gamma$-function, we have
		\begin{multline}\label{eq_thm_exp_12_aux_1}
		(n^2 \cdot\laplcomp_{\Psi_n}^{F_n} + P_{\Psi_n}^{F_n})^{\frac{k}{2} - s} 
			= 
			\frac{1}{\Gamma(s)}
			\int_0^{\epsilon} t^{s-1} (n^2 \cdot\laplcomp_{\Psi_n}^{F_n} + P_{\Psi_n}^{F_n})^{k} \exp \big(-t (n^2 \cdot\laplcomp_{\Psi_n}^{F_n} + P_{\Psi_n}^{F_n})^{1/2} \big) dt
			\\
			+
			\frac{1}{\Gamma(s)}
			\int_{\epsilon}^{+\infty} t^{s-1} (n^2 \cdot\laplcomp_{\Psi_n}^{F_n} + P_{\Psi_n}^{F_n})^{k} \exp \big(-t (n^2 \cdot\laplcomp_{\Psi_n}^{F_n} + P_{\Psi_n}^{F_n})^{1/2} \big) dt
			.
		\end{multline}
		By Theorem \ref{thm_exp_12_bound}, Cauchy theorem and (\ref{eq_thm_exp_12_aux_0}), we conclude  that for any $l \in \nat$, there is $C > 0$ such that we have
		\begin{equation}\label{eq_thm_exp_12_aux_5}
			\Big\| \phi \cdot (n^2 \cdot\laplcomp_{\Psi_n}^{F_n} + P_{\Psi_n}^{F_n})^{l/2} \exp \big(-t (n^2 \cdot\laplcomp_{\Psi_n}^{F_n} + P_{\Psi_n}^{F_n})^{1/2} \big) \cdot \psi \Big\|_{L^2(\Psi_n, F_n)}^{0} \leq C \norm{\phi}_{L^{\infty}(\Psi)} \norm{\psi}_{L^{\infty}(\Psi)}.
		\end{equation}
		Since $\Gamma(s)^{-1}$, $s \in \comp$ is a holomorphic function, and for any $s \in \comp$, $\Re(s) > 0$, the integral $\int_0^{\epsilon} t^{s - 1} dt$ converges, by (\ref{eq_thm_exp_12_aux_5}), we conclude that for any compact $K \subset \{s \in \comp : \Re(s) > 0 \}$, and any $s \in K$, there is $C > 0$ such that
		\begin{multline}\label{eq_thm_exp_12_aux_2}
			\Big\|
			\frac{1}{\Gamma(s)}
			\int_0^{\epsilon} t^{s-1} \phi \cdot (n^2 \cdot\laplcomp_{\Psi_n}^{F_n} + P_{\Psi_n}^{F_n})^{k} \exp \big(-t (n^2 \cdot\laplcomp_{\Psi_n}^{F_n} + P_{\Psi_n}^{F_n})^{1/2} \big) \cdot \psi dt
			\Big\|_{L^2(\Psi_n, F_n)}^{0} 
			\\
			\leq 
			C  \norm{\phi}_{L^{\infty}(\Psi)} \norm{\psi}_{L^{\infty}(\Psi)}.
		\end{multline}
		By integration by parts, we have
		\begin{equation}\label{eq_thm_exp_12_aux_3}
		\begin{aligned}
			\int_{\epsilon}^{+\infty} t^{s-1} & (n^2 \cdot\laplcomp_{\Psi_n}^{F_n} + P_{\Psi_n}^{F_n})^{k} \exp \big(-t (n^2 \cdot\laplcomp_{\Psi_n}^{F_n} + P_{\Psi_n}^{F_n})^{1/2} \big) dt
			\\
			&			
			=
			\int_{\epsilon}^{+\infty} t^{s-1} \Big(- \frac{d}{dt}  \Big)^{2k} \exp \big(-t (n^2 \cdot\laplcomp_{\Psi_n}^{F_n} + P_{\Psi_n}^{F_n})^{1/2} \big) dt
			\\
			&
			=
			(s - 1) \cdots (s - 2k + 1)
			\int_{\epsilon}^{+\infty} t^{-2k + s - 1} \exp \big(-t (n^2 \cdot\laplcomp_{\Psi_n}^{F_n} + P_{\Psi_n}^{F_n})^{1/2} \big) dt
			\\
			&
			\phantom{= \,}
			+
			\sum_{l = 1}^{2n - 1} (s-1) \cdots (s - l + 1) \epsilon^{s - l} (n^2 \cdot\laplcomp_{\Psi_n}^{F_n} + P_{\Psi_n}^{F_n})^{k - l/2}
			\\
			&
			\qquad \qquad \qquad \qquad \qquad \qquad \qquad \qquad \qquad
			\cdot
			\exp \big(-\epsilon (n^2 \cdot\laplcomp_{\Psi_n}^{F_n} + P_{\Psi_n}^{F_n})^{1/2} \big).
		\end{aligned}
		\end{equation}
		Now, clearly Theorem \ref{thm_unif_bnd_eig} implies that there is $\mu > 0$ such that for any $n \in \nat^*$, we have
		\begin{equation}\label{spec_bound_unif_sss}
			\min \spec (n^2 \cdot\laplcomp_{\Psi_n}^{F_n} + P_{\Psi_n}^{F_n}) > \mu.
		\end{equation}
		We conclude by (\ref{spec_bound_unif_sss}) that for any $K \subset \comp$ there is $C > 0$ such that for any $s \in K$, $k \in \nat^*$:
		\begin{equation}\label{eq_thm_exp_12_aux_4}
			\Big\|
				\int_{\epsilon}^{+\infty} t^{-2k + s - 1} \exp \big(-t (n^2 \cdot\laplcomp_{\Psi_n}^{F_n} + P_{\Psi_n}^{F_n})^{1/2} \big) dt
			\Big\|_{L^2(\Psi_n, F_n)}^{0} 
			\leq 
			C.
		\end{equation}
		Now, by (\ref{eq_thm_exp_12_aux_1}) - (\ref{eq_thm_exp_12_aux_4}), we deduce Theorem \ref{thm_power_bnd_supp_disj}. 
	\end{proof}
	\par 
	Let's now prove Theorem \ref{thm_power_diff_bnd}, from which we borrow the notation for the following statement, which is the main ingredient in the proof of Theorem \ref{thm_power_diff_bnd}.
	\begin{thm}\label{thm_exp_12_bound_diff}
		There is $\epsilon > 0$ and $C > 0$ such that for any $t \in \comp$, $|t| < \epsilon$ and $n \in \nat^*$, we have
		\begin{equation}
			\Big\| 
			\exp \big( - t (n^2 \cdot\laplcomp_{\Psi_n}^{F_n, \perp})^{1/2} \big) - \sum_{\alpha \in I} \phi_{\alpha} \exp \big( -t  (n^2 \cdot  \laplcomp_{U_{\alpha, n}}^{F_n, \perp})^{1/2} \big) \psi_{\alpha}
			\Big\|_{L^2(\Psi_n, F_n)}^{0} \leq C,
		\end{equation}
		where $\norm{\cdot}_{L^2(\Psi_n, F_n)}^{0}$ is the operator norm.
	\end{thm}
	\begin{proof}
		The strategy of the proof is the same as in Theorem \ref{thm_exp_12_bound}, so we only highlight the main steps.
		We define $A_n^{\alpha}(t), B_n^{\alpha}(t) \in \enmr{{\rm{Map}}(V(\Psi_n), F_n)}$, $t \in \comp$ by the same formulas as $A_n^{\alpha}(t), B_n^{\alpha}(t)$, from (\ref{eq_exp_12_bound_aux_1}), only changing the Laplacian on $\Psi$ by the Laplacian on $U_{\alpha}$.
		\par 
		By repeating the proof of (\ref{eq_exp_12_bound_aux_7}), we see that there is $C > 0$ such that for any $t \in \comp$, $|t| \leq \frac{c}{2 e}$:
		\begin{equation}\label{eq_exp_12_bound_aux_7analogue}
			\Big\| A_n(t) - \sum_{\alpha \in I} \phi_{\alpha} A_n^{\alpha}(t) \psi_{\alpha} \Big\| _{L^2(\Psi_n, F_n)}^{0}
			\leq
			C \norm{\phi}_{L^{\infty}(\Psi)} \cdot \norm{\psi}_{L^{\infty}(\Psi)}.
		\end{equation}
		\par By using again positivity and identity (\ref{eq_exp_12_bound_aux_8}), we see by (\ref{eq_exp_12_bound_aux_7analogue}) that there is $C > 0$ such that for any $t \in \comp$, $|t| \leq \frac{c}{2 e}$, $\Re(t) \geq 0$, we have
		\begin{equation}\label{eq_exp_12_bound_aux_9analog}
			\Big\| B_n(t) - \sum_{\alpha \in I} \phi_{\alpha} B_n^{\alpha}(t) \psi_{\alpha} \Big\| _{L^2(\Psi_n, F_n)}^{0}
			\leq
			C \norm{\phi}_{L^{\infty}(\Psi)} \cdot \norm{\psi}_{L^{\infty}(\Psi)}.
		\end{equation}
		However, since the operators  $B_n(t)$, $B_n^{\alpha}(t)$ satisfy $B_n(t) = -B_n(-t)$, $B_n^{\alpha}(t) = -B_n^{\alpha}(-t)$, we see that (\ref{eq_exp_12_bound_aux_9analog}) hold without the restriction on $\Re(t)$. We conclude by (\ref{eq_exp_12_bound_aux_8}), (\ref{eq_exp_12_bound_aux_7analogue}) and (\ref{eq_exp_12_bound_aux_9analog}).
	\end{proof}
	Let's see how Theorem \ref{thm_exp_12_bound_diff} can be applied in the proof of  Theorem  \ref{thm_power_diff_bnd}.
	\begin{proof}[Proof of Theorem \ref{thm_power_diff_bnd}.]
		Our proof is a repetition of the proof of Theorem \ref{thm_power_bnd_supp_disj}, where we replace the use of Theorem \ref{thm_exp_12_bound} by the use of Theorem \ref{thm_exp_12_bound_diff}.
		Let's just highlight the main steps.
		\par 
		Let's fix $\epsilon > 0$ as in Theorem \ref{thm_exp_12_bound_diff}.
		By Theorem \ref{thm_exp_12_bound_diff} and Cauchy theorem, we see  that for any $l \in \nat$, there is $C > 0$ such that
		\begin{multline}\label{eq_thm_exp_12_dif_aux_5}
			\Big\| (n^2 \cdot\laplcomp_{\Psi_n}^{F_n, \perp})^{l/2} \exp \big(-t (n^2 \cdot\laplcomp_{\Psi_n}^{F_n, \perp})^{1/2} \big) 
			\\			
			-
			\sum_{\alpha \in I} \phi_{\alpha}
			\cdot
			(n^2 \cdot \laplcomp_{U_{\alpha, n}}^{F_n, \perp})^{l/2}  \exp \big( -t  (n^2 \cdot  \laplcomp_{U_{\alpha, n}}^{F_n, \perp})^{1/2} \big) 
			\cdot
			\psi_{\alpha}
			\Big\|_{L^2(\Psi_n, F_n)}^{0} \leq C.
		\end{multline}
		\par 
		Similarly to (\ref{eq_thm_exp_12_aux_2}), but by using Theorem \ref{thm_exp_12_bound_diff}, we conclude that for any compact $K \subset \{s \in \comp : \Re(s) > 0 \}$, and any $s \in K$, there is $C > 0$ such that
		\begin{multline}\label{eq_thm_exp_12_dif_aux_2}
			\Big\|
			\frac{1}{\Gamma(s)}
			\int_0^{\epsilon} t^{s-1} (n^2 \cdot\laplcomp_{\Psi_n}^{F_n, \perp})^{k} \exp \big(-t (n^2 \cdot\laplcomp_{\Psi_n}^{F_n, \perp})^{1/2} \big) dt
			\\
			-
			\frac{1}{\Gamma(s)}
			\sum_{\alpha \in I} \phi_{\alpha}
			\cdot
			\int_0^{\epsilon} t^{s-1} 
			(n^2 \cdot \laplcomp_{U_{\alpha, n}}^{F_n, \perp})^{k}  \exp \big( -t  (n^2 \cdot  \laplcomp_{U_{\alpha, n}}^{F_n, \perp})^{1/2} \big) 
			dt
			\cdot
			\psi_{\alpha}
			\Big\|_{L^2(\Psi_n, F_n)}^{0} 
			\leq 
			C.
		\end{multline}
		\par 
		Now, by (\ref{eq_thm_exp_12_aux_1}),  (\ref{eq_thm_exp_12_aux_3}),  (\ref{eq_thm_exp_12_aux_4}), applied for $\Psi_n$ and $U_{\alpha, n}$ and by (\ref{eq_thm_exp_12_dif_aux_5}), (\ref{eq_thm_exp_12_dif_aux_2}), we deduce Theorem \ref{thm_power_diff_bnd}. 
	\end{proof}

\section{Fourier analysis on square}\label{sect_four_an_sq}
	The goal of this section is to prove Theorem \ref{thm_explicit_calc}. The main idea is to use the discrete Fourier analysis on rectangle and then to recover the quantities from (\ref{eq_thm_explicit_calc}), such as integral of $\phi$ over the surface, integral of $\phi$ over the boundary and the value of $\phi$ the corner, through their expressions by Fourier coefficients.
	This section is organized as follows. In Section \ref{sect_spec_square} we recall the description of the spectrum and the eigenvectors of the Laplacian on a mesh graph, and we will study the multiplication operator in the basis of eigenvalues of the Laplacian.
	Then in Section \ref{sect_szego_thm}, by using those results and some simple but technical calculations, we prove Theorem \ref{thm_explicit_calc}.

	\subsection{Spectrum of the discrete Laplacian on square}\label{sect_spec_square}
		In this section we will recall the description of the spectrum and the eigenvectors of the Laplacian on a mesh graph.
		Then we study the multiplication operator in the basis of eigenvectors.
		\par 
		We fix $a, b \in \nat^*$ and construct mesh graphs $A_{an \times bn}$, $n \in \nat^*$ as $an \times bn$ subgraphs of $\integ^2$. 
		The spectrum $\spec(\laplcomp_{A_{an \times bn}})$ of their Laplacians is well-known and it can be obtained in a variety of different methods. 
		The eigenvectors $f_{i, j}^{n}$, $0 \leq i \leq an - 1$, $0 \leq j \leq bn - 1$,  of $\laplcomp_{A_{an \times bn}}$ are
		\begin{equation}\label{eq_defn_fij}
			f_{i, j}^{n}(k, l) = 
			 \cos \Big( \frac{2 \pi i}{2 a n} \Big( \frac{1}{2} + k \Big) \Big)
			 \cos \Big( \frac{2 \pi j}{2 b n} \Big( \frac{1}{2} + l \Big) \Big), 
			 \quad
			 0 \leq k \leq an - 1, 0 \leq l \leq bn - 1.
		\end{equation}
		The corresponding eigenvalues $\lambda^{n}_{i, j}$ of $n^2 \cdot \laplcomp_{A_{an \times bn}} + P_{A_{an \times bn}}$ are given by
		\begin{equation}\label{eq_eigenval_sq}
			\lambda^{n}_{i, j} = 
			\begin{cases} 
				\hfill 1, & \text{ for } i, j = 0, \\
				\hfill 4 n^2 \sin  \big( \frac{\pi i}{2 a n} \big)^2 + 4 n^2 \sin  \big( \frac{\pi j}{2 b n} \big)^2, & \text{ otherwise.}
 			\end{cases}
		\end{equation}
		Here $P_{A_{an \times bn}}$ is the orthogonal projection onto the constant functions.
		\par 
		Denote by $\delta_{i,j}$ the the Kronecker symbol.
		We verify  easily that
		\begin{equation}\label{eq_norm_fij_n}
			\big \| f_{i, j}^{n} \big \|_{L^2(A_{an \times bn})} = \frac{abn^2}{2^{\delta_{i, 0} + \delta_{j, 0}}}.
		\end{equation}
		\par 	
		Now let's fix a smooth function $\phi$, defined on the square $[0; a] \times [0; b]$. 
		For simplicity, suppose that it satisfies 
		\begin{equation}\label{eq_symm_phi}
			\phi(x, y) = \phi(a-x, y) = \phi(x, b- y), \qquad (x, y) \in [0; a] \times [0; b].
		\end{equation}
		\begin{sloppypar}
		Roughly, our goal now is to write down the multiplication by $\phi$ on $V(A_{an \times bn})$ in the basis of the eigenvalues of $\laplcomp_{A_{an \times bn}}$.
		The resulting formula will be huge, so we will only content ourselves with some remarks in this direction.
		\par We first remark that by (\ref{eq_symm_phi}), only the terms of the form $\cos(\cdot) \cdot \cos(\cdot)$ appear in the Fourier coefficients of $\phi$. Thus, there are $a_{i, j} \in \real$, $i \in \nat$, $j \in \nat$ such that the following expansion holds
		\begin{equation}\label{eq_phi_four_expansion}
			\phi(x, y) = \sum_{i, j = 0}^{+ \infty} a_{i, j} \cos \Big( \frac{2 \pi i x}{a} \Big) \cos \Big( \frac{2 \pi j y}{b} \Big).
		\end{equation}
		If in addition to (\ref{eq_symm_phi}), we require that the induced function $\tilde{\phi}$ on the torus $[0; a] \times [0; b]/_{(0, y) \sim (a, y), (x, 0) \sim (x, b)}$ is smooth, then by the classical bound on Fourier expansion of a smooth function, for any $c > 0$, there is $C > 0$ such that the following bound holds
		\begin{equation}\label{eq_asymp_bound_coeff}
			a_{i, j} \leq C (i + j + 1)^{-c}.
		\end{equation}
		\end{sloppypar}
		\par Recall that we view the graph $A_{an \times bn}$ injected in $[0; a] \times [0, b]$ as a subgraph of the nearest-neighbor graph on the vertices of $\frac{1 + \imun}{2n}  + \frac{1}{n} \integ^2$, which stays inside of $[0; a] \times [0, b]$.
		To simplify the notation, we denote the restriction of $\cos ( \frac{2 \pi i x}{a} ) \cos ( \frac{2 \pi j y}{b} )$ to $V(A_{an \times bn})$ by
		\begin{equation}\label{eq_gij_expl}
			g^{n}_{i, j}(r, s) := 
			\cos \Big( \frac{2 \pi i}{a n} \Big( \frac{1}{2} + r \Big) \Big)
			 \cos \Big( \frac{2 \pi j}{b n} \Big( \frac{1}{2} + s \Big) \Big),
		\end{equation}
		where $0 \leq r \leq an - 1$, $0 \leq s \leq bn - 1$.
		In particular, for $i, j \in \nat$, we have the following identities
		\begin{equation}\label{eq_gij_symm}
			g^{n}_{i, j} = - g^{n}_{i + an, j} = - g^{n}_{i, j + bn}.
		\end{equation}
		From (\ref{eq_gij_expl}), we see
		\begin{multline}
			\scal{g^{n}_{i, j} \cdot f^{n}_{k, l}}{f^{n}_{k, l}}_{L^2(V(A_{an \times bn}))}
			=
			\sum_{r = 0}^{an - 1} \sum_{s = 0}^{bn - 1}
			\cos \Big( \frac{2 \pi i}{a n} \Big( \frac{1}{2} + r \Big) \Big)
			\cos \Big( \frac{2 \pi j}{b n} \Big( \frac{1}{2} + s \Big) \Big)
			\\
			\cdot
			\cos \Big( \frac{2 \pi k}{2a n} \Big( \frac{1}{2} + r \Big) \Big)^2
			\cos \Big( \frac{2 \pi l}{2b n} \Big( \frac{1}{2} + s \Big) \Big)^2.
		\end{multline}
		It is rather easy to conclude that for any $0 \leq i, k \leq an - 1$, we have
		\begin{multline}\label{eq_sum_prod_cos}
			\sum_{r = 0}^{an - 1} 
			\cos \Big( \frac{2 \pi i}{a n} \Big( \frac{1}{2} + r \Big) \Big)
			\cdot
			\cos \Big( \frac{2 \pi k}{2a n} \Big( \frac{1}{2} + r \Big) \Big)^2
			\\
			=
			\begin{cases} 
				\hfill an \delta_{i, 0}, & \text{ for } k = 0, \\
				\hfill \frac{an}{2},  & \text{ for } i = 0, k \neq 0, \\
				\hfill \frac{an}{4} \delta_{i, k} - \frac{an}{4} \delta_{i, an - k}, & \text{ for } i \neq 0, k \neq 0.\\
 			\end{cases}
		\end{multline}
		From this and (\ref{eq_gij_expl}), we see that the multiplication by $g^{n}_{i, j}$ is “almost diagonal" unless $i$ or $j$ is equal to $0$ modulo $an$ or $bn$ respectively.
		We will use this fact extensively in the next section.

	\subsection{Szegő-type theorem, a proof of Theorem \ref{thm_explicit_calc}}\label{sect_szego_thm}
		The main goal of this section is to prove Theorem \ref{thm_explicit_calc}.
		As in Section \ref{sect_spec_square}, we fix $a, b \in \nat^*$ and construct a family of graphs $A_{an \times bn}$, $n \in \nat^*$. 
		We fix a smooth function $\phi : [0; a] \times [0; b] \to \real$ satisfying (\ref{eq_symm_phi}) and such that the induced function $\tilde{\phi}$ on the torus $[0; a] \times [0; b]/_{(0, y) \sim (a, y), (x, 0) \sim (x, b)}$ is smooth. 
		In particular, the Fourier expansion of $\phi$ is of the form (\ref{eq_phi_four_expansion}), and the coefficients $a_{i, j}$, $i, j \in \nat$ satisfy the asymptotic expansion (\ref{eq_asymp_bound_coeff}).
		\par 
		Recall that the functions $f_{i, j}^{n}, g_{i, j}^{n} \in {\rm{Map}}(V(A_{an \times bn}), \real)$ were defined in (\ref{eq_defn_fij}) and (\ref{eq_gij_expl}) respectively.
		By (\ref{eq_phi_four_expansion}), (\ref{eq_gij_expl}) and (\ref{eq_gij_symm}), we have
		\begin{equation}\label{eq_tr_phi_exp_first}
		\begin{aligned}
			& \tr{\phi \cdot \log (n^2 \cdot \laplcomp_{A_{an \times bn}}^{\perp})}
			=
			a_{0, 0} \sum_{i = 0}^{an - 1} \sum_{j = 0}^{bn - 1} \log \big( \lambda_{i, j}^{n} \big)
			\\
			& 
			\qquad \qquad \qquad
			+
			\sum_{i = 1}^{an} \Big( \sum_{k = 0}^{+\infty} (-1)^k a_{i + kan, 0} \Big) \tr{g_{i, 0}^{n} \cdot \log (n^2 \cdot \laplcomp_{A_{an \times bn}}^{\perp})}
			\\
			& 
			\qquad \qquad \qquad
			+
			\sum_{j = 1}^{bn} \Big( \sum_{l = 0}^{+\infty} (-1)^l a_{0, j + lbn} \Big) \tr{g_{0, j}^{n} \cdot \log (n^2 \cdot \laplcomp_{A_{an \times bn}}^{\perp})}
			\\
			& 
			\qquad \qquad \qquad
			+
			\sum_{i = 1}^{an} \sum_{j = 1}^{bn} \Big( \sum_{k, l = 0}^{+\infty} (-1)^{k + l} a_{i + kan, j + lbn} \Big) \tr{g_{i, j}^{n} \cdot \log (n^2 \cdot \laplcomp_{A_{an \times bn}}^{\perp})}.
		\end{aligned}
		\end{equation}
		We will now state the two main propositions of this section. Then we will explain how they are useful in proving Theorem \ref{thm_explicit_calc} and then we will prove them.
		\begin{prop}\label{prop_trgi_first}
			For any $a, b \in \nat^*$ and $\phi$ as in the beginning of this section, there are $c_{1, \phi}^{a, b}, c_{2, \phi}^{a, b} \in \real$, such that as $n \to \infty$, the following asymptotic expansions hold
			\begin{multline}\label{eq_prop_trgi_first}
				\sum_{i = 1}^{an} \Big( \sum_{k = 0}^{+\infty} (-1)^k a_{i + kan, 0} \Big) \tr{g_{i, 0}^{n} \cdot \log (n^2 \cdot \laplcomp_{A_{an \times bn}}^{\perp})} 
				\\
				=
				-
				\Big(
					bn \log(1 + \sqrt{2})
					+
					\frac{\log(n)}{2}
				\Big)
				\cdot
				\sum_{i = 1}^{+\infty}		
				a_{i, 0}	
				+
				c_{1, \phi}^{a, b}
				+
				o(n^{-1/3}),
			\end{multline}
			\vspace*{-0.8cm}
			\begin{multline}\label{eq_prop_trgi_second}
				\sum_{j = 1}^{bn} \Big( \sum_{l = 0}^{+\infty} (-1)^l a_{0, j + lbn} \Big) \tr{g_{0, j}^{n} \cdot \log (n^2 \cdot \laplcomp_{A_{an \times bn}}^{\perp})} 
				\\
				=
				-
				\Big(
					an \log(1 + \sqrt{2})
					+
					\frac{\log(n)}{2}
				\Big)
				\cdot
				\sum_{j = 1}^{+\infty}		
				a_{0, j}	
				+
				c_{2, \phi}^{a, b}
				+
				o(n^{-1/3}).
			\end{multline}
		\end{prop}
		\begin{prop}\label{prop_trgi_second}
			For any $a, b \in \nat^*$ and $\phi$ as in the beginning of this section, there is $c_{3, \phi}^{a, b} \in \real$ such that as $n \to \infty$, the following asymptotic expansion holds
			\begin{multline}\label{eq_prop_gij_nonequ_easy}
				\sum_{i = 1}^{an} 
				\sum_{j = 1}^{bn} 
				\Big( \sum_{k, l = 0}^{+\infty} (-1)^{k + l} a_{i + kan, j + lbn} \Big) \tr{g_{i, j}^{n} \cdot \log (n^2 \cdot \laplcomp_{A_{an \times bn}}^{\perp})}
				\\
				=
				-\frac{\log(n)}{2}
				\cdot
				\sum_{i, j= 1}^{+\infty}		
				a_{i, j}	
				+
				c_{3, \phi}^{a, b}
				+
				o(n^{-1/3}).
			\end{multline}
		\end{prop}
		\begin{proof}[Proof of Theorem \ref{thm_explicit_calc}]
			By averaging with respect to the symmetry group of rectangle, we see that we can suppose that $\phi$ satisfies (\ref{eq_symm_phi}).
			\par 
			By Propositions \ref{prop_trgi_first}, \ref{prop_trgi_second}, Remark \ref{rem_main_thm}a) and (\ref{eq_tr_phi_exp_first}), we deduce that for $c([0, a] \times [0, b], \phi) \in \real$, defined as 
			\begin{equation}
				c([0, a] \times [0, b], \phi)  := c_{1, \phi}^{a, b} + c_{2, \phi}^{a, b}  + c_{3, \phi}^{a, b} + a_{0, 0} \cdot 
				\Big(
					\log \det{}' \Delta_{[0, a] \times [0, b]}
					-
					\frac{\log(2)}{4}
				\Big)
			\end{equation}
			the following expansion, as $n \to \infty$, holds
			\begin{multline}\label{eq_pf_thm_expl_aux1}
				\tr{ \phi \cdot \log \big( n^2\cdot \laplcomp_{U_{n}}^{\perp} \big)}
				= 
				\Big(
				2 ab n^2 \log(n) + 
				\frac{4G}{\pi} abn^2
				\Big)
				a_{0, 0}
				+ 
				\log(\sqrt{2} - 1) n  
				\Big( 
					b \sum_{i = 0}^{+ \infty} a_{i, 0}
					+
					a \sum_{j = 0}^{+ \infty} a_{0, j}
				\Big)				
				\\
				-
				\frac{\log(n)}{2} \sum_{i = 0}^{+ \infty} \sum_{j = 0}^{+ \infty} a_{i, j}
				+ 
				c([0, a] \times [0, b], \phi)  + o(1).
			\end{multline}
			Now, Theorem \ref{thm_explicit_calc} follows from (\ref{eq_pf_thm_expl_aux1}) and the following identities
			\begin{equation}
			\begin{aligned}
				&
				a_{0, 0} 
				= 
				\frac{1}{ab} \int_{0}^{a} \int_{0}^{b} \phi(x, y) dx dy,
				\qquad
				&&
				\sum_{i = 0}^{+ \infty} a_{i, 0}
				=
				\frac{1}{b} \int_{0}^{b} \phi(0, y) dy ,
				\\
				&
				\sum_{j = 0}^{+ \infty} a_{0, j}
				=
				\frac{1}{a} \int_{0}^{a} \phi(x, 0) dy ,
				\qquad
				&&
				\sum_{i = 0}^{+ \infty} \sum_{j = 0}^{+ \infty} a_{i, j}
				=
				\phi(0, 0).
			\end{aligned}
			\end{equation}
			which follow trivially from (\ref{eq_phi_four_expansion}).
		\end{proof}
		\par
		From now on till the end of this section we will be proving Propositions \ref{prop_trgi_first}, \ref{prop_trgi_second}.
		\begin{prop}\label{prod_prod_sin}
			For any $x \in \real$ and $m \in \nat^*$, we have
			\begin{multline}
				\prod_{j = 0}^{m-1} \Big( \sin \Big( \frac{\pi j }{2m} \Big)^2 + x^2 \Big)
				=
				\frac{|x| 
				\cdot
				(1 + x^2)^{-1/2}}{2^{2m}}
				\Big|
					\big( 
						 \sqrt{1 + x^2} + x
					\big)^{2m}
					-
					1
				\Big|
				\\
				\cdot
				\Big|
					\big( 
						\sqrt{1 + x^2} - x
					\big)^{2m}
					-
					1
				\Big|.
			\end{multline}
		\end{prop}
		\begin{proof}
			First, we have
			\begin{equation}\label{eq_prod_sin_aux_1}
				\Big( \sin \Big( \frac{\pi j }{2m} \Big)^2 + x^2 \Big)
				=
				\Big( \sin \Big( \frac{\pi j }{2m} \Big) + \imun x \Big)
				\Big( \sin \Big( \frac{\pi j }{2m} \Big) - \imun x \Big).
			\end{equation}
			Now, by Euleur's identity, we have
			\begin{equation}\label{eq_prod_sin_aux_2}
				\sin \Big( \frac{\pi j }{2m} \Big) \pm \imun x 
				= 
				\frac{e^{-\frac{\pi j \imun}{2m}}}{2 \imun} 
				\Big(
					e^{\frac{2 \pi j \imun}{2m}}
					\mp
					2x
					e^{\frac{\pi j \imun}{2m}}
					-
					1
				\Big).
			\end{equation}
			By resolving quadratic equation, we deduce
			\begin{equation}\label{eq_prod_sin_aux_3}
				e^{\frac{2 \pi j \imun}{2m}}
				\mp
				2x
				e^{\frac{\pi j \imun}{2m}}
				-
				1
				=
				\Big(
					e^{\frac{\pi j \imun}{2m}} 
					- 
					\big( 
						\pm x + \sqrt{1 + x^2}
					\big)
				\Big)
				\cdot
				\Big(
					e^{\frac{\pi j \imun}{2m}} 
					- 
					\big( 
						\pm x - \sqrt{1 + x^2}
					\big)
				\Big).
			\end{equation}
			We also clearly have 
			\begin{equation}\label{eq_prod_sin_aux_332a}
				\prod_{j = 0}^{m-1} \Big( \sin \Big( \frac{\pi j }{2m} \Big)^2 + x^2 \Big) 
				= 
				|x| 
				\cdot
				(1 + x^2)^{-1/2}
				\cdot
				\prod_{j = 0}^{2m-1} \Big( \sin \Big( \frac{\pi j }{2m} \Big)^2 + x^2 \Big)^{1/2}.
			\end{equation}
			By (\ref{eq_prod_sin_aux_1}), (\ref{eq_prod_sin_aux_2}),  (\ref{eq_prod_sin_aux_3}) and  (\ref{eq_prod_sin_aux_332a}), we conclude
			\begin{multline}\label{eq_prod_sin_aux_4}
				\prod_{j = 0}^{m-1} \Big( \sin \Big( \frac{\pi j }{2m} \Big)^2 + x^2 \Big)
				=
				\frac{|x| 
				\cdot
				(1 + x^2)^{-1/2}}{2^{2m}} 
				\\
				\cdot
				\prod_{\alpha, \beta \in \{-1, 1\}}				
				\prod_{j = 0}^{2m-1}
				\Big|
					\big( 
						\alpha x + \beta \sqrt{1 + x^2}
					\big)
					- 
					e^{\frac{\pi j \imun}{2m}} 
				\Big|^{1/2}.
			\end{multline}
			But by the definition of the roots of unity, we have
			\begin{equation}\label{eq_prod_sin_aux_5}
				\prod_{j = 0}^{2m-1}
				\Big(
					\big( 
						\alpha x + \beta \sqrt{1 + x^2}
					\big)
					-
					e^{\frac{\pi j \imun}{m}} 
				\Big)
				=
					\big( 
						\alpha x + \beta \sqrt{1 + x^2}
					\big)^{2m}
					-
					1.
			\end{equation}
			From (\ref{eq_prod_sin_aux_4}) and (\ref{eq_prod_sin_aux_5}), we easily conclude.
		\end{proof}
		\begin{prop}
			There is $C > 0$ such that for any $n \in \nat^*$, $1 \leq i \leq  n^{1/3}$,  $1 \leq j \leq  n^{1/3}$, we have
			\begin{multline}\label{eq_prod_sin_aux_20}
				\bigg|
				\sum_{l = 0}^{bn - 1} 
				\Big( 
				\log(\lambda_{i, l}^{n})
				-
				\log(\lambda_{an-i, l}^{n})
				\Big)
				+
				2bn \log(1 + \sqrt{2})
				+
				\log(n)
				\\
				-
				\log(e^{\pi i \frac{b}{a}} - 1)
				-
				\log(1 - e^{- \pi i \frac{b}{a}})
				-
				\log
				\Big(
					\frac{\pi i}{2a}
				\Big)
				-
				\frac{\log(2)}{2}
				\bigg|
				\leq
				\frac{C}{n^{1/3}},
			\end{multline}
			\vspace*{-0.7cm}
			\begin{multline}\label{eq_prod_sin_aux_2020}
				\bigg|
				\sum_{k = 0}^{an - 1} 
				\Big( 
				\log(\lambda_{k, j}^{n})
				-
				\log(\lambda_{k, bn - j}^{n})
				\Big)
				+
				2an \log(1 + \sqrt{2})
				+
				\log(n)
				\\
				-
				\log(e^{\pi i \frac{a}{b}} - 1)
				-
				\log(1 - e^{- \pi i \frac{a}{b}})
				-
				\log
				\Big(
					\frac{\pi i}{2b}
				\Big)
				-
				\frac{\log(2)}{2}
				\bigg|
				\leq
				\frac{C}{n^{1/3}}.
			\end{multline}
		\end{prop}		 
		\begin{proof}
			Clearly, the statement (\ref{eq_prod_sin_aux_2020}) is completely analogic to (\ref{eq_prod_sin_aux_20}), so we will only concentrate on proving (\ref{eq_prod_sin_aux_20}).
			From Proposition \ref{prod_prod_sin} and (\ref{eq_eigenval_sq}), for any $n \in \nat^*$, $1 \leq i \leq  n^{1/3}$,  we have
			\begin{multline}\label{eq_prod_sin_aux_7}
				\sum_{l = 0}^{bn - 1} 
				\log(\lambda_{i, l}^{n})
				=
				2 bn \log(n) 
				+
				\log \Big( \sin  \big( \frac{\pi i}{2 a n} \big) \Big) 
				\\
				- 
				\frac{1}{2} 
				\log \Big(
					1 + \sin  \big( \frac{\pi i}{2 a n} \big)^2
				\Big)
				+
				\log 
				\Big(
					\Big( \sqrt{ 1 + \sin  \big( \frac{\pi i}{2 a n} \big)^2} + \sin  \big( \frac{\pi i}{2 a n} \big)  \Big)^{2bn} - 1
				\Big)
				\\
				+
				\log 
				\Big(
					1 - \Big(  \sqrt{ 1 + \sin  \big( \frac{\pi i}{2 a n} \big)^2} - \sin  \big( \frac{\pi i}{2 a n} \big) \Big)^{2bn}
				\Big).
			\end{multline}
			We now study the asymptotic expansion of the right-hand side of (\ref{eq_prod_sin_aux_7}).
			By Taylor expansion, there is $C > 0$ such that for any $n \in \nat^*$, $1 \leq i \leq n^{1/3}$, we have
			\begin{equation}\label{eq_prod_sin_aux_8}
				\Big| \sin  \big( \frac{\pi i}{2 a n} \big) - \frac{\pi i}{2 a n} \Big| \leq \frac{C}{n^{2}}.
			\end{equation}
			Similarly, there is $C > 0$ such that for any $n \in \nat^*$, $1 \leq i \leq  n^{1/3}$, we have
			\begin{equation}\label{eq_prod_sin_aux_9}
				\Big| \sqrt{ 1 + \sin  \big( \frac{\pi i}{2 a n} \big)^2} - 1 \Big| \leq \frac{C}{n^{4/3}}.
			\end{equation}
			From (\ref{eq_prod_sin_aux_8}) and (\ref{eq_prod_sin_aux_9}), there is $C > 0$ such that for any $n \in \nat^*$, $1 \leq i \leq  n^{1/3}$, we have
			\begin{equation}\label{eq_prod_sin_aux_10}
				\bigg|
				\Big( \sqrt{ 1 + \sin  \big( \frac{\pi i}{2 a n} \big)^2} + \sin  \big( \frac{\pi i}{2 a n} \big) \Big)^{2bn}
				-
				e^{\pi i \frac{b}{a}}
				\bigg|
				\leq
				\frac{C}{n^{1/3}}e^{\pi i \frac{b}{a}}.
			\end{equation}
			As a consequence of (\ref{eq_prod_sin_aux_10}), there is $C > 0$ such that for any $n \in \nat^*$, $1 \leq i \leq  n^{1/3}$, we have
			\begin{equation}\label{eq_prod_sin_aux_11}
				\bigg|
				\log 
				\Big(
					\Big( \sqrt{ 1 + \sin  \big( \frac{\pi i}{2 a n} \big)^2} + \sin  \big( \frac{\pi i}{2 a n} \big) \Big)^{2bn} - 1
				\Big)
				-
				\log(e^{\pi i \frac{b}{a}} - 1)
				\bigg|
				\leq
				\frac{C}{n^{1/3}}.
			\end{equation}
			Similarly, there is $C > 0$, such that for any $n \in \nat^*$, $1 \leq i \leq  n^{1/3}$, we have
			\begin{equation}\label{eq_prod_sin_aux_12}
				\bigg|
				\log 
				\Big(
					1 - \Big(  \sqrt{ 1 + \sin  \big( \frac{\pi i}{2 a n} \big)^2} - \sin  \big( \frac{\pi i}{2 a n} \big) \Big)^{2bn}
				\Big)
				-
				\log(1 - e^{- \pi i \frac{b}{a}})
				\bigg|
				\leq
				\frac{C}{n^{1/3}}.
			\end{equation}
			From (\ref{eq_prod_sin_aux_7}), (\ref{eq_prod_sin_aux_11}) and (\ref{eq_prod_sin_aux_12}), there is $C > 0$ such that for any $n \in \nat^*$, $1 \leq i \leq  n^{1/3}$, we have
			\begin{equation}\label{eq_prod_sin_aux_13}
				\bigg|
				\sum_{l = 0}^{bn - 1} 
				\log(\lambda_{i, l}^{n})
				-
				(2bn - 1) \log(n) 
				-
				\log(e^{\pi i \frac{b}{a}} - 1)
				-
				\log(1 - e^{- \pi i \frac{b}{a}})
				-
				\log 
				\Big(
					\frac{\pi i}{2 a}
				\Big)
				\bigg|
				\leq
				\frac{C}{n^{1/3}}.
			\end{equation}
			\par Now, from (\ref{eq_eigenval_sq}) and the identity $\sin(\frac{\pi}{2} - x) = \cos(x)$, we have
			\begin{multline}\label{eq_prod_sin_aux_14}
				\sum_{l = 0}^{bn - 1} 
				\log(\lambda_{an - i, l}^{n})
				=
				2 bn \log(n) 
				+
				\log \Big(
					\cos  \big( \frac{\pi i}{2 a n} \big)
				\Big)
				-
				\frac{1}{2}
				\log \Big(
					1 + \cos  \big( \frac{\pi i}{2 a n} \big)^2
				\Big)
				\\
				+
				\log 
				\Big(
					\Big( \sqrt{ 1 + \cos  \big( \frac{\pi i}{2 a n} \big)^2} +  \cos  \big( \frac{\pi i}{2 a n} \big) \Big)^{2bn} - 1
				\Big)
				\\
				+
				\log 
				\Big(
					1 - \Big(  \sqrt{ 1 + \cos  \big( \frac{\pi i}{2 a n} \big)^2} - \cos  \big( \frac{\pi i}{2 a n} \big) \Big)^{2bn}
				\Big).
			\end{multline}
			We will now study the asymptotic expansion of (\ref{eq_prod_sin_aux_14}).
			By Taylor expansion, there is $C > 0$ such that for any $n \in \nat^*$, $i \leq n^{1/3}$, we have
			\begin{equation}\label{eq_prod_sin_aux_15}
				\Big| \cos  \big( \frac{\pi i}{2 a n} \big) - 1 \Big| \leq \frac{C}{n^{4/3}}.
			\end{equation}
			Similarly, there is $C > 0$ such that for any $n \in \nat^*$, $i \leq  n^{1/3}$, we have
			\begin{equation}\label{eq_prod_sin_aux_16}
				\Big| \sqrt{ 1 + \cos  \big( \frac{\pi i}{2 a n} \big)^2} - \sqrt{2} \Big| \leq \frac{C}{n^{4/3}}.
			\end{equation}
			From (\ref{eq_prod_sin_aux_15}) and (\ref{eq_prod_sin_aux_16}), we see that there is $C > 0$ such that for any $n \in \nat^*$, $i \leq  n^{1/3}$, we have
			\begin{equation}\label{eq_prod_sin_aux_17}
				\bigg|
				\log 
				\Big(
					\Big( \sqrt{ 1 + \cos  \big( \frac{\pi i}{2 a n} \big)^2} + \cos  \big( \frac{\pi i}{2 a n} \big)  \Big)^{2bn} - 1
				\Big)
				-
				2bn \log(1 + \sqrt{2})
				\bigg|
				\leq
				\frac{C}{n^{1/3}}.
			\end{equation}
			Now, due to the fact that $|\sqrt{2} - 1| < 1$, we conclude that there is $C > 0$ such that for any $n \in \nat^*$, $i \leq  n^{1/3}$, we have
			\begin{equation}\label{eq_prod_sin_aux_18}
				\bigg|
				\log 
				\Big(
					1- \Big( \sqrt{ 1 + \cos  \big( \frac{\pi i}{2 a n} \big)^2} - \cos  \big( \frac{\pi i}{2 a n} \big)  \Big)^{2bn}
				\Big)
				\bigg|
				\leq
				\frac{C}{n^{1/3}}.
			\end{equation}
			As a conclusion from (\ref{eq_prod_sin_aux_14}), (\ref{eq_prod_sin_aux_17}) and (\ref{eq_prod_sin_aux_18}), we see that there is $C > 0$ such that for any $n \in \nat^*$, $i \leq  n^{1/3}$, we have
			\begin{equation}\label{eq_prod_sin_aux_19}
				\bigg|
				\sum_{l = 0}^{bn - 1} 
				\log(\lambda_{an-i, l}^{n})
				-
				2bn \log(n) 
				-
				2bn \log(1 + \sqrt{2})
				+
				\frac{\log(2)}{2}
				\bigg|
				\leq
				\frac{C}{n^{1/3}}.
			\end{equation}
			By combining (\ref{eq_prod_sin_aux_13}) and   (\ref{eq_prod_sin_aux_19}), we deduce (\ref{eq_prod_sin_aux_20}).
		\end{proof}
		Now we are finally ready to give 
		\begin{proof}[Proof of Proposition \ref{prop_trgi_first}]
			Clearly, the statements (\ref{eq_prop_trgi_first}) and  (\ref{eq_prop_trgi_second}) are analogical, so we will only concentrate on the proof of the former one.
			\par 
			First, by (\ref{eq_spec_bnd_triv}) and the fact that $|g_{i, j}^{n}| < 1$, for any $0 \leq i \leq an - 1$ and any $n \in \nat^*$, we have
			\begin{equation}\label{eq_prop_trgi_bnd_triv_0}
				\tr{g_{i, 0}^{n} \cdot \log (n^2 \cdot \laplcomp_{A_{an \times bn}}^{\perp})} \leq ab n^2 \log(8n^2).
			\end{equation}
			By (\ref{eq_asymp_bound_coeff}) and (\ref{eq_prop_trgi_bnd_triv_0}), we see that for any $c > 0$ there is $C > 0$ such that for any $n \in \nat^*$, the following bound holds
			\begin{multline}\label{eq_prop_trgi_bnd_triv_1}
				\bigg|
					\sum_{i = 1}^{an} \Big( \sum_{k = 0}^{+\infty} (-1)^k a_{i + kan, 0} \Big) \tr{g_{i, 0}^{n} \cdot \log (n^2 \cdot \laplcomp_{A_{an \times bn}}^{\perp})} 
					\\
					-
					\sum_{i = 1}^{n^{1/3}} a_{i, 0} \tr{g_{i, 0}^{n} \cdot \log (n^2 \cdot \laplcomp_{A_{an \times bn}}^{\perp})}
				\bigg|
				\leq
				\frac{C}{n^c}.
			\end{multline}
			\par By the definition of trace and (\ref{eq_eigenval_sq}), we have
			\begin{equation}\label{eq_exp_tr_four}
				\tr{g_{i, 0}^{n} \cdot \log (n^2 \cdot \laplcomp_{A_{an \times bn}}^{\perp})}
				=
				\sum_{k = 0}^{an - 1} \sum_{l = 0}^{bn - 1} 
				\frac{ \log(\lambda_{k, l}^{n})}{\| f^{n}_{k, l} \|_{L^2(A_{an \times bn})}^{2}}
				\scal{g^{n}_{i, 0} \cdot f^{n}_{k, l}}{f^{n}_{k, l}}_{L^2(A_{an \times bn})}.
			\end{equation}
			\par By (\ref{eq_sum_prod_cos}), it is easy to conclude that for $1 \leq i \leq an - 1$, the only non-vanishing terms of the sum in the right-hand side of (\ref{eq_exp_tr_four}) are those which correspond to $k = i$ or $k = an - i$.
			We study their contributions to the sum (\ref{eq_exp_tr_four}) separately.
			\par 
			By (\ref{eq_norm_fij_n}), (\ref{eq_sum_prod_cos}), the contribution of the terms corresponding to $k = i$ is given by
			\begin{equation}\label{eq_exp_tr_four_aux1}
				\sum_{l = 0}^{bn - 1} 
				\frac{\log(\lambda_{i, l}^{n})}{\| f^{n}_{i, l} \|_{L^2(A_{an \times bn})}^{2}}
				\scal{g^{n}_{i, 0} \cdot f^{n}_{i, l}}{f^{n}_{i, l}}_{L^2(A_{an \times bn})}
				=
				\frac{1}{2}
				\sum_{l = 0}^{bn - 1} 
				\log(\lambda_{i, l}^{n}).
			\end{equation}
			Similarly, the  contribution of the terms corresponding to  $k = an - i$ is given by
			\begin{equation}\label{eq_exp_tr_four_aux2}
				\sum_{l = 0}^{bn - 1} 
				\frac{\log(\lambda_{an-i, l}^{n})}{\| f^{n}_{an-i, l} \|_{L^2(A_{an \times bn})}^{2}}
				\scal{g^{n}_{i, 0} \cdot f^{n}_{an-i, l}}{f^{n}_{an-i, l}}_{L^2(A_{an \times bn})}
				=
				-
				\frac{1}{2}
				\sum_{l = 0}^{bn - 1} 
				\log(\lambda_{an-i, l}^{n}).
			\end{equation}
			By (\ref{eq_exp_tr_four}), (\ref{eq_exp_tr_four_aux1}) and (\ref{eq_exp_tr_four_aux2}), we conclude that for $1 \leq i \leq an - 1$, we have 
			\begin{equation}\label{eq_exp_tr_final}
				\tr{g_{i, 0}^{n} \cdot \log (n^2 \cdot \laplcomp_{A_{an \times bn}}^{\perp})}
				=
				\frac{1}{2}
				\sum_{l = 0}^{bn - 1} 
				\Big( 
				\log(\lambda_{i, l}^{n})
				-
				\log(\lambda_{an-i, l}^{n})
				\Big).
			\end{equation}
			Now, by (\ref{eq_prod_sin_aux_20}), (\ref{eq_prop_trgi_bnd_triv_1}) and (\ref{eq_exp_tr_final}), we see that (\ref{eq_prop_trgi_first}) holds for
			\begin{equation}
				c_{1, \phi}^{a, b} := 
				\frac{1}{2}
				\sum_{i = 1}^{\infty} 
				a_{i, 0} 
				\Big(
					\log(e^{\pi i \frac{b}{a}} - 1)
					+
					\log(1 - e^{- \pi i \frac{b}{a}})
					+
					\log
					\Big(
						\frac{\pi i}{2a}
					\Big)
					+
					\frac{\log(2)}{2}
				\Big),
			\end{equation}
			which converges by (\ref{eq_asymp_bound_coeff}), for any $\phi$ as in the statement of Proposition \ref{prop_trgi_first}.
		\end{proof}
		\begin{proof}[Proof of Proposition \ref{prop_trgi_second}.]
			First, similarly to (\ref{eq_prop_trgi_bnd_triv_1}), we see that for any $c > 0$ there is $C > 0$ such that for any $n \in \nat^*$, the following bound holds
			\begin{multline}\label{eq_prop_trgi_bnd_triv_1121}
				\bigg|
					\sum_{i = 1}^{an} \sum_{j = 1}^{bn} \Big( \sum_{k, l = 0}^{+\infty} (-1)^{k + l} a_{i + kan, j + lbn} \Big) \tr{g_{i, j}^{n} \cdot \log (n^2 \cdot \laplcomp_{A_{an \times bn}}^{\perp})}
					\\
					-
					\sum_{i, j = 1}^{n^{1/3}} a_{i, j} \tr{g_{i, j}^{n} \cdot \log (n^2 \cdot \laplcomp_{A_{an \times bn}}^{\perp})}
				\bigg|
				\leq
				\frac{C}{n^c}.
			\end{multline}
			By (\ref{eq_sum_prod_cos}), for any $1 \leq i \leq an - 1$, $1 \leq j \leq bn - 1$, we have
			\begin{equation}\label{eq_gij_exp_non0}
				\tr{g_{i, j}^{n} \cdot \log (n^2 \cdot \laplcomp_{A_{an \times bn}}^{\perp})}
				=
				\frac{1}{4}
				\Big(
					\log(\lambda_{i, j}^{n})
					-
					\log(\lambda_{an - i, j}^{n})
					-
					\log(\lambda_{i, bn - j}^{n})
					+
					\log(\lambda_{an - i, bn - j}^{n})
				\Big).
			\end{equation}
			By (\ref{eq_eigenval_sq}), we see that there is $C > 0$ such that for any $n \in \nat^*$ and $1 \leq i, j \leq n^{1/3}$, we get
			\begin{multline}\label{eq_gij_exp_non0_aux_1}
				\Big|
					\log(\lambda_{i, j}^{n})
					-
					\log(\lambda_{an - i, j}^{n})
					-
					\log(\lambda_{i, bn - j}^{n})
					+
					\log(\lambda_{an - i, bn - j}^{n})
				\\
				-
				\log 
				\Big( 
					\frac{\pi^2 i^2}{a^2} + \frac{\pi^2 j^2}{b^2}
				\Big)
				+
				\log(2)
				+
				2
				\log ( n )
				\Big|
				\leq
				\frac{C}{n^{1/3}}.
			\end{multline}
			By (\ref{eq_prop_trgi_bnd_triv_1121}), (\ref{eq_gij_exp_non0}) and (\ref{eq_gij_exp_non0_aux_1}), Proposition \ref{prop_trgi_second} holds for $c_{3, \phi}$, defined as
			\begin{equation}\label{eq_gij_exp_non0_aux_2}
				c_{3, \phi}^{a, b} := 
				\frac{1}{4}
				\sum_{i, j = 1}^{\infty} 
				a_{i, j} 
				\Big(
					\log 
					\Big( 
						\frac{\pi^2 i^2}{a^2} + \frac{\pi^2 j^2}{b^2}
					\Big)
					-
					\log(2)
				\Big),
			\end{equation}
			which converges by (\ref{eq_asymp_bound_coeff}), for any $\phi$ as in the statement of Proposition \ref{prop_trgi_first}.
		\end{proof}

\appendix
\section{Appendix: heat kernel and analytic torsion}
\subsection{Small-time asymptotic expansion of the heat kernel}\label{sect_hk_sm_t_exp}
 		\par In this section we will prove Proposition \ref{prop_hk_expansion}. The proof is done by reducing the small-time asymptotic expansion of the heat kernel on the general flat surface to a number of model cases, for which the explicit calculations are possible.
 		\par 
 		First, let's prove that for any $l \in \nat$, $\epsilon > 0$ there are $c, C > 0$ such that for any $x, y \in \Psi$ satisfying $\dist_{\Psi}(\{x, y\}, {\rm{Con}}(\Psi) \cup  {\rm{Ang}}(\Psi)) > \epsilon$ and $\dist_{\Psi}(x, y) > \epsilon$, and any $0 < t < 1$, we have
 		\begin{equation}\label{eq_hk_exp_decay}
 			\Big|  \nabla^l \exp(-t \laplcomp_{\Psi}^{F})(x, y) \Big| \leq C\exp \Big(- \frac{c}{t} \Big).
 		\end{equation}
 		The estimate (\ref{eq_hk_exp_decay}) can be proven by a variety of different methods. We do it by using finite propagation speed of solutions of hyperbolic equations and interior elliptic estimates.
	\par More precisely, for $r > 0$, we introduce smooth even functions (cf. \cite[(4.2.11)]{MaHol})
	\begin{equation}\label{def_kth_lth}
	\begin{aligned}
		& K_{t, r}(a) = \int_{- \infty}^{+ \infty} \exp(\imun v \sqrt{2t} a) \exp \Big( -\frac{v^2}{2} \Big) \Big(1 - \psi\Big( \frac{\sqrt{2t} v}{r} \Big)\Big) \frac{d v}{\sqrt{2 \pi}}, \\
		& G_{t, r}(a)  = \int_{- \infty}^{+ \infty} \exp(\imun v \sqrt{2t} a) \exp \Big( -\frac{v^2}{2} \Big) \psi\Big( \frac{\sqrt{2t} v}{r} \Big) \frac{d v}{\sqrt{2 \pi}},
	\end{aligned}
	\end{equation}
	where $\psi: \real \to [0,1]$ is a cut-off function satisfying
	\begin{equation}\label{eq_defn_psi}
		\psi(u) = 
		\begin{cases} 
      				\hfill 1 & \text{ for } |u| < 1/2, \\
      				\hfill 0  & \text{ for } |u| > 1. \\
 				\end{cases}
	\end{equation}
	Let $\widetilde{K}_{t,r}, \widetilde{G}_{t,r} : \real_+ \to \real$ be the smooth functions given by $\widetilde{K}_{t,r}(a^2) = K_{t,r}(a), \widetilde{G}_{t,r}(a^2) = G_{t, r}(a)$. Then the following identity holds
	\begin{equation}\label{exp_scin_kuh}
		\exp(-t \laplcomp_{\Psi}^{F} ) 
		 = 
		\widetilde{G}_{t,r}(\laplcomp_{\Psi}^{F}) + 
		\widetilde{K}_{t,r}(\laplcomp_{\Psi}^{F}).
	\end{equation}
	\par
	By the finite propagation speed of solutions of hyperbolic equations (cf. \cite[Theorems D.2.1, 4.2.8]{MaHol}), the section $\widetilde{G}_{t,r}(\laplcomp_{\Psi}^{F}) \big(y, \cdot \big)$, $y \in M$, depends only on the restriction of $\laplcomp_{\Psi}^{F}$ onto the ball $B_{\Psi}(y, r)$ of radius $r$ around $y$. Moreover, we have
	\begin{equation}\label{eq_guh_zero}
		{\rm{supp}}\, \widetilde{G}_{t,r}(\laplcomp_{\Psi}^{F}) \big(y, \cdot) \subset B_{\Psi}(y, r).
	\end{equation}
	From (\ref{exp_scin_kuh}) and (\ref{eq_guh_zero}), we get
	\begin{equation}\label{eq_exp_kuh_id}
		\exp(-t \laplcomp_{\Psi}^{F})(y, z) = \widetilde{K}_{t,r}(\laplcomp_{\Psi}^{F})(y, z) \quad \text{if} \quad \dist(y, z) > r.
	\end{equation}
	From (\ref{def_kth_lth}), for any $r_0 > 0$ fixed, there exists $c' > 0$ such that for any $m \in \nat$, there is $C>0$ such that for any $t \in ]0, 1], r > r_0, a\in \real$, the following inequality holds (cf. \cite[(4.2.12)]{MaHol})
	\begin{equation}\label{est_kuh}
		|a|^m | K_{t, r}(a) |  \leq C \exp (- c' r^2/t ).
	\end{equation}
	Thus, by (\ref{est_kuh}), for $t \in ]0,1], r > r_0, a \in \real_+$, we have
	\begin{equation}\label{est_tilde_kuh}
		|a|^m | \widetilde{K}_{t, r}(a) |  \leq C \exp ( -c' r^2/ t ).
	\end{equation}
	Now, by (\ref{est_tilde_kuh}), there exists $c' > 0$ such that for any $k, k' \in \nat$, there is $C > 0$ such that for any $t \in ]0,1]$ and $r > r_0$, we have
	\begin{equation} \label{eqn_kth_norm}
		\norm{(\laplcomp_{\Psi}^{F})^k \widetilde{K}_{t, r}( \laplcomp_{\Psi}^{F} )(\laplcomp_{\Psi}^{F})^{k'} }^{0}_{L^2(\Psi)} 
		\leq
		C \exp ( - c' r^2/ t),
	\end{equation}
	where $\norm{\cdot}^{0}_{L^2(\Psi)} $ is the operator norm between the corresponding $L^2$-spaces. 
	By interior elliptic estimates, applied in a $\epsilon$-neighborhood of $(x, y) \in \Psi^2$, we deduce that for any $l \in \nat$, for some $C' > 0$, we have
	\begin{equation} \label{eqn_kth_norm_sup}
	\Big|  \nabla^l \widetilde{K}_{t, r}( \laplcomp_{\Psi}^{F} ) (x, y) \Big|
		\leq
		C' \exp ( - c' r^2/ t).
	\end{equation}
	We get (\ref{eq_hk_exp_decay}) from (\ref{eq_exp_kuh_id}) and (\ref{eqn_kth_norm_sup}) by taking $r = \epsilon$.
	\par 
	Now, by using (\ref{eq_hk_exp_decay}), we compare the small-time expansions of the heat kernels on $\Psi$ and on some model manifolds.
	To do so, we prove Duhamel's formula, which, to simplify the presentation, we formulate in a vicinity of a conical point.
	\par 
	We fix $P \in {\rm{Con}}(\Psi)$. We denote $\alpha = \angle(P)$ and consider the infinite cone $C_{\alpha}$, (\ref{eq_v_theta}), with the induced metric (\ref{eq_conical_metric}).
	We denote by $\laplcomp_{C_{\alpha}}$ the Friedrichs extension of the Riemannian Laplacian on $C_{\alpha}$.
	Let $\epsilon > 0$ be such that $B_{\Psi}(\epsilon, P)$ is isometric to  $C_{\alpha, \epsilon} := B_{C_{\epsilon}}(\epsilon, 0)$. 
	From now on, we identify those neighborhoods implicitly.
	\par 
	For $x, y \in C_{\alpha}$ and $t > 0$, we define
	\begin{equation}\label{eq_fg_defn}
	\begin{aligned}
		&F(x, y, t) := \exp(- t \laplcomp_{\Psi}^{F})(x, y) - \exp(- t \laplcomp_{C_{\alpha}})(x, y), 
		\\
		&G(x, y, s) := \int_{C_{\alpha, \epsilon}} \exp(- (t-s) \laplcomp_{\Psi}^{F})(x, z) \cdot \exp(- s \laplcomp_{C_{\alpha}})(z, y) dv_{C_{\epsilon}}(z).
	\end{aligned}
	\end{equation}
	Then by the definition of the heat kernel, we have
	\begin{equation}\label{eq_g_asymp}
	\begin{aligned}
		& \lim_{s \to t-} G(x, y, s)  = \exp(- t \laplcomp_{C_{\alpha}})(x, y), 
		\\
		& \lim_{s \to 0+} G(x, y, s)  = \exp(- t \laplcomp_{\Psi}^{F})(x, y).
	\end{aligned}
	\end{equation}
	From (\ref{eq_g_asymp}), we deduce that 
	\begin{equation}\label{eq_hk_diff}
		\exp(- t \laplcomp_{C_{\alpha}})(x, y) -  \exp(- t \laplcomp_{\Psi}^{F})(x, y) = \int_{0}^{t} \frac{d G(x, y, s)}{ds}  ds.
	\end{equation}
	However, by the definition of the heat kernel, we have
	\begin{multline}\label{eq_gder}
		\frac{d G(x, y, s)}{ds} = \int_{C_{\epsilon}} \Big( \laplcomp_{\Psi, x} \exp(- (t-s) \laplcomp_{\Psi}^{F})(x, z) \Big) \cdot \exp(- s \laplcomp_{C_{\alpha, \alpha}})(z, y) dv_{C_{\alpha}}(z) 
		\\
		-
		\int_{C_{\alpha, \epsilon}} \exp(- (t-s) \laplcomp_{\Psi}^{F})(x, z) \cdot \Big(  \laplcomp_{C_{\alpha}, z} \exp(- s \laplcomp_{C_{\alpha}})(z, y) \Big) dv_{C_{\alpha}}(z),
	\end{multline}
	where by $\laplcomp_{\Psi, x}$ and $\laplcomp_{C_{\alpha}, z}$ we mean the Laplace operators acting on variables $x$ and $z$ respectively.
	By the symmetry of the heat kernel, we have
	\begin{equation}\label{eq_hk_sym}
		\laplcomp_{\Psi, x} \exp(- (t-s) \laplcomp_{\Psi}^{F})(x, z)  = \laplcomp_{\Psi, z} \exp(- (t-s) \laplcomp_{\Psi}^{F})(x, z).
	\end{equation}
	Now, since both operators $\laplcomp_{\Psi}^{F}$, $\laplcomp_{C_{\alpha}}$ come from Friedrichs extension of the Riemannian Laplacian, for $x$ and $y$ fixed, the functions $\exp(- (t-s) \laplcomp_{\Psi}^{F})(x, \cdot)$ and $\exp(- s \laplcomp_{C_{\alpha}})(\cdot, y)$ are in the domain of Friedrichs extension.
	By Proposition \ref{prop_green_identity}, applied to $C_{\alpha, \epsilon}$, and (\ref{eq_hk_diff}), (\ref{eq_gder}), (\ref{eq_hk_sym}), we deduce
	\begin{multline}\label{eq_hk_diff_duhamel}
		\exp(- t \laplcomp_{C_{\alpha}})(x, y) -  \exp(- t \laplcomp_{\Psi}^{F})(x, y) 
		\\
		= 
		\int_{0}^{t} 
		\int_{\partial C_{\alpha, \epsilon}}
		 \exp(- (t-s) \laplcomp_{\Psi}^{F})(x, z)  \cdot
		\Big( \frac{\partial}{\partial n_z} \exp(- s \laplcomp_{C_{\alpha}})(z, y)\Big) 
		dv_{\partial C_{\alpha}}(z)
		ds
		\\
		-
		\int_{0}^{t} 
		\int_{\partial C_{\alpha, \epsilon}}
		\Big( \frac{\partial}{\partial n_z} \exp(- (t-s) \laplcomp_{\Psi}^{F})(x, z) \Big) \cdot \exp(- s \laplcomp_{C_{\alpha}})(z, y)
		dv_{\partial C_{\alpha}}(z)
		ds
		.
	\end{multline}
	The formulas of type (\ref{eq_hk_diff_duhamel}) are also known as Duhamel's formula (cf. \cite[Theorem 2.48]{BGV}).
	\par 
	From (\ref{eq_hk_exp_decay}) and (\ref{eq_hk_diff_duhamel}), we deduce that there are $c, C > 0$ such that for any $x, y \in C_{\alpha, \epsilon / 2}$, $0 < t < 1$, we have
	\begin{equation}\label{eq_rel_hk_compar_cone}
		\Big| \exp(- t \laplcomp_{C_{\alpha}})(x, y) -  \exp(- t \laplcomp_{\Psi}^{F})(x, y) \Big| \leq C \exp\Big( - \frac{c}{t} \Big).
	\end{equation}
	\par 
	Similarly, for $Q \in {\rm{Ang}}(\Psi)$, we denote $\beta = \angle(Q)$ and consider the infinite angle $A_{\beta}$ with the induced metric (\ref{eq_conical_metric}). We denote by $\laplcomp_{A_{\beta}}$ the Friedrichs extension of the Riemannian Laplacian with von Neumann boundary conditions on $\partial A_{\beta}$. We fix $\epsilon > 0$ in such a way that $B_{\Psi}(\epsilon, Q)$ is isometric to $B_{A_{\beta}}(\epsilon, 0)$.
	Similarly to (\ref{eq_rel_hk_compar_cone}), we deduce that there are $c, C > 0$ such that for any $x, y \in B_{\Psi}(\epsilon/2, Q)$ and $0 < t < 1$, we have
	\begin{equation}\label{eq_rel_hk_compar_angle}
		\Big| \exp(- t \laplcomp_{A_{\beta}})(x, y) -  \exp(- t \laplcomp_{\Psi}^{F})(x, y) \Big| \leq C \exp\Big( - \frac{c}{t} \Big),
	\end{equation}
	where we implicitly identified $x, y \in \Psi$ with corresponding points in $A_{\beta}$.
	\par 
	Now, let $R \in \partial \Psi$ satisfy $\dist_{\Psi}(R, {\rm{Con}}(\Psi) \cup {\rm{Ang}}(\Psi)) > \epsilon$.
	We consider a half plane $\hh = \{(x, y) \in \real^2 : y \geq 0 \}$ and identify $0 \in \hh$ with $R$. Then $B_{\Psi}(\epsilon, R)$ is isometric to $B_{\hh}(\epsilon, 0)$.
	We denote by $\laplcomp_{\hh}$ the self-adjoint extension of the standard Laplacian with von Neumann boundary conditions on $\partial \hh$.
	Similarly to (\ref{eq_rel_hk_compar_cone}), we deduce that there are $c, C > 0$ such that for any $x, y \in B_{\Psi}(\epsilon, R)$ and $0 < t < 1$, we have
	\begin{equation}\label{eq_rel_hk_compar_half}
		\Big| \exp(- t \laplcomp_{\hh})(x, y) -  \exp(- t \laplcomp_{\Psi}^{F})(x, y) \Big| \leq C \exp\Big( - \frac{c}{t} \Big).
	\end{equation}
	\par 
	Finally, let $R \in \Psi$ satisfy $\dist_{\Psi}(R, {\rm{Con}}(\Psi) \cup \partial \Psi) > \epsilon$.
	We consider the real plane $\real^2$ and identify $0 \in \real^2$ with $R$. 
	Then $B_{\Psi}(\epsilon, R)$ is isometric to $B_{\real^2}(\epsilon, 0)$.
	We denote by $\laplcomp_{\real^2}$ the self-adjoint extension of the standard Laplacian.
	Similarly to (\ref{eq_rel_hk_compar_cone}), we deduce that there are $c, C > 0$ such that for any $x, y \in B_{\Psi}(\epsilon, R)$ and $0 < t < 1$, we have
	\begin{equation}\label{eq_rel_hk_compar_plane}
		\Big| \exp(- t \laplcomp_{\real^2})(x, y) -  \exp(- t \laplcomp_{\Psi}^{F})(x, y) \Big| \leq C \exp\Big( - \frac{c}{t} \Big).
	\end{equation}
	\begin{sloppypar}
	To conclude, the estimations (\ref{eq_rel_hk_compar_cone}) - (\ref{eq_rel_hk_compar_plane}) show that to study the asymptotic expansion of $\exp(- t \laplcomp_{\Psi}^{F})(x, x)$, as $t \to 0$, it is enough to consider the analogical problem for a number of model cases: $\real^2, \hh, C_{\alpha} \text{ for } \alpha > 0$ and $A_{\beta}$ for $\beta > 0$.
	\end{sloppypar}
	\par We will start with the simplest one, which is $\real^2$. 
	Here, we have the following identity
	\begin{equation}\label{eq_hk_plane}
		\exp(- t \laplcomp_{\real^2})(x, y) = \frac{1}{4 \pi t} \exp \Big(- \frac{|x - y|^2}{4t} \Big).
	\end{equation}
	By (\ref{eq_rel_hk_compar_plane}) and (\ref{eq_hk_plane}), we deduce that there are $c, C > 0$ such that for any $x \in \Psi$ satisfying $\dist_{\Psi}(x, {\rm{Con}}(\Psi) \cup \partial \Psi) > \epsilon$, and any $0 < t < 1$, we have
	\begin{equation}\label{eq_hk_exp_plane}
		\Big| \exp(- t \laplcomp_{\Psi}^{F})(x, x) - \frac{1}{4 \pi t} \Big|  \leq C \exp(-c/t).
	\end{equation}
	\par 
	Now, on the half-plane $\hh$ and $z_1, z_2 \in \hh$, we clearly have
	\begin{equation}\label{eq_hk_hpln}
		\exp(- t \laplcomp_{\hh})(z_1, z_2) = \exp(- t \laplcomp_{\real^2})(z_1, z_2) + \exp(- t \laplcomp_{\real^2})(z_1, \overline{z_2}).
	\end{equation}
	From (\ref{eq_hk_plane}) and (\ref{eq_hk_hpln}), after an easy calculation, we see that for any $k \in \nat$ and for any $f \in \ccal^{\infty}_{0}(\hh)$, there is $C > 0$ such that for $0 < t < 1$, we have
	\begin{multline}\label{eq_hk_exp_half_full}
		\Big|
		\int_{\hh} f(x) \exp(- t \laplcomp_{\hh})(x, x) dv_{\hh}(x) 
		- \frac{1 }{4 \pi t} \int_{\hh} f(x) dv_{\hh}(x)
		\\
		- \frac{1}{4 \pi \sqrt{t}} \sum_{l = 0}^{2 k + 1} \frac{t^l \Gamma(\frac{l + 1}{2})}{l!} \int_{\partial \hh }\frac{\partial^{l} f}{\partial y^{l}}(x, 0) dv_{\partial \hh}(x)
		\Big|
		\leq
		C t^{k}.
	\end{multline}
	\par 
	Let's now study the cone $C_{\alpha}$, $\alpha > 0$.
	Carslaw in \cite{CarslawGreenSom} (cf. also Kokotov \cite[\S 3.1]{KokotCompPolyh}), gave an explicit formula for the heat kernel on $C_{\alpha}$. 
	By using this formula, Kokotov proved in \cite[Proposition 1, Remark 1]{KokotCompPolyh} that for any $\epsilon > 0$, there are $c, C > 0$ such that for any $0 < t < 1$:
	\begin{equation}\label{eq_cone_exp_full}
		\Big|
			\int_{B_{C_{\alpha}}(\epsilon, 0)} \exp(- t \laplcomp_{C_{\alpha}})(x, x) dv_{C_{\alpha}}(x)
			-
			\frac{\alpha \epsilon^2}{8 \pi t}
			-
			\frac{1}{12} \frac{4 \pi^2 - \alpha^2}{2 \pi \alpha}
		\Big|
		\leq
		C
		\exp\Big(-\frac{c}{t} \Big).
	\end{equation}
	\par 
	Let's finally study the angle $A_{\beta}$, $\beta > 0$.
	Recall that in \cite[Theorem 1]{BergSri}, Berg-Srisatkunarajah have obtained the small-time asymptotic expansion of the heat trace associated to the Friedrichs extension of the Laplacian on the angle endowed with Dirichlet boundary conditions.
	In \cite[Proposition 2.1 and (4.2)]{MazzRow}, Mazzeo-Rowlett used this result, along with (\ref{eq_cone_exp_full}), and an interpretation of a cone through gluing of two angles, one with Dirichlet boundary conditions and another with von Neumann boundary conditions, to show that for any and $\epsilon > 0$, there are $c, C > 0$ such that for any $0 < t < 1$, we have
	\begin{equation}\label{eq_angle_exp_full}
		\Big|
			\int_{B_{A_{\beta}}(\epsilon, 0)} \exp(- t \laplcomp_{A_{\beta}})(x, x) dv_{A_{\beta}}(x)
			-
			\frac{\beta \epsilon^2}{8 \pi t}
			-
			\frac{2 \epsilon}{4 \sqrt{\pi} \sqrt{t}}
			-
			\frac{1}{12} \frac{\pi^2 - \beta^2}{2 \pi \beta}
		\Big|
		\leq
		C
		\exp\Big(-\frac{c}{t} \Big).
	\end{equation}
	By (\ref{eq_rel_hk_compar_cone})-(\ref{eq_rel_hk_compar_plane}), (\ref{eq_hk_exp_plane}), (\ref{eq_hk_exp_half_full}), (\ref{eq_cone_exp_full}) and (\ref{eq_angle_exp_full}), we easily conclude.	
	
\subsection{Kronecker limit formula and the analytic torsion}\label{sect_kron}
	The goal of this section is to recall some explicit formulas for the analytic torsion in case of rectangle and tori.
	The material of this section is very classical, but we include it here only to make this paper self-contained.
	\par We fix a torus $\mathbb{T}_{a,b}$ with perpendicular axes and periods $a$, $b$. In other words
	\begin{equation}
		\mathbb{T}_{a,b} := \real^2 / (a \integ \times b \integ).
	\end{equation}
	The standard Laplacian $\Delta_{\real^2} := - \frac{\partial^2}{\partial x^2} - \frac{\partial^2}{\partial y^2}$  descends to Laplacian $\Delta_{\mathbb{T}_{a,b}}$, acting on the functions on $\mathbb{T}_{a,b}$.
	Classically, the eigenvalues of  $\Delta_{\mathbb{T}_{a,b}}$ are given by
	\begin{equation}
		\spec \big( \Delta_{\mathbb{T}_{a,b}} \big) 
		=
		\Big \{
			(2 \pi)^2 \Big(
				\Big( \frac{n}{a} \Big)^2 + \Big( \frac{m}{b} \Big)^2
			\Big)
			:
			n, m \in \integ
		\Big \}.
	\end{equation}
	So the associated spectral zeta function $\zeta_{\mathbb{T}_{a,b}}(s)$ is given by
	\begin{equation}
		\zeta_{\mathbb{T}_{a,b}}(s) 
		= 
		(2 \pi)^{-2s}
		\sum_{(m, n) \neq (0, 0)}
		\frac{1}{((n/a)^2 + (m/b)^2)^s}.
	\end{equation}
	We rewrite it in the following form
	\begin{equation}\label{eq_zeta_tor_eis}
		\zeta_{\mathbb{T}_{a,b}}(s) 
		= 
		(ab)^s E(z, s),
	\end{equation}
	where $z := i \frac{a}{b}$, and 
	\begin{equation}
		E(z, s) = (2 \pi)^{-s} \sum_{(m, n) \neq (0, 0)}
		\frac{\Im(z)^s}{|nz + m|^{2s}}.
	\end{equation}
	The function $E(z, s)$ is the Eisenstein series multiplied by $(2 \pi)^{-s}$ and thus it admits a meromorphic continuation to $\comp$. 
	By Kronecker limit formula, it has the following expansion near $s = 0$:
	\begin{equation}\label{eq_kron_lim_f}
		E(z, s) = - 1 - s \log \Big(  \frac{a}{b}  \cdot \nu \big( e^{- \frac{2 \pi a}{b}} \big)^4 \Big) + O(s),
	\end{equation}
	where $\nu : D(1) \to \comp$ is the Dedekind eta-function, defined by 
	\begin{equation}
		\nu(q) = q^{1/24} \prod_{n \geq 1} (1 - q^n).
	\end{equation}
	By (\ref{eq_zeta_tor_eis}) and (\ref{eq_kron_lim_f}), we obtain
	\begin{equation}\label{eq_an_tor_torus}
		\log \det{}' \Delta_{\mathbb{T}_{a,b}}
		=
		-\zeta_{\mathbb{T}_{a,b}}'(0)  
		=
		\log(ab)
		+
		\log \Big(  \frac{a}{b}  \cdot \nu \big( e^{- \frac{2 \pi a}{b}} \big)^4 \Big).
	\end{equation}
	Remark that (\ref{eq_an_tor_torus}) goes in line with Duplantier-David \cite[(3.18)]{DuplDav} and Corollary \ref{cor_tor_rect_conv}.
	\par 
	Now, consider a square $[0, a] \times [0, b] \subset \comp$. Classically, the standard Laplacian $\Delta_{[0, a] \times [0, b]}$ endowed with von Neumann boundary conditions is an essentially self-adjoint operator.
	The eigenvalues of  $\Delta_{[0, a] \times [0, b]}$ are given by
	\begin{equation}
		\spec \big( \Delta_{[0, a] \times [0, b]} \big) 
		=
		\Big \{
			\pi^2 \Big(
				\Big( \frac{n}{a} \Big)^2 + \Big( \frac{m}{b} \Big)^2
			\Big)
			:
			n, m \in \nat
		\Big \}.
	\end{equation}
	So the associated spectral zeta function $\zeta_{[0, a] \times [0, b]}(s)$ is related to the zeta function of torus $\mathbb{T}_{a, b}$ by 
	\begin{equation}
		\zeta_{[0, a] \times [0, b]}(s) 
		= 
		4^{s - 1} \zeta_{\mathbb{T}_{a, b}}(s) 
		+
		\frac{1}{2} \Big( 
			\Big( \frac{a}{\pi} \Big)^s + \Big(  \frac{b}{\pi} \Big)^s
		\Big)
		\zeta(2s),
	\end{equation}
	where $\zeta(s)$ is the Riemann zeta function.
	By using identities
	\begin{equation}
		\zeta(0) = -\frac{1}{2}, \qquad \zeta'(0) = - \frac{\log(2\pi)}{2},
	\end{equation}
	and (\ref{eq_kron_lim_f}), (\ref{eq_an_tor_torus}), we get the following expression
	\begin{equation}\label{eq_an_tors_rect}
		\log \det{}' \Delta_{[0, a] \times [0, b]}
		=
		-\zeta_{[0, a] \times [0, b]}'(0)  
		=
		\frac{3}{4} \log(ab) 
		+ 
		\frac{1}{4} \log(y |\nu(z)|^4)
		+
		\frac{3}{2} \log(2).
	\end{equation}
	This goes in line with Duplantier-David in \cite[(4.7) and (4.23)]{DuplDav} and Theorem \ref{thm_asympt_exansion}.

		\bibliographystyle{abbrv}

\Addresses

\end{document}